\documentclass[12pt,english]{amsart}
\usepackage[T1]{fontenc}
\usepackage[utf8]{inputenc}
\usepackage{geometry}
\usepackage{color}
\usepackage{babel}
\usepackage{mathtools}
\usepackage{bm}
\usepackage{amstext}
\usepackage{amsthm}
\usepackage{amssymb}
\usepackage{esint}
\PassOptionsToPackage{normalem}{ulem}
\usepackage{ulem}
\usepackage[unicode=true,pdfusetitle,
 bookmarks=true,bookmarksnumbered=false,bookmarksopen=false,
 breaklinks=false,pdfborder={0 0 1},backref=false,colorlinks=false]
 {hyperref}

\makeatletter
\numberwithin{equation}{section}
\theoremstyle{plain}
\newtheorem{thm}{\protect\theoremname}[section]
\theoremstyle{remark}
\newtheorem{rem}[thm]{\protect\remarkname}
\theoremstyle{plain}
\newtheorem{lem}[thm]{\protect\lemmaname}
\theoremstyle{plain}
\newtheorem{prop}[thm]{\protect\propositionname}

\makeatother

\providecommand{\lemmaname}{Lemma}
\providecommand{\propositionname}{Proposition}
\providecommand{\remarkname}{Remark}
\providecommand{\theoremname}{Theorem}

\begin{document}
\global\long\def\bbC{\mathbb{C}}%
\global\long\def\bbN{\mathbb{N}}%
\global\long\def\bbQ{\mathbb{Q}}%
\global\long\def\bbR{\mathbb{R}}%
\global\long\def\bbS{\mathbb{S}}%
\global\long\def\bbZ{\mathbb{Z}}%

\global\long\def\calA{\mathcal{A}}%
\global\long\def\calB{\mathcal{B}}%
\global\long\def\calC{\mathcal{C}}%
\global\long\def\calD{\mathcal{D}}%
\global\long\def\calE{\mathcal{E}}%
\global\long\def\calF{\mathcal{F}}%
\global\long\def\calH{\mathcal{H}}%
\global\long\def\calI{\mathcal{I}}%
\global\long\def\calK{\mathcal{K}}%
\global\long\def\calL{\mathcal{L}}%
\global\long\def\calM{\mathcal{M}}%
\global\long\def\calN{\mathcal{N}}%
\global\long\def\calU{\mathcal{U}}%
\global\long\def\calO{\mathcal{O}}%
\global\long\def\calP{\mathcal{P}}%
\global\long\def\calQ{\mathcal{Q}}%
\global\long\def\calR{\mathcal{R}}%
\global\long\def\calS{\mathcal{S}}%
\global\long\def\calT{\mathcal{T}}%
\global\long\def\calU{\mathcal{U}}%
\global\long\def\calV{\mathcal{V}}%
\global\long\def\calW{\mathcal{W}}%
\global\long\def\calZ{\mathcal{Z}}%
\global\long\def\calX{\mathcal{X}}%
\global\long\def\calY{\mathcal{Y}}%

\global\long\def\alp{\alpha}%
\global\long\def\dlt{\delta}%
\global\long\def\Dlt{\Delta}%
\global\long\def\eps{\varepsilon}%
\global\long\def\gmm{\gamma}%
\global\long\def\Gmm{\Gamma}%
\global\long\def\tht{\theta}%
\global\long\def\kpp{\kappa}%

\global\long\def\kap{\kappa}%
\global\long\def\lda{\lambda}%
\global\long\def\lmb{\lambda}%
\global\long\def\Lda{\Lambda}%
\global\long\def\Lmb{\Lambda}%
\global\long\def\vphi{\varphi}%
\global\long\def\omg{\omega}%

\global\long\def\rnd{\partial}%
\global\long\def\rd{\partial}%
\global\long\def\aleq{\lesssim}%
\global\long\def\ageq{\gtrsim}%
\global\long\def\aeq{\simeq}%

\global\long\def\nla{\nabla}%

\global\long\def\peq{\mathrel{\phantom{=}}}%
\global\long\def\To{\longrightarrow}%
\global\long\def\weakto{\rightharpoonup}%
\global\long\def\embed{\hookrightarrow}%
\global\long\def\Re{\mathrm{Re}}%
\global\long\def\Im{\mathrm{Im}}%
\global\long\def\chf{\mathbf{1}}%
\global\long\def\lan{\langle}%
\global\long\def\ran{\rangle}%
\global\long\def\td#1{\widetilde{#1}}%
\global\long\def\br#1{\overline{#1}}%
\global\long\def\ubr#1{\underline{#1}}%
\global\long\def\ul#1{\underline{#1}}%

\global\long\def\wh#1{\widehat{#1}}%
\global\long\def\tint#1#2{{\textstyle \int_{#1}^{#2}}}%
\global\long\def\tsum#1#2{{\textstyle \sum_{#1}^{#2}}}%

\global\long\def\loc{\mathrm{loc}}%
\global\long\def\NL{\mathrm{NL}}%
\global\long\def\coer{\mathrm{c}}%
\global\long\def\rmD{\mathrm{D}}%
\global\long\def\bfp{\mathbf{p}}%
\global\long\def\bfi{{\bf i}}%
\global\long\def\ud{\mathrm{d}}%
\global\long\def\Blan{\Big\langle}%
\global\long\def\Bran{\Big\rangle}%
\global\long\def\Mor{\mathrm{Mor}}%
\global\long\def\ex{\mathrm{ex}}%

\global\long\def\red#1{\textcolor{red}{#1}}%
\global\long\def\blue#1{\textcolor{blue}{#1}}%
\global\long\def\green#1{\textcolor{green}{#1}}%

\global\long\def\define{\coloneqq}%
\global\long\def\setJ{\llbracket J\rrbracket}%

\title[bubble tower solutions to wave maps]{Construction of infinite time bubble tower solutions to critical
wave maps equation}
\author{Seunghwan Hwang}
\address{Department of Mathematical Sciences, Seoul National University, 1
Gwanak-ro, Gwanak-gu, Seoul 08826, Republic of Korea}
\email{seunghwan.hwang777@snu.ac.kr}
\author{Kihyun Kim}
\address{Department of Mathematical Sciences and Research Institute of Mathematics,
Seoul National University, 1 Gwanak-ro, Gwanak-gu, Seoul 08826, Republic
of Korea}
\email{kihyun.kim@snu.ac.kr}
\thanks{S. Hwang was partially supported by the NRF grant funded by the Korea
government (MSIT) RS-2025-00523523. K. Kim was supported by the New
Faculty Startup Fund from Seoul National University, the POSCO Science
Fellowship of POSCO TJ Park Foundation, and the National Research
Foundation of Korea (NRF) grant funded by the Korea government (MSIT)
RS-2025-00523523.}
\begin{abstract}
We construct infinite time bubble tower solutions to the critical
wave maps equation taking values in the two-sphere. More precisely,
for any integers $k\geq3$ and $J\geq1$, we construct a solution
that is global in one time direction, has $k$-corotational symmetry,
and asymptotically decomposes into $J$-many concentric bubbles of
alternating signs with asymptotically vanishing radiation. The scales
of each bubble are of order $t^{-\alpha_{j}}$ with $\alpha_{j}=(\frac{k}{k-2})^{j-1}-1$.
This shows the existence of multi-bubble solutions with an arbitrary
number of bubbles in soliton resolution, provided that $k\geq3$,
global existence in one time direction, and alternating signs are
considered. Our proof is based on modulation analysis with the method
of backward construction. The key new ingredient is a Morawetz-type
functional that provides suitable monotonicity estimates for solutions
around \emph{multi-bubble }configurations.
\end{abstract}

\maketitle
\tableofcontents{}

\section{Introduction}

In this paper, we consider the \emph{wave maps equation from $\bbR^{1+2}\to\bbS^{2}$}:
\begin{equation}
\rd_{tt}\phi=\Dlt\phi+(-|\rd_{t}\phi|^{2}+|\nabla\phi|^{2})\phi,\label{eq:WM}
\end{equation}
where $\phi=\phi(t,x)$ takes values in the standard unit sphere $\bbS^{2}\subset\bbR^{3}$.
Equation (\ref{eq:WM}) conserves \emph{energy} 
\begin{equation}
E[(\phi,\rd_{t}\phi)]\coloneqq\frac{1}{2}\int_{\bbR^{2}}|\nabla_{t,x}\phi|^{2}dx\label{eq:energy}
\end{equation}
and has \emph{scaling symmetry}: if $\phi(t,x)$ is a solution to
(\ref{eq:WM}), then so is 
\[
(t,x)\mapsto\phi(\lmb^{-1}t,\lmb^{-1}x),\qquad\forall\lmb\in(0,\infty).
\]
Since the energy remains invariant under this scaling transform, (\ref{eq:WM})
is called \emph{energy-critical}. The local-in-time Cauchy theory
for (\ref{eq:WM}) is now well-understood and it is well-posed in
the \emph{energy space}; see \cite{KlainermanMachedon1993CPAM,KlainermanMachedon1995Duke,KlainermanSelberg1997CommPDE,Tao2001CMP,Krieger2004CMP,Tataru2005AJM}
as well as a book \cite{GebaGrillakisWMbook}.

The goal of this paper is to construct so-called \emph{bubble tower
solutions} to (\ref{eq:WM}). As the construction is concerned and
these solutions will be concentric at the spatial origin, we impose
a symmetry condition to describe them more clearly. 

We impose\emph{ $k$-corotational symmetry}; $\phi(t,x)$ takes the
form 
\begin{equation}
\phi(t,x)=(\sin u(t,r)\cos k\tht,\sin u(t,r)\sin k\tht,\cos u(t,r))\in\bbS^{2},\qquad r=|x|,\label{eq:corotational-ansatz}
\end{equation}
where $k$ is a positive integer and $u=u(t,r)\in\bbR$. This symmetry
is preserved under the evolution of (\ref{eq:WM}) and the evolution
equation reduces to 
\begin{equation}
\rd_{tt}u=\rd_{rr}u+\frac{1}{r}\rd_{r}u-k^{2}\frac{\sin(2u)}{2r^{2}}.\tag{\ensuremath{k}-WM}\label{eq:k-equiv-WM}
\end{equation}
We often view (\ref{eq:k-equiv-WM}) in vectorial form 
\begin{equation}
\left|\begin{aligned}\rd_{t}u & =\dot{u},\\
\rd_{t}\dot{u} & =\rd_{rr}u+\tfrac{1}{r}\rd_{r}u-k^{2}\tfrac{\sin(2u)}{2r^{2}},
\end{aligned}
\right.\label{eq:k-equiv-WM-vectorial}
\end{equation}
and denote $\bm{u}=(u,\dot{u})$. 

Within $k$-corotational symmetry, the conserved energy (\ref{eq:energy})
reads 
\begin{equation}
E[(u,\rd_{t}u)]\coloneqq2\pi\int_{0}^{\infty}\frac{1}{2}\Big((\rd_{t}u)^{2}+(\rd_{r}u)^{2}+k^{2}\frac{\sin^{2}u}{r^{2}}\Big)rdr.\label{eq:def-k-equiv-energy}
\end{equation}
The energy space is the disjoint union of its connected components
$\calE_{\ell,m}$ for $\ell,m\in\bbZ$, where 
\[
\calE_{\ell,m}\coloneqq\{(u_{0},\dot{u}_{0}):(0,\infty)\to\bbR^{2}\mid E[(u_{0},\dot{u}_{0})]<\infty,\ \lim_{r\to0}u_{0}(r)=\ell\pi,\ \lim_{r\to\infty}u_{0}(r)=m\pi\}.
\]
The dynamics of (\ref{eq:k-equiv-WM}) can then be considered separately
on each component $\calE_{\ell,m}$. Each $\calE_{\ell,m}$ is an
affine space and $\calE\coloneqq\calE_{0,0}$ becomes a vector space.
We equip $\calE$ with the norm (we equip all $L^{p}$ spaces with
$rdr$-measure)
\[
\|(u_{0},\dot{u}_{0})\|_{\dot{\calH}^{1}}^{2}\coloneqq\|u_{0}\|_{\dot{H}_{k}^{1}}^{2}+\|\dot{u}_{0}\|_{L^{2}}^{2}\coloneqq\int_{0}^{\infty}\Big((\rd_{r}u_{0})^{2}+\frac{u_{0}^{2}}{r^{2}}+\dot{u}_{0}^{2}\Big)rdr
\]
and equip $\calE_{\ell,m}$ with the metric induced by this norm.
(We use the subscript $k$ in $\dot{H}_{k}^{1}$ to indicate that
our $u_{0}$ comes from a $k$-corotational map $\phi_{0}$. We have
$r^{-1}u_{0}\in L^{2}$ due to $k\neq0$.)

In the description of asymptotic behavior of finite energy solutions
to (\ref{eq:k-equiv-WM}) (or, more generally to (\ref{eq:WM})),
a crucial role is played by static profiles, called \emph{harmonic
maps} (we often call them \emph{solitons }or\emph{ bubbles} instead)\emph{.}
Within $k$-corotational symmetry, they are completely classified
and must be equal to 
\[
\iota Q\Big(\frac{\cdot}{\lmb}\Big)+\ell\pi
\]
for some $\iota\in\{\pm1\}$, $\lmb\in(0,\infty)$, and $\ell\in\bbZ$,
where 
\begin{equation}
Q(r)\coloneqq2\arctan(r^{k}).\label{eq:def-Q}
\end{equation}
With earlier developments such as global regularity results \cite{ChristodoulouTahvildarZadeh1993CPAM,ShatahTahvildarZadeh1992CPAM,ShatahTahvildarZadeh1994CPAM}
and the exclusion of self-similar blow-up \cite{ChristodoulouTahvildarZadeh1993Duke},
the emergence of harmonic maps in the asymptotic behavior (of finite-time
blow-up) is more directly indicated by the bubbling result of Struwe
\cite{Struwe2003CPAM}, at least within symmetry. 

Soliton resolution (or bubble decomposition) for (\ref{eq:k-equiv-WM}),
which is now a theorem by Jendrej and Lawrie \cite{JendrejLawrie2025JAMS}
(and by \cite{DuyckaertsKenigMartelMerle2022CMP} when $k=1$), asserts
that every finite energy solution $u(t)$ asymptotically decomposes
into the sum of scale-decoupled harmonic maps and a radiation
\begin{equation}
u(t,r)=\ell\pi+\sum_{j=1}^{J}\iota_{j}Q\Big(\frac{r}{\lmb_{j}(t)}\Big)+\text{(radiation)},\label{eq:rough-soliton-resolution}
\end{equation}
for some $\ell\in\bbZ$, $J\in\bbN$, signs $\iota_{1},\dots,\iota_{J}\in\{\pm1\}$,
and time-dependent scales $\lmb_{1}(t),\dots,\lmb_{J}(t)$ satisfying
$\lmb_{1}(t)\gg\dots\gg\lmb_{J}(t)$ as $t$ approaches the maximal
time of existence $T_{+}\in(0,+\infty]$. There holds the strict upper
bound by the self-similar scale: $\lmb_{1}(t)\ll T_{+}-t$ if $T_{+}<+\infty$
(finite-time blow-up) and $\lmb_{1}(t)\ll t$ if $T_{+}=+\infty$
(forward-in-time global). Some prior studies on soliton resolution
(\ref{eq:k-equiv-WM}) are \cite{Cote2015CPAM,CoteKenigLawrieSchlag2015AJM1,CoteKenigLawrieSchlag2015AJM2,JiaKenig2017AJM},
which build upon the deep insights developed for the critical wave
equation (NLW) under radial symmetry \cite{DuyckaertsKenigMerle2011JEMS,DuyckaertsKenigMerle2012GAFA,DuyckaertsKenigMerle2013CambJMath}.
See also \cite{DuyckaertsJiaKenigMerle2017GAFA,DuyckaertsJiaKenigMerle2018IMRN,Grinis2017CMP}
for non-radial results. We refer the reader to \cite{DuyckaertsKenigMerle2013CambJMath,DuyckaertsKenigMerle2023Acta,JendrejLawrie2023AnnPDE,DuyckaertsKenigMartelMerle2022CMP,CollotDuyckaertsKenigMerle2024VietJM}
for the continuous-in-time soliton resolution for the radial critical
wave equation.

Although soliton resolution guarantees the asymptotic decomposition
(\ref{eq:rough-soliton-resolution}), it does not provide any further
information on the number of bubbles, the signs $\iota_{1},\dots,\iota_{J}$,
and the behavior of scales $\lmb_{1}(t)\gg\dots\gg\lmb_{J}(t)$ other
than the strict upper bound of $\lmb_{1}(t)$ by the self-similar
scale. The construction of various scenarios is a natural subject
of investigation. 

The first construction of finite-time blow-up for (\ref{eq:k-equiv-WM})
was provided by \cite{KriegerSchlagTataru2008Invent,RodnianskiSterbenz2010Ann.Math.,RaphaelRodnianski2012Publ.Math.}
(with an earlier numerical evidence provided in \cite{BizonChmajTabor2001Nonlinearity})
in the single-bubble case $J=1$. There appear more examples and refinements
in this single-bubble case; see for instance \cite{GaoKrieger2015CPAA,JendrejLawrieRodriguez2022ASENS,Pillai2023MAMS,Pillai2023CMP,Kim2023CMP,Jeong2025AnalPDE}
for (\ref{eq:k-equiv-WM}), \cite{KriegerSchlagTataru2009AdvMath,KriegerSchlagTataru2009Duke,DonningerKrieger2013MathAnn,KriegerSchlag2014JMPA}
for closely related critical wave equations, and \cite{BurzioKrieger2022MAMS,KriegerMiao2020Duke,KriegerMiaoSchlag2020arXiv}
for their stability. The first example of two-bubble $J=2$ solutions
for (\ref{eq:k-equiv-WM}) was provided by Jendrej \cite{Jendrej2019AJM}
for $k\geq3$ with $T_{+}=+\infty$ (the proof works for $k=2$ as
well and see also \cite{JendrejLawrie2018Invent,JendrejLawrie2023CPAM}),
and finite-time blowing up two-bubble was recently constructed for
$k=2$ by Jendrej and Krieger \cite{JendrejKrieger2025arXiv}. 

\subsection{Main result}

The goal of the present work and a very recent\footnote{The work \cite{KriegerPalacios2026arXiv} appeared on arXiv a week
earlier than the present work. The result and method therein are independent
to ours.} work Krieger--Palacios \cite{KriegerPalacios2026arXiv} is to provide
the first construction of bubble tower solutions to (\ref{eq:k-equiv-WM})
with arbitrary number of bubbles $J\in\bbN$. (Here, we call multi-bubble
solutions \emph{bubble tower solutions} due to their concentric nature.)
For any $J\in\bbN$, the work \cite{KriegerPalacios2026arXiv} constructs
finite-time blow-up ($T_{+}<+\infty$) solutions for $k=2$, whereas
our main result constructs infinite-time blow-up ($T_{+}=+\infty$)
solutions for any $k\geq3$. 

In the following, we denote $\|g\|_{\dot{H}_{k}^{2}}^{2}\coloneqq\|\rd_{rr}g\|_{L^{2}}^{2}+\|\tfrac{1}{r}\rd_{r}g\|_{L^{2}}^{2}+\|\tfrac{1}{r^{2}}g\|_{L^{2}}^{2}$
and $\dot{\calH}^{2}\coloneqq\dot{H}_{k}^{2}\times\dot{H}_{k}^{1}$.
\begin{thm}[Construction of bubble towers]
\label{thm:main}Let $k\geq3$ and $J\in\bbN$. Then, there exists
a solution $\bm{u}(t)=(u(t),\rd_{t}u(t))$ to (\ref{eq:k-equiv-WM})
with data in $\calE_{0,J\mathrm{mod}2}\cap\dot{\calH}^{2}$, defined
for all large positive times $t$, such that 
\[
\lim_{t\to+\infty}\Big\{\Big\| u(t)-\sum_{j=1}^{J}(-1)^{j-1}Q\Big(\frac{\cdot}{\gmm_{j}t^{-\alp_{j}}}\Big)\Big\|_{\dot{H}_{k}^{1}}+\|\rd_{t}u(t)\|_{L^{2}}\Big\}=0,
\]
where 
\begin{equation}
\alp_{j}\coloneqq\Big(\frac{k}{k-2}\Big)^{j-1}-1,\qquad\forall j\in\bbN,\label{eq:def-alp_j}
\end{equation}
and $\gmm_{j}$ are the universal constants defined in (\ref{eq:def-gmm_j}).
By applying the time-reversal symmetry $u(t,r)\mapsto u(-t,r)$, one
also obtains backward-in-time bubble tower solutions.
\end{thm}

\begin{rem}
To our knowledge, the present work and Krieger--Palacios \cite{KriegerPalacios2026arXiv}
provide the first construction of bubble tower solutions with an arbitrary
number of bubbles for wave-type equations. For $k=2$ and any $J\in\bbN$,
\cite{KriegerPalacios2026arXiv} constructs a solution with $T_{+}<+\infty$,
$\lmb_{1}(t)=\frac{T_{+}-t}{|\log(T_{+}-t)|^{\beta}}$ (with arbitrary
$\beta>\frac{3}{2}$), and $\lmb_{2}(t)\dots,\lmb_{J}(t)$ described
via exponential towers. On the other hand, for $k\geq3$ and any $J\in\bbN$,
our Theorem~\ref{thm:main} constructs a solution with $T_{+}=+\infty$,
$\lmb_{1}(t)=1$, and $\lmb_{j}(t)=\gmm_{j}t^{-\alp_{j}}$. Thus Theorem~\ref{thm:main}
and \cite{KriegerPalacios2026arXiv} are complementary in terms of
results. The methods are also significantly different.
\end{rem}

\begin{rem}[On strategy, difficulty, and novelty]
\label{rem:strategy-difficulty-novelty}For the proof, we use modulation
analysis with the backward construction method as in the two-bubble
construction of Jendrej \cite{Jendrej2019AJM}, which goes back to
Merle \cite{Merle1990CMP} and Martel \cite{Martel2005AJM}. See also
the strategy in Section~\ref{subsec:Strategy-of-the-proof}.

First, we note that the functional framework of \cite{Jendrej2019AJM},
i.e., working solely in the critical energy space, seems difficult
to be applied directly for arbitrary $J\geq3$. This is because the
justification of formal ODE system (\ref{eq:intro-formal-ODE}) for
the scales $\lmb_{j}(t)$ of bubbles requires stronger estimates for
the $\dot{\calH}^{1}$-energy of $\bm{g}$ in the region $r\aleq\lmb_{j}$
as $j$ increases (recall $\lmb_{1}\gg\dots\gg\lmb_{J}$). 

An important observation to overcome this difficulty is the following;
if we modify the framework of \cite{Jendrej2019AJM} by adding a new
higher Sobolev norm control $\|\bm{g}\|_{\dot{\calH}^{2}}\aeq\|g\|_{\dot{H}_{k}^{2}}+\|\dot{g}\|_{\dot{H}_{k}^{1}}$
into the set of bootstrap assumptions, then the formal ODE system
(\ref{eq:intro-formal-ODE}) can be justified. Indeed, Hardy's inequality
roughly yields $\|\chf_{r\aleq\lmb_{j}}\bm{g}\|_{\dot{\calH}^{1}}\aleq\lmb_{j}\|r^{-1}\bm{g}\|_{\dot{\calH}^{1}}\aleq\lmb_{j}\|\bm{g}\|_{\dot{\calH}^{2}}$,
saying that the more localized the region is, the stronger the estimate
we have. Then, we will show that a global-in-space $\dot{\calH}^{2}$-control
of the form $\|\bm{g}\|_{\dot{\calH}^{2}}\aleq|t|^{-1-}$ is sufficient
to justify the modulation ODE system. The main remaining task is then
to propagate this highest Sobolev bootstrap control of $\bm{g}$.
Let us mention that the work \cite{JendrejLawrie2022AnalPDE} considers
two-bubble ($J=2$) solutions including this $\dot{\calH}^{2}$-framework
to obtain more refined estimates on two-bubble solutions for their
proof of uniqueness \cite{JendrejLawrie2023CPAM}. An $\dot{\calH}^{2}$-framework
was also used earlier in \cite{RodnianskiSterbenz2010Ann.Math.,RaphaelRodnianski2012Publ.Math.}
in the single-bubble context.

Our key new input is the introduction of a \emph{Morawetz-type functional
around multi-bubbles} with arbitrary $J\in\bbN$; see Proposition~\ref{prop:Morawetz}.
As is typical in blow-up dynamics, the highest Sobolev control of
the form $\|\bm{g}\|_{\dot{\calH}^{2}}\aleq|t|^{-1-}$ cannot be closed
by itself. This is overcome by this new ingredient, which provides
additional monotonicity estimates for $\bm{g}$. We believe that the
\emph{multi-bubble} Morawetz-type functional introduced here is of
independent interest and might potentially be used in more general
situations (beyond constructions) involving bubble towers with arbitrary
$J\in\bbN$. Our Morawetz-type functional is ultimately based on a
Morawetz functional around single-bubbles (and on the Bogomol'nyi
trick) of \cite{RodnianskiSterbenz2010Ann.Math.}. 

Finally, we use a simple decomposition of solution (\ref{eq:intro-decomp})
without introducing any profile modifications, which is sufficient
thanks to the Morawetz controls.
\end{rem}

\begin{rem}[Lower corotational index]
Let $k=2$. In view of the formal ODE system (\ref{eq:intro-formal-ODE})
of scales $\lmb_{j}(t)$, an infinite time blow-up of similar type
(asymptotically vanishing radiation) seems to exist. Note that such
a solution with $J=2$ is essentially constructed in \cite{Jendrej2019AJM}.
Our analysis might be applicable to small $J$ such as $J\leq3$,
but generalizing to arbitrary $J\in\bbN$ seems to require further
ideas due to the repeated exponential nature in the nonlinear ODE
system (\ref{eq:intro-formal-ODE}) including error terms. Nevertheless,
by a completely different method, \cite{KriegerPalacios2026arXiv}
constructs finite-time blow-up ($T_{+}<+\infty$) solutions for any
$J\in\bbN$ when $k=2$.

When $k=1$, the formal ODE system (\ref{eq:intro-formal-ODE}), possibly
with modifications, suggests finite-time blow-up whenever $J\geq2$.
This is demonstrated in \cite{Rodriguez2021APDE} for $J=2$. However,
when $k=1$, the soliton-radiation interaction is the strongest and
has to be taken into account. We expect rich possibility of dynamics
when $k=1$.
\end{rem}

\begin{rem}[Related works]
Let us mention some related works. Since there are numerous works
on the construction of multi-solitons and multi-bubbles, we cannot
mention them all. For the construction of bubble tower solutions (within
symmetry), see for instance \cite{Jendrej2017AnalPDE,JendrejLiXu2026arXiv}
for two-bubble solutions to energy-critical Schrödinger equations
and \cite{DaskalopoulosDelPinoSesum2018Crelle,delPinoMussoWei2021AnalPDE,SunWeiZhang2022CVPDE,KimMerle2025CPAM}
for parabolic models. For the non-existence of finite time blowing
up $k$-corotational bubble towers for harmonic map heat flows (the
parabolic counterpart of (\ref{eq:k-equiv-WM})), see \cite{vanderHout2003JDE,Samuelian2025arXiv}.
Let us also mention some non-radial multi-bubble (or multi-soliton)
examples for dispersive equations \cite{Merle1990CMP,Martel2005AJM,CoteZaag2013CPAM,MerleZaag2018CPAM,MartelRaphael2018AnnSci,JendrejMartel2020JMPA,MartelMerle2025arXiv}.
\end{rem}

\subsection{\label{subsec:Notation}Notation}

\subsubsection{Basic notation}

~
\begin{itemize}
\item $\bbN$ is the set of positive integers, $\bbN_{0}\coloneqq\bbN\cup\{0\}$. 
\item For $A\in\bbR$ and $B>0$, we use the standard asymptotic notation
$A\aleq B$ or $A=\calO(B)$ if there is a constant $C>0$ such that
$|A|\leq CB$. $A\aeq B$ means $A\aleq B$ and $B\aleq A$. The dependence
on parameters is written in subscripts. In this paper, any dependency
on the corotational index $k$, the number of bubbles $J$ will be
omitted. 
\item We use the small $o$-notation such as $o_{t\to-\infty}(1)$ and $o_{\alp\to0}(1)$,
which denote quantities going to zero as $t\to-\infty$ and $\alp\to0$,
respectively. In this paper, $o(1)$ denotes a function in time of
size $o_{t\to-\infty}(1)$. 
\item For a set (or a statement) $A$, $\chf_{A}$ denotes the characteristic
function of $A$. $\chi=\chi(r)$ denotes a smooth radial function
such that $0\leq\chi\leq1$, $\chi(r)=1$ for $r\leq1$, and $\chi(r)=0$
for $r\geq2$. 
\item $\lan x\ran\coloneqq(1+|x|^{2})^{\frac{1}{2}}$ for $x\in\bbR$.
\item The integral sign $\int$ without any specification denotes $\int f=\int_{0}^{\infty}f(r)rdr$.
All the $L^{p}$ norms are equipped with the $rdr$-measure unless
otherwise stated. For functions $f,g:(0,\infty)\to\bbR$, their $L^{2}$
inner product is defined by $\lan f,g\ran\define\int fg$.
\item Given $\lda\in(0,\infty)$ and a function $f:(0,\infty)\to\bbR$,
we denote 
\[
f_{\lda}(r)\coloneqq f(y),\quad y\coloneqq\frac{r}{\lmb},\quad\text{and}\quad f_{\ubr{\lda}}\coloneqq\frac{1}{\lda}f_{\lmb}.
\]
$\Lda_{s}\coloneqq r\rnd_{r}+1-s$, where $s\in\bbR$, denotes the
generator of $\dot{H}^{s}$-scaling. $\Lmb\coloneqq\Lmb_{1}$.
\item $|g|_{-k}^{2}\define|\rd_{r}^{k}g|^{2}+|r^{-1}\rd_{r}^{k-1}g|^{2}+\dots+|r^{-k}g|^{2}$
for $g:(0,\infty)\to\bbR$ and $k\in\bbN$.
\item $f(u)\define\frac{\sin2u}{2}$ so that (\ref{eq:k-equiv-WM}) is written
as $\rd_{tt}u=\rd_{rr}u+\frac{1}{r}\rd_{r}u-k^{2}\frac{f(u)}{r^{2}}$.
Denote $H_{U}\coloneqq-\rd_{rr}-\frac{1}{r}\rd_{r}+k^{2}\frac{f'(U)}{r^{2}}$
for a function $U$ on $(0,\infty)$. For $\lmb\in(0,\infty)$, $H_{\lda}\define H_{Q_{\lmb}}$
and $H\coloneqq H_{1}$. Finally, $\NL_{U}(g)\define\frac{k^{2}}{r^{2}}\{f(U+g)-f(U)-f'(U)g\}$
denotes the nonlinear term in $g$.
\item For $\bm{g}=(g,\dot{g})$, we define the norms 
\begin{align*}
\|\bm{g}\|_{\dot{\calH}^{1}}^{2} & \coloneqq\|g\|_{\dot{H}_{k}^{1}}^{2}+\|\dot{g}\|_{L^{2}}^{2}\coloneqq\int|\rd_{r}g|^{2}+|\tfrac{1}{r}g|^{2}+|\dot{g}|^{2},\\
\|\bm{g}\|_{\dot{\calH}^{2}}^{2} & \coloneqq\|g\|_{\dot{H}_{k}^{2}}^{2}+\|\dot{g}\|_{\dot{H}_{k}^{1}}^{2}\coloneqq\int\{|\rd_{rr}g|^{2}+|\tfrac{1}{r}\rd_{r}g|^{2}+|\tfrac{1}{r^{2}}g|^{2}\}+\{|\rd_{r}\dot{g}|^{2}+|\tfrac{1}{r}\dot{g}|^{2}\}.
\end{align*}
We note that, the Euler angle $u_{0}$ of a smooth finite energy $k$-corotational
map $\phi_{0}$ with $u_{0}(0)=0$ belongs to $\dot{H}_{k}^{2}$ due
to $k>1$. 
\item The sequences of universal constants $(\alp_{j})_{j\in\bbN}$ and
$(\gmm_{j})_{j\in\bbN}$ are defined by 
\begin{align}
\alp_{j} & \coloneqq\Big(\frac{k}{k-2}\Big)^{j-1}-1,\tag{\ref{eq:def-alp_j}}\nonumber \\
\gmm_{j} & \coloneqq\Big(\frac{\|\Lmb Q\|_{L^{2}}^{2}}{8k^{2}}\alp_{j}(\alp_{j}+1)\Big)^{\frac{1}{k-2}}\gmm_{j-1}^{\frac{k}{k-2}}\quad\text{for }j\geq2,\quad\text{with }\gmm_{j}\coloneqq1.\label{eq:def-gmm_j}
\end{align}
We denote $\kap\define\|\Lda Q\|_{L^{2}}^{2}$.
\end{itemize}

\subsubsection{Notation for multi-bubbles}

~
\begin{itemize}
\item $J\in\bbN$ denotes the number of bubbles. The sign and scale of each
bubble are typically denoted by $(\iota_{j},\lmb_{j})\in\{\pm\}\times(0,\infty)$.
These can be collectively denoted by $\vec{\iota}\in\{\pm\}^{J}$
and $\vec{\lmb}\in(0,\infty)^{J}$.
\item For $\alp\in(0,1]$, we define 
\[
\calP_{J}(\alp)\coloneqq\{(\vec{\iota},\vec{\lmb})\in\{\pm\}^{J}\times(0,\infty)^{J}:\max_{j\in\{2,\dots,J\}}\frac{\lmb_{j}}{\lmb_{j-1}}<\alp\}.
\]
If $J=1$, then $\calP_{J}(\alp)=\{\pm\}\times(0,\infty)$. We also
use a slight abuse of notation $\vec{\lmb}\in\calP_{J}(\alp)$ to
mean that $\max_{j\in\{2,\dots,J\}}\lmb_{j}/\lmb_{j-1}<\alp$.
\item For $(\vec{\iota},\vec{\lmb})\in\calP_{J}(1)$ and a function $\phi$
on $(0,\infty)$, we denote 
\[
\phi_{;j}(r)\coloneqq\iota_{j}\phi(y_{j}),\quad y_{j}\coloneqq\frac{r}{\lmb_{j}},\quad\text{and}\quad\phi_{\ubr{;j}}\coloneqq\frac{1}{\lmb_{j}}\phi_{;j}.
\]
\item For $(\vec{\iota},\vec{\lmb})\in\calP_{J}(1)$, the sum of $Q$-bubbles
is denoted by 
\[
\calQ(\vec{\iota},\vec{\lda})\coloneqq\sum_{j=1}^{J}\iota_{j}Q_{\lmb_{j}}=\sum_{j=1}^{J}Q_{;j}\quad\text{and}\quad\bm{\calQ}(\vec{\iota},\vec{\lmb})=\begin{pmatrix}\calQ(\vec{\iota},\vec{\lmb})\\
0
\end{pmatrix}.
\]
For $\dlt\in(0,1)$, we define a $\dlt$-tubular neighborhood 
\[
\calT_{J}(\dlt)\define\{\bm{u}\in\cup_{m}\calE_{0,m}\colon\exists(\vec{\iota},\vec{\lda})\in\calP_{J}(\dlt)\text{ such that }\|\bm{u}-\bm{\calQ}(\vec{\iota},\vec{\lda})\|_{\dot{\calH}^{1}}<\dlt\}.
\]
\item $f_{\mathbf{i}}(\vec{\iota},\vec{\lda})\define-\frac{k^{2}}{r^{2}}\{f(\calQ(\vec{\iota},\vec{\lda}))-\tsum{j=1}Jf(Q_{;j})\}$
denotes the interaction term.
\end{itemize}

\subsection{\label{subsec:Strategy-of-the-proof}Strategy of proof and organization
of paper}

By time-reversal symmetry, it suffices to construct backward-in-time
bubble tower solutions: 
\begin{thm}[Construction of backward bubble towers]
\label{thm:main-negative-time}Let $k\geq3$ and $J\in\bbN$. Then,
there exists a solution $\bm{u}(t)=(u(t),\rd_{t}u(t))$ to (\ref{eq:k-equiv-WM})
with data in $\calE_{0,J\mathrm{mod}2}\cap\dot{\calH}^{2}$, defined
for all large negative times $t$, such that 
\[
\lim_{t\to-\infty}\{\|u(t)-\sum_{j=1}^{J}(-1)^{j-1}Q_{\gmm_{j}|t|^{-\alp_{j}}}\|_{\dot{H}_{k}^{1}}+\|\rd_{t}u(t)\|_{L^{2}}\}=0,
\]
where $\alp_{j}$ and $\gmm_{j}$ are the universal constants defined
in (\ref{eq:def-alp_j}) and (\ref{eq:def-gmm_j}), respectively.
\end{thm}

As mentioned in Remark~\ref{rem:strategy-difficulty-novelty}, we
use modulation analysis with the backward construction method \cite{Merle1990CMP,Martel2005AJM,Jendrej2019AJM}.
This method, adapted to our setting, consists of the following two
steps:
\begin{enumerate}
\item Construct a family of solutions $\bm{u}^{t_{0}}$ on $[t_{0},T_{boot}]$
for all $t_{0}\leq T_{boot}$ ($T_{boot}$ is some fixed time independent
of $t_{0}$) that exhibit backward-in-time bubbling as in Theorem~\ref{thm:main-negative-time}
on $[t_{0},T_{boot}]$, uniformly in $t_{0}$. (A more precise version
is formulated in Proposition~\ref{prop:solutions_u^t_0}.)
\item Take a limit $t_{0}\to-\infty$ to obtain a solution $\bm{u}$ as
in Theorem~\ref{thm:main-negative-time}. (This is rather standard
and is done in Section~\ref{sec:Proof-of-main-theorem}.)
\end{enumerate}
The heart of the matter is the first step.

For a nonlinear solution $\bm{u}(t)=(u(t),\dot{u}(t))$ to (\ref{eq:k-equiv-WM-vectorial})
in the vicinity of bubble tower configurations, we use a simple decomposition
\begin{equation}
\left|\begin{aligned}u(t) & =\calQ(\vec{\iota},\vec{\lmb}(t))+g(t),\\
\dot{u}(t) & =\tsum j{}b_{j}(t)\Lmb Q_{\ul{;j}}(t)+\dot{g}(t),
\end{aligned}
\right.\label{eq:intro-decomp}
\end{equation}
where $\vec{\lmb}(t)$ and $\vec{b}(t)$ are determined for each time
$t$ through certain orthogonality conditions on $\bm{g}(t)$; see
Section~\ref{subsec:Decomposition}. The evolution equations of $\vec{\lmb}(t)$
and $\vec{b}(t)$ are then obtained by differentiating the orthogonality
conditions in time, from which one can estimate $\lmb_{j,t}$ and
$b_{j,t}$ (Section~\ref{subsec:Modulation-estimates}). 

To realize the first step mentioned above, we need to prepare (a family
of) well-chosen initial data at $t_{0}$ (Section~\ref{subsec:Family-of-initial-data}),
formulate bootstrap assumptions on parameters $\vec{\lmb}(t)$, $\vec{b}(t)$,
and $\bm{g}(t)$ that quantify the bubbling behavior as in Theorem~\ref{thm:main-negative-time}
(Section~\ref{subsec:Family-of-initial-data}), and try to close
bootstrap assumptions (forward in time) up to $T_{boot}$ independent
of $t_{0}$ (Section~\ref{sec:Proof-of-bootstrap}). 

Let us first record the formal ODE system of $\vec{\lmb}(t)$ and
$\vec{b}(t)$. Assume formally that $u(t)=\calQ(\vec{\iota},\vec{\lmb}(t))$
with $\iota_{j}=(-1)^{j-1}$ and $\lmb_{1}(t)\gg\dots\gg\lmb_{J}(t)$
is a solution to (\ref{eq:k-equiv-WM}). Projecting the evolution
equation satisfied by $\calQ(\vec{\iota},\vec{\lmb})$ onto each $\Lmb Q_{\ul{;j}}$
and ignoring some interaction terms not contributing to the dynamics,
one derives the following formal ODE system: 
\begin{equation}
\left|\begin{aligned}\lmb_{j,t} & =-b_{j},\\
b_{j,t} & =-8k^{2}\kpp^{-1}\frac{\lmb_{j}^{k-1}}{\lmb_{j-1}^{k}}\chf_{j\geq2}.
\end{aligned}
\right.\label{eq:intro-formal-ODE}
\end{equation}
This system admits the exact solution ($k\geq3$)
\begin{equation}
\left|\begin{aligned}\lmb_{j}^{\ex}(t) & \coloneqq\gmm_{j}|t|^{-\alp_{j}},\\
b_{j}^{\ex}(t) & \coloneqq-\alp_{j}\gmm_{j}|t|^{-(\alp_{j}+1)},
\end{aligned}
\right.\label{eq:intro-def-lmb-b-exact}
\end{equation}
where $\alp_{j}$ and $\gmm_{j}$ are the universal constants defined
in (\ref{eq:def-alp_j}) and (\ref{eq:def-gmm_j}), respectively. 

To deal with nonlinear solutions, we need to include the error term
$\bm{g}(t)$. Typically, formulating bootstrap assumptions on $\bm{g}(t)$
and closing them constitute the most difficult part of the analysis.
We recall the explanation in Remark~\ref{rem:strategy-difficulty-novelty}.
Simply working in the critical energy space as in \cite{Jendrej2019AJM}
seems insufficient to justify (\ref{eq:intro-formal-ODE}); we observe
that a \emph{higher Sobolev bound} of the form $\|\bm{g}\|_{\dot{\calH}^{2}}\aleq|t|^{-1-}$
can justify (\ref{eq:intro-formal-ODE}) and this is also included
in our set of bootstrap assumptions (Section~\ref{subsec:Family-of-initial-data}).
Note that the $\dot{\calH}^{2}$-framework appeared in \cite{JendrejLawrie2022AnalPDE}
to obtain more refined estimates on two-bubble solutions (and in earlier
\cite{RodnianskiSterbenz2010Ann.Math.,RaphaelRodnianski2012Publ.Math.}
in the single-bubble context). As is typical in blow-up dynamics,
the highest Sobolev control of the form $\|\bm{g}\|_{\dot{\calH}^{2}}\aleq|t|^{-1-}$
cannot be closed by itself. We overcome this difficulty by newly introducing
a \emph{Morawetz-type functional $\calM[\vec{\lmb};g,\dot{g}]$ around
multi-bubbles} with arbitrary $J\in\bbN$ (Section~\ref{sec:Monotonicity-estimate}),
which provide suitable monotonicity estimates on $\bm{g}$. Then,
we form an \emph{energy-Morawetz functional} (Section~\ref{subsec:Energy-Morawetz-functional})
\[
\lan H_{\calQ}g,H_{\calQ}g\ran+\lan\dot{g},H_{\calQ}g\ran+\td{\dlt}\calM[\vec{\lmb};g,\dot{g}]\aeq\|\bm{g}\|_{\dot{\calH}^{2}}^{2}
\]
for a fixed $\td{\dlt}>0$ and exploit monotonicity estimates arising
from $\frac{d}{dt}\calM$ to propagate the control $\|\bm{g}\|_{\dot{\calH}^{2}}\aleq|t|^{-1-}$.
Construction of this functional is based on a Morawetz functional
around single-bubbles (and on the Bogomol'nyi trick) of \cite{RodnianskiSterbenz2010Ann.Math.}.
We believe that this Morawetz-type functional is of independent interest. 

We note that we need additional modulation estimates regarding $\vec{b}$
in order to treat the case of $k=3$; see refined modulation estimates
in Proposition~\ref{prop:modulation-estimates}.

Finally, bootstrap assumptions on some parameters ($\lmb_{2}(t),\dots,\lmb_{J}(t)$
in fact) cannot be propagated forward in time. This is due to their
inherent instability of the formal ODE system (\ref{eq:intro-formal-ODE})
around the exact solution $(\vec{\lmb}^{\ex}(t),\vec{b}^{\ex}(t))$,
which is explained in more detail in Section~\ref{subsec:Linearization-of-formal-ODE}.
This problem is remedied by a standard topological (shooting) argument
as in \cite{CoteMartelMerle2011RMI}, which begins with prescribing
a family of initial data at $t_{0}$ and then finds a special initial
data whose forward-in-time evolution satisfies the bootstrap control
for all $t\in[t_{0},T_{boot}]$.

\section{\label{sec:Preliminaries}Preliminaries}

In Section~\ref{sec:Preliminaries} and Section~\ref{sec:Monotonicity-estimate},
we assume 
\[
k\geq2
\]
unless stated otherwise. Many of the estimates in these sections also
hold when $k=1$, possibly with slight modifications. However, the
key monotonicity estimate in Section~\ref{sec:Monotonicity-estimate}
crucially relies on $k\geq2$.

The goal of this section is to record useful (but standard) facts
in the multi-bubble analysis of (\ref{eq:k-equiv-WM}). For instance,
we will discuss the linearized operators $H_{\calQ}$ around multi-bubbles
and their coercivity estimates in Section~\ref{subsec:Coercivity-estimates}
and record a few multi-bubble interaction estimates in Section~\ref{subsec:Estimates-for-multi-bubble}.
Most of the proofs are standard and hence relegated to Appendix~\ref{sec:computation-for-bubble}.

\subsection{\label{subsec:Coercivity-estimates}Coercivity estimates for linearized
operators around multi-bubble}

In this subsection, we discuss linearized operators around multi-bubbles
and their coercivity estimates. For a function $U$, we denote (recall
$f'(u)=\cos2u$)
\[
H_{U}\coloneqq-\rd_{rr}-\frac{1}{r}\rd_{r}+k^{2}\frac{f'(U)}{r^{2}}.
\]

\uline{Single-bubble linearized operator \mbox{$H_{Q}$}}. We recall
some well-known facts on the single-bubble linearized operator $H_{Q}\eqqcolon H$.
First, we have boundedness: $H:\dot{H}_{k}^{2}\to L^{2}$ and $H:\dot{H}_{k}^{1}\to(\dot{H}_{k}^{1})^{\ast}$.
Next, $H$ is \emph{non-degenerate}, i.e., $\ker H=\mathrm{span}\{\Lmb Q\}$.
More importantly, the fact that $\Lmb Q$ is positive on $(0,\infty)$
implies the crucial \emph{factorization property} 
\begin{equation}
H=A^{\ast}A,\label{eq:one-bubble-factorization}
\end{equation}
($H$ is non-negative in particular) where 
\[
A=-\rd_{y}+\frac{k\cos Q}{y}\quad\text{and}\quad A^{\ast}=\rd_{y}+\frac{1+k\cos Q}{y}.
\]
We also note that $A:\dot{H}_{k}^{1}\to L^{2}$ and $\ker A=\mathrm{span}\{\Lmb Q\}$.
With this factorization property, its supersymmetric conjugate $\td H\coloneqq AA^{\ast}$
exhibits a repulsive character that induces a Morawetz-type monotonicity
estimate near one bubble. This will be discussed in more detail in
Section~\ref{subsec:Monotonicity-near-single-bubble}.

Now, we fix the orthogonality profile 
\begin{equation}
\calZ\coloneqq\frac{\chi\Lmb Q}{\int\chi(\Lmb Q)^{2}}\quad\text{so that}\quad\lan\calZ,\Lmb Q\ran=1.\label{eq:def-calZ}
\end{equation}
This transversality combined with (\ref{eq:one-bubble-factorization}),
$\ker A=\mathrm{span}\{\Lmb Q\}$, and $\calZ\in(\dot{H}_{k}^{2})^{\ast}\cap(\dot{H}_{k}^{1})^{\ast}$
implies the following \emph{coercivity estimates}: 
\begin{align}
\lan Hg,g\ran & \geq c\|g\|_{\dot{H}_{k}^{1}}^{2}-C\lan\calZ,g\ran^{2}\qquad\forall g\in\dot{H}_{k}^{1},\label{eq:H-Hdot1-coer}\\
\lan Hg,Hg\ran & \geq c\|g\|_{\dot{H}_{k}^{2}}^{2}-C\lan\calZ,g\ran^{2}\qquad\forall g\in\dot{H}_{k}^{2},\label{eq:H-Hdot2-coer}
\end{align}
where $c,C>0$ are some constants independent of $g$. Note that,
when $k\geq4$, one may take $\calZ=\Lmb Q/\|\Lmb Q\|_{L^{2}}^{2}$
(which can be more convenient for modulation estimates in later sections)
as $\Lmb Q(y)\aeq\lan y\ran^{-k}$ decays sufficiently fast. 

Let us briefly record scaled versions of the linearized operators
above. For $\lmb\in(0,\infty)$, we denote $H_{\lmb}\coloneqq H_{Q_{\lmb}}$
and it admits the factorization 
\begin{equation}
H_{\lda}=A_{\lda}^{*}A_{\lda},\text{ where }A_{\lda}=-\rnd_{r}+\frac{k\cos Q_{\lda}}{r}\text{ and }A_{\lda}^{*}=\rnd_{r}+\frac{1+k\cos Q_{\lda}}{r}.\label{eq:def-scaled-H-A}
\end{equation}
Its supersymmetric conjugate is denoted by $\td H_{\lmb}\coloneqq A_{\lmb}A_{\lmb}^{\ast}$.

\uline{Multi-bubble linearized operator \mbox{$H_{\calQ}$}}. We
turn to discuss coercivity estimates for the linearized operator $H_{\calQ}$
around multi-bubbles $\calQ=\calQ(\vec{\iota},\vec{\lmb})=\sum_{j=1}^{J}\iota_{j}Q_{\lmb_{j}}$.
If $(\vec{\iota},\vec{\lmb})\in\calP_{J}(\alp)$ for a sufficiently
small $\alp>0$, then $H_{\calQ}$ can be well approximated by $H_{\lmb_{j}}$
in each soliton region. Applying the single-bubble coercivity estimates
(\ref{eq:H-Hdot1-coer})--(\ref{eq:H-Hdot2-coer}) in each soliton
region and summing them up, one can obtain coercivity estimates in
the multi-bubble setting. The following lemma collects such estimates,
whose proof is essentially contained in \cite{JendrejLawrie2025JAMS,KimMerle2025CPAM}.
\begin{lem}[Coercivity estimates for $H_{\calQ}$]
Let $J\in\bbN$. There exist constants $\alp_{\coer}\in(0,\frac{1}{10})$
and $c,C>0$ with the following property.
\begin{enumerate}
\item For $(\vec{\iota},\vec{\lmb})\in\calP_{J}(1)$, we have 
\begin{align}
|\lan H_{\calQ}g,h\ran| & \leq C\|g\|_{\dot{H}_{k}^{1}}\|h\|_{\dot{H}_{k}^{1}},\qquad\forall g,h\in\dot{H}_{k}^{1},\label{eq:H_calQ_Hdot1-bdd}\\
\|H_{\calQ}g\|_{L^{2}} & \leq C\|g\|_{\dot{H}_{k}^{2}},\qquad\qquad\quad\forall g\in\dot{H}_{k}^{2}.\label{eq:H_calQ-Hdot2-bdd}
\end{align}
\item For $(\vec{\iota},\vec{\lmb})\in\calP_{J}(\alp_{\coer})$, we have
\begin{align}
\lan H_{\calQ}g,g\ran & \geq c\|g\|_{\dot{H}_{k}^{1}}^{2}-C\tsum{i=1}J\lmb_{i}^{-2}\lan\calZ_{\ul{;i}},g\ran^{2}\qquad\forall g\in\dot{H}_{k}^{1},\label{eq:H_calQ-Hdot1-coer}\\
\lan H_{\calQ}g,H_{\calQ}g\ran & \geq c\|g\|_{\dot{H}_{k}^{2}}^{2}-C\tsum{i=1}J\lmb_{i}^{-4}\lan\calZ_{\ul{;i}},g\ran^{2}\qquad\forall g\in\dot{H}_{k}^{2}.\label{eq:H_calQ-Hdot2-coer}
\end{align}
\end{enumerate}
\end{lem}

\begin{proof}
\uline{Proof of \mbox{(\ref{eq:H_calQ_Hdot1-bdd})} and \mbox{(\ref{eq:H_calQ-Hdot2-bdd})}}.
These follow from (we integrated by parts for the first)
\[
|\lan H_{\calQ}g,h\ran|\aleq\tint{}{}|g|_{-1}|h|_{-1}\quad\text{and}\quad|H_{\calQ}g|\aleq|g|_{-2}.
\]

\uline{Proof of \mbox{(\ref{eq:H_calQ-Hdot1-coer})}}. See \cite[Lemma 2.19]{JendrejLawrie2025JAMS}.

\uline{Proof of \mbox{(\ref{eq:H_calQ-Hdot2-coer})}}. See \cite[Lemma 2.7]{KimMerle2025CPAM}
together with Section~1.4.2 therein.
\end{proof}

\subsection{\label{subsec:Estimates-for-multi-bubble}Estimates for multi-bubble
interactions}
\begin{lem}
\label{lem:2.2}Let $k>1$. For $\lda_{in},\lda_{out}\in(0,\infty)$
with $\lda_{in}\leq\lda_{out}$, we have 
\begin{align}
\|\Lmb Q_{\ul{\lmb_{in}}}\Lmb Q_{\ul{\lmb_{out}}}\|_{L^{1}} & \aleq\Big(\frac{\lda_{in}}{\lda_{out}}\Big)^{k-1},\label{eq:bubble-int}\\
\|r^{-3}\sqrt{(y_{in}+1)(r+\lda_{in})}\Lda Q_{\lda_{in}}\Lda Q_{\lda_{out}}\|_{L^{2}} & \aleq\frac{1}{\lda_{in}^{3/2}}\Big(\frac{\lda_{in}}{\lda_{out}}\Big)^{k}.\label{eq:bubble-dual-Mor}
\end{align}
\end{lem}

\begin{proof}
See Appendix~\ref{sec:computation-for-bubble}.
\end{proof}
\begin{lem}[Energy and interaction of multi-bubbles \cite{JendrejLawrie2025JAMS}]
Let $k\geq1$ and $J\in\bbN$. For $(\vec{\iota},\vec{\lda})\in\calP_{J}(1)$,
we have 
\begin{align}
|E[\bm{\calQ}(\vec{\iota},\vec{\lda})]-JE[\bm{Q}]| & \aleq\max_{j\in\{2,\dots,J\}}\Big(\frac{\lmb_{j}}{\lmb_{j-1}}\Big)^{k},\label{eq:multi-bubble-energy}\\
|\lan-\rd_{rr}\calQ-\frac{1}{r}\rd_{r}\calQ+k^{2}\frac{f(\calQ)}{r^{2}},g\ran| & \aleq\|g\|_{\dot{H}_{k}^{1}}\cdot\max_{j\in\{2,\dots,J\}}\Big(\frac{\lmb_{j}}{\lmb_{j-1}}\Big)^{k},\qquad\forall g\in\dot{H}_{k}^{1},\label{eq:multi-bubble-small-linear}
\end{align}
and the following inner product estimate for the interaction 
\begin{equation}
\lan\Lmb Q_{;i},f_{\mathbf{i}}(\vec{\iota},\vec{\lda})\ran=\Big\{-\iota_{i-1}\iota_{i}8k^{2}+\calO\Big(\max_{j\in\{2,\dots,J\}}\frac{\lmb_{j}}{\lmb_{j-1}}\Big)\Big\}\cdot\Big(\frac{\lmb_{i}}{\lmb_{i-1}}\Big)^{k}+\calO\Big(\Big(\frac{\lmb_{i+1}}{\lmb_{i}}\Big)^{k}\Big),\label{eq:inner-prod-interaction}
\end{equation}
under the convection that $\lda_{0}=\infty$ and $\lda_{J+1}=0$.
\end{lem}

\begin{proof}
For the proof of (\ref{eq:multi-bubble-energy}) and (\ref{eq:multi-bubble-small-linear}),
see (2.20) and (2.21) in \cite{JendrejLawrie2025JAMS}, respectively.
For the proof of (\ref{eq:inner-prod-interaction}), a very similar
estimate is proved in \cite[Lemma 2.23]{JendrejLawrie2025JAMS}; we
provide its proof in Appendix~\ref{sec:computation-for-bubble}.
\end{proof}
\begin{lem}[More estimates]
\label{lem:2.4}Let $J\in\bbN$. For $(\vec{\iota},\vec{\lmb})\in\calP_{J}(1)$,
we have the following pointwise estimates (for $k\geq1$) 
\begin{align}
|H_{\calQ}-H_{\lda_{i}}|_{-\ell} & \aleq_{\ell}r^{-2-\ell}\sum_{j:j\neq i}\Lda Q_{\lda_{j}},\;\quad\qquad\qquad\forall\ell\in\bbN,\label{eq:pointwise-G_i}\\
|f_{\mathbf{i}}(\vec{\iota},\vec{\lda})|_{-\ell} & \aleq_{\ell}r^{-2-\ell}\sum_{i,j:i>j}\Lda Q_{\lda_{j}}\Lda Q_{\lda_{i}},\quad\qquad\forall\ell\in\bbN,\label{eq:pointwise-interaction}
\end{align}
and weighted $L^{2}$ estimates (for $k>1$)
\begin{align}
\|H_{\calQ}\Lda Q_{;i}\|_{L^{2}} & \aleq\frac{1}{\lmb_{i+1}}\Big(\frac{\lda_{i+1}}{\lda_{i}}\Big)^{k}+\frac{1}{\lmb_{i}}\Big(\frac{\lda_{i}}{\lda_{i-1}}\Big)^{k},\label{eq:pre-H_Q-interaction-0}\\
\|rH_{\calQ}\Lda Q_{;i}\|_{L^{2}} & \aleq(1+\log(\frac{\lmb_{i}}{\lmb_{i+1}}))^{1/2}\Big(\frac{\lda_{i+1}}{\lda_{i}}\Big)^{k}+(1+\log(\frac{\lmb_{i-1}}{\lmb_{i}}))^{1/2}\Big(\frac{\lda_{i}}{\lda_{i-1}}\Big)^{k},\label{eq:pre-H_Q-interaction-1}\\
\||f_{\mathbf{i}}(\vec{\iota},\vec{\lmb})|_{-1}\|_{L^{2}} & \aleq\max_{j\in\{2,\dots,J\}}\frac{1}{\lmb_{j}^{2}}\Big(\frac{\lda_{j}}{\lda_{j-1}}\Big)^{k},\label{eq:H1-est-interatction}
\end{align}
\textup{under the} convection that $\lda_{0}=\infty$ and $\lda_{J+1}=0$.
Here, $H_{\calQ}-H_{\lmb_{i}}$ is regarded as a function $\frac{k^{2}}{r^{2}}(\cos(2\calQ)-\cos(2Q_{;i}))$.
\end{lem}

\begin{proof}
See Appendix~\ref{sec:computation-for-bubble}.
\end{proof}

\subsection{\label{subsec:Decomposition}Decomposition}

For solutions $\bm{u}(t)$ to (\ref{eq:k-equiv-WM}) in the vicinity
of multi-bubbles, we will decompose $\bm{u}(t)$ by applying the following
lemma at each time $t$.
\begin{lem}[Static modulation]
\label{lem:decomposition}Let $J\in\bbN$. There exist constants
$\eta_{0}\in(0,\tfrac{1}{10})$ and $C_{J}\geq1$ with the following
property. For any $\bm{u}\in\calT_{J}(\eta_{0})$, there exist unique
$\bm{g}\in\dot{\calH}^{1}$, $(\vec{\iota},\vec{\lda})\in\calP_{J}(2\eta_{0})$,
and $\vec{b}\in\bbR^{J}$ such that the decomposition 
\begin{equation}
\left|\begin{aligned}u & =\calQ(\vec{\iota},\vec{\lmb})+g\\
\dot{u} & =\tsum j{}b_{j}\Lmb Q_{\ul{;j}}+\dot{g}
\end{aligned}
\right.\label{eq:decomp-u}
\end{equation}
satisfies the orthogonality conditions 
\begin{align}
\lan\calZ_{\ubr{;i}},g\ran & =\lan\calZ_{\ubr{;i}},\dot{g}\ran=0\qquad\forall i\in\{1,\dots,J\},\label{eq:orthog-g}
\end{align}
and smallness 
\begin{align}
\|\bm{g}\|_{\dot{\calH}^{1}} & \leq C_{J}\eta_{0}.\label{eq:decomp-small}
\end{align}
Moreover, the map $\bm{u}\mapsto(\bm{g},\vec{\iota},\vec{\lmb},\vec{b})$
is continuous.
\end{lem}

\begin{proof}[Sketch of the proof]
This is a standard consequence of (a quantitative version of) the
implicit function theorem; let us only give a brief sketch of the
proof and rather refer the reader to \cite[Lemma 2.25]{JendrejLawrie2025JAMS}
and \cite[Lemma B.1]{DuyckaertsKenigMerle2023Acta}.

Let $\eta_{0}>0$ be a constant that can shrink in the course of the
proof. Note first that if the decomposition exists, then $\vec{\iota}\in\{\pm\}^{J}$
must be unique due to (\ref{eq:decomp-small}) (see also \cite[Lemma 2.27]{JendrejLawrie2025JAMS}).
We also note that the statement can be considered separately for $u$
and $\dot{u}$.

Consider first the decomposition for $u$. We consider the map 
\[
\vec{F}:(u,\vec{\lmb})\mapsto(\lan\lmb_{i}^{-2}\calZ_{;i},u-\calQ(\vec{\iota},\vec{\lmb})\ran)_{i\in\{1,\dots,J\}},
\]
where $\|u-\calQ(\vec{\iota},\vec{\td{\lmb}})\|_{\dot{H}_{k}^{1}}<\eta_{0}$
for some $(\vec{\iota},\vec{\td{\lmb}})\in\calP_{J}(\eta_{0})$ and
$\vec{\lmb}\in\calP_{J}(2\eta_{0})$. Observe first that $\vec{F}(\calQ(\vec{\iota},\vec{\lmb}),\vec{\lmb})=0$.
Next, observe using $\lan\calZ,\Lmb Q\ran=1$ that 
\begin{align*}
\lmb_{j}\rd_{\lmb_{j}}F_{i}(u,\vec{\lmb}) & =\lan\lmb_{i}^{-2}\calZ_{;i},\Lmb Q_{;j}\ran-\chf_{i=j}\lan\lmb_{i}^{-2}\Lmb_{-1}\calZ_{;i},u-\calQ(\vec{\iota},\vec{\lmb})\ran\\
 & =\chf_{i=j}+\chf_{i\neq j}\tfrac{\lmb_{j}}{\lmb_{i}}\lan\calZ_{\ul{;i}},\Lmb Q_{\ul{;j}}\ran+\calO(\|u-\calQ(\vec{\iota},\vec{\lmb})\|_{\dot{H}_{k}^{1}}).
\end{align*}
By the choice (\ref{eq:def-calZ}) of $\calZ$, one has $\chf_{i\neq j}\tfrac{\lmb_{j}}{\lmb_{i}}\lan\calZ_{\ul{;i}},\Lmb Q_{\ul{;j}}\ran=\calO(\eta_{0})$
(this is a variant of (\ref{eq:bubble-int}); see \cite[Corollary 2.21]{JendrejLawrie2025JAMS}).
Therefore, the $J\times J$ matrix $(\lmb_{j}\rd_{\lmb_{j}}F_{i}(u,\vec{\lmb}))_{i,j}$
is close to the identity matrix as long as $\vec{\lmb}$ is close
to $\vec{\td{\lmb}}$. By a quantitative version of the implicit function
theorem, this implies the unique existence of $\vec{\lmb}$ in the
vicinity of $\vec{\td{\lmb}}$ ensuring the orthogonality condition
$\lan\calZ_{;i},u-\calQ(\vec{\iota},\vec{\lmb})\ran=0$ and smallness
(\ref{eq:decomp-small}) for $g$. Moreover, $\vec{\lmb}$ depends
(Lipschitz) continuously on $u$. The uniqueness of $\vec{\lmb}$
in a larger class $\calP_{J}(2\eta_{0})$, i.e., not necessarily in
the vicinity of $\vec{\td{\lmb}}$, follows from (\ref{eq:decomp-small})
for $g$ (see also \cite[Lemma 2.27]{JendrejLawrie2025JAMS}).

For $\dot{u}$, observe that the decomposition (\ref{eq:decomp-u})
is linear. For $k\geq2$, (\ref{eq:bubble-int}) gives $\lan\calZ_{\ul{;i}},\Lmb Q_{\ul{;j}}\ran=\chf_{i=j}+\calO(\eta_{0})$,
so the $J\times J$ matrix $(\lan\calZ_{\ul{;i}},\Lmb Q_{\ul{;j}}\ran)_{i,j}$
is close to the identity matrix. For $k=1$, this matrix is still
invertible with uniformly bounded inverse (in fact, close to some
fixed triangular matrix; see \cite[Corollary 2.21]{JendrejLawrie2025JAMS}).
Thus the unique decomposition for $\dot{u}$ as well as the estimate
$\|\dot{g}\|_{L^{2}}\aleq\|\dot{u}\|_{L^{2}}$ follow.
\end{proof}

\section{\label{sec:Monotonicity-estimate}Monotonicity estimate near multi-bubble}

This section introduces the key Morawetz-type monotonicity estimate
near multi-bubbles (Proposition~\ref{prop:Morawetz}). This estimate
serves as the key ingredient to control the $\dot{\calH}^{2}$-energy
(the highest Sobolev norm) of the radiation in later modulation analysis. 

As in the previous section, we assume 
\[
k\geq2.
\]
Before we give a full statement for the monotonicity estimate near
multi-bubbles, it is instructive to discuss first a Morawetz-type
monotonicity estimate near single-bubbles (Section~\ref{subsec:Monotonicity-near-single-bubble}),
which was introduced in \cite{RodnianskiSterbenz2010Ann.Math.}. Building
upon the single-bubble case, in Section~\ref{subsec:Monotonicity-near-multi-bubbles},
we state and prove the key monotonicity estimate near multi-bubbles.

\subsection{\label{subsec:Monotonicity-near-single-bubble}Monotonicity near
single-bubble}

We begin with a Morawetz-type monotonicity estimate for the linearized
dynamics around $\bm{Q}=(Q,0)$. In this paragraph, we set $\lmb=1$
so that $r=y$. Suppose $\eps=\eps(t,r)$ is a solution to 
\[
\rd_{tt}\eps=-H\eps.
\]
Using the factorization $H=A^{\ast}A$, we conjugate $A$ to obtain
\[
\rd_{tt}A\eps=-\td HA\eps.
\]
This supersymmetric conjugate $\td H=AA^{\ast}$ exhibits a repulsive
character, that is, 
\[
\td H=-\rd_{yy}-\frac{1}{y}\rd_{y}+\frac{1}{y^{2}}\td V,\qquad\td V=k^{2}+1+2k\cos Q,
\]
where the potential satisfies 
\begin{equation}
\td V\geq(k-1)^{2}\qquad\text{and}\qquad-y\rd_{y}\td V\geq0.\label{eq:repulsivity}
\end{equation}
Thanks to (\ref{eq:repulsivity}), the following virial-type computation
formally gives a \emph{monotonicity estimate}: 
\begin{align*}
\frac{d}{dt}\lan A\rd_{t}\eps,\Lmb_{0}(A\eps)\ran & =\lan A\rd_{tt}\eps,\Lmb_{0}(A\eps)\ran+\lan A\rd_{t}\eps,\Lmb_{0}(A\rd_{t}\eps)\ran=-\lan\td HA\eps,\Lmb_{0}(A\eps)\ran+0\\
 & =-\int\Big\{|\rd_{y}A\eps|^{2}+(\td V-\frac{1}{2}y\rd_{y}\td V)\frac{|A\eps|^{2}}{y^{2}}\Big\}\leq-\int\Big\{|\rd_{y}A\eps|^{2}+\td V\frac{|A\eps|^{2}}{y^{2}}\Big\}.
\end{align*}

In applications, we suitably truncate the unbounded operator $\Lmb_{0}$
to obtain a Morawetz-type monotonicity estimate. It will also be convenient
to add some scalings for multi-bubble generalization. For a function
$\psi(y)\aeq\frac{y}{1+y}$ to be chosen in the following lemma, introduce
a truncated version of $\Lmb_{0}$ rescaled by $\lmb\in(0,\infty)$:
\begin{align*}
\Lda_{\psi,\lda}g & \coloneqq\psi_{\lda}\rnd_{r}g+\tfrac{1}{2}(\rnd_{r}\psi_{\lmb}+r^{-1}\psi_{\lda})g,\\
\Lmb_{\psi} & \coloneqq\Lmb_{\psi,1}.
\end{align*}
Note that $\Lmb_{\psi,\lmb}$ is anti-symmetric and $(\Lmb_{\psi}\eps)_{\lmb}=\lmb(\Lmb_{\psi,\lmb}\eps_{\lmb})$. 
\begin{lem}[{Choice of $\psi$ \cite[Lemma 2.1]{Kim2025JEMS}}]
Let $k\geq2$. There exists a smooth function $\psi:(0,\infty)\to\bbR$
with the following properties.
\begin{enumerate}
\item (Smooth extension on $\bbR^{2}$) $\psi=\rnd_{y}\phi$ for some function
$\phi$ admitting a smooth radially symmetric extension on $\bbR^{2}$. 
\item (Bounds for $\psi$) We have $0\leq\psi\leq1$ and 
\begin{equation}
\psi\aeq\frac{y}{1+y}.\label{eq:psi-bound}
\end{equation}
\item (Pointwise estimates for $\Lda_{\psi,\lda}$) For $\lmb\in(0,\infty)$
and $y=r/\lmb$, we have
\begin{align}
|\Lda_{\psi,\lda}g| & \aleq y(1+y)^{-1}|g|_{-1},\label{eq:Lda_psi}\\
|(\lda\rnd_{\lda}\Lda_{\psi,\lda})g| & \aleq y(1+y)^{-2}|g|_{-1}.\label{eq:Lda_psi_lmb-deriv}
\end{align}
\item (Morawetz-type repulsivity) We have 
\begin{equation}
\lan\td Hf,\Lda_{\psi}f\ran\ageq\int\Big(\frac{|\rnd_{y}f|^{2}}{(1+y)^{2}}+\frac{|y^{-1}f|^{2}}{1+y}\Big).\label{eq:Morawetz-repuls}
\end{equation}
\end{enumerate}
\end{lem}

\begin{proof}
For a small fixed $\dlt>0$, choose 
\[
\psi(y)=\frac{y}{\lan y\ran}-\dlt\frac{y}{\lan y\ran^{2}}.
\]
The proof of the first, second, and fourth items is given in \cite[Lemma 2.1]{Kim2025JEMS}
after replacing $m$ and $\psi'$ by $k-1$ and $\psi$, respectively.
The proof of (\ref{eq:Lda_psi}) and (\ref{eq:Lda_psi_lmb-deriv})
follows from direct computations with $|\lmb\rd_{\lmb}\psi_{\lmb}|\aleq y\lan y\ran^{-2}$. 
\end{proof}
The above $\Lmb_{\psi,\lmb}$ suggests the Morawetz functional $\lan A_{\lmb}\rd_{t}g,\Lmb_{\psi,\lmb}(A_{\lmb}g)\ran$
at scale $\lmb\in(0,\infty)$. Assuming $\rd_{tt}g=-H_{\lmb}g$ with
$\lmb$ fixed in time, we have 
\[
\frac{d}{dt}\lan A_{\lmb}\rd_{t}g,\Lmb_{\psi,\lmb}(A_{\lmb}g)\ran=-\lan\td H_{\lmb}A_{\lmb}g,\Lmb_{\psi,\lmb}A_{\lmb}g\ran.
\]
Applying a rescaled version of (\ref{eq:Morawetz-repuls}) and suitable
coercivity estimates (Lemma~\ref{lem:coer-A}) for $A_{\lmb}g$,
one obtains a Morawetz-type monotonicity estimate for $g$. As is
well-known, adding a small correction to this Morawetz functional
yields a control on $\rd_{t}g$ as well. The following lemma records
a monotonicity estimate around single-bubbles in vectorial form written
without time dependence. 
\begin{lem}[Morawetz-type functional around $\bm{Q}_{\lmb}$]
Let $k\geq2$. There exists $\dlt>0$ such that the bilinear form
\begin{equation}
\calM_{0}[\lmb;g,\dot{g}]=\lan A_{\lmb}\dot{g},\Lmb_{\psi,\lmb}(A_{\lmb}g)\ran-\dlt\lan A_{\lmb}\dot{g},\lda(r+\lda)^{-2}A_{\lmb}g\ran,\label{eq:def-M_0}
\end{equation}
defined for $\lmb\in(0,\infty)$ and $\bm{g}\in\dot{\calH}^{2}$,
satisfies the following properties. 
\begin{enumerate}
\item (Boundedness) We have 
\begin{align}
|\calM_{0}[\lmb;g,\dot{g}]| & \aleq\|y(1+y)^{-1}|\dot{g}|_{-1}|g|_{-2}\|_{L^{1}}\aleq\|\bm{g}\|_{\dot{\calH}^{2}}^{2}.\label{eq:M0-bdd}
\end{align}
\item (Monotonicity) We have (recall the notation $y=r/\lmb$)
\begin{equation}
-\{\calM_{0}[\lmb;g,-H_{\lmb}g]+\calM_{0}[\lmb;\dot{g},\dot{g}]\}\geq c\|\bm{g}\|_{\Mor(\lmb)}^{2}-C\{\lmb^{7}\lan\calZ_{\lmb},g\ran^{2}+\lmb^{5}\lan\calZ_{\lmb},\dot{g}\ran^{2}\}\label{eq:M0-monotonicity}
\end{equation}
for some constant $c,C>0$, where 
\begin{equation}
\|\bm{g}\|_{\Mor(\lmb)}^{2}\coloneqq\int\Big\{\frac{|\rd_{rr}g|^{2}+|\dot{g}|_{-1}^{2}}{(\lmb+r)(1+y)}+\frac{r^{-2}|g|_{-1}^{2}}{\lmb+r}\Big\}.\label{eq:def-Mor_lmb-norm}
\end{equation}
\item ($\lmb\rd_{\lmb}$-estimate) We have 
\begin{equation}
|\lmb\rd_{\lmb}\calM_{0}[\lmb;g,\dot{g}]|\aleq\lda^{1/2}\|\bm{g}\|_{\dot{\calH}^{2}}\|\bm{g}\|_{\Mor(\lmb)}.\label{eq:M0-lmb-rd-lmb-bdd}
\end{equation}
\item (Cancellation of modulation vector) We have 
\begin{equation}
\calM_{0}[\lmb;g,\Lmb Q_{\ubr{\lmb}}]=0\label{eq:M0-bubble-can}
\end{equation}
\end{enumerate}
\end{lem}

\begin{proof}
By scaling, we may assume $\lmb=1$ (hence $r=y$). 

\uline{Proof of \mbox{(\ref{eq:M0-bdd})}}. This follows from applying
(\ref{eq:Lda_psi}) to (\ref{eq:def-M_0}).

\uline{Proof of \mbox{(\ref{eq:M0-bubble-can})}}. This follows
directly from $A\Lda Q=0$.

\uline{Proof of \mbox{(\ref{eq:M0-monotonicity})}}. We begin with
\begin{align*}
 & -\{\calM_{0}[1;g,-Hg]+\calM_{0}[1;\dot{g},\dot{g}]\}\\
 & =\lan AHg,\Lda_{\psi}Ag\ran-\dlt\lan AHg,(y+1)^{-2}Ag\ran-\lan A\dot{g},\Lmb_{\psi}(A\dot{g})\ran+\dlt\lan A\dot{g},(y+1)^{-2}A\dot{g}\ran.
\end{align*}
For the first term, we apply $AH=\td HA$, (\ref{eq:Morawetz-repuls}),
to have 
\[
\lan AHg,\Lda_{\psi}Ag\ran=\lan\td HAg,\Lda_{\psi}Ag\ran\ageq\int\Big\{\frac{|\rnd_{y}(Ag)|^{2}}{(1+y)^{2}}+\frac{|y^{-1}Ag|^{2}}{1+y}\Big\}.
\]
For the second term, we have 
\[
|\dlt\lan AHg,(y+1)^{-2}Ag\ran|=|\dlt\lan A^{\ast}Ag,A^{*}((y+1)^{-2}Ag)\ran|\aleq\dlt\int\frac{|Ag|_{-1}^{2}}{(1+y)^{2}}.
\]
The third term vanishes by the anti-symmetry of $\Lda_{\psi}$. The
last term is 
\[
\dlt\lan A\dot{g},(y+1)^{-2}A\dot{g}\ran=\dlt\int\frac{|A\dot{g}|^{2}}{(1+y)^{2}}.
\]
Summing the previous three displays, choosing $\dlt>0$ small, we
obtain 
\[
-\{\calM_{0}[1;g,-Hg]+\calM_{0}[1;\dot{g},\dot{g}]\}\ageq\int\Big\{\frac{|\rnd_{y}(Ag)|^{2}}{(1+y)^{2}}+\frac{|y^{-1}Ag|^{2}}{1+y}\Big\}+\dlt\int\frac{|A\dot{g}|^{2}}{(1+y)^{2}}.
\]
Finally, applying the coercivity estimates (\ref{eq:coer-A-1}) to
$Ag$ and (\ref{eq:coer-A-2}) to $A\dot{g}$, and ignoring $\dlt$
as it is fixed, we conclude 
\begin{align*}
 & -\{\calM_{0}[1;g,-H_{\lmb}g]+\calM_{0}[1;\dot{g},\dot{g}]\}\\
 & \geq c\int\Big\{\frac{|\rd_{yy}g|^{2}+|\dot{g}|_{-1}^{2}}{(1+y)^{2}}+\frac{y^{-2}|g|_{-1}^{2}}{1+y}\Big\}-C\{\lan\calZ,g\ran^{2}+\lan\calZ,\dot{g}\ran^{2}\}
\end{align*}
for some constants $c,C>0$, which is (\ref{eq:M0-monotonicity}).

\uline{Proof of \mbox{(\ref{eq:M0-lmb-rd-lmb-bdd})}}. Taking $\lmb\rd_{\lmb}$
to $\calM_{0}[\lmb;g,\dot{g}]=\lan A_{\lmb}\dot{g},(\Lmb_{\psi,\lmb}-\dlt\lda(r+\lda)^{-2})A_{\lmb}g\ran$,
we get 
\begin{align*}
\lmb\rd_{\lmb}\calM_{0}[\lda;g,\dot{g}] & =\lan(\lmb\rnd_{\lda}A_{\lda})\dot{g},(\Lmb_{\psi,\lmb}-\dlt\tfrac{\lmb}{(r+\lmb)^{2}})A_{\lmb}g\ran+\lan A_{\lmb}\dot{g},(\Lmb_{\psi,\lmb}-\dlt\tfrac{\lmb}{(r+\lmb)^{2}})\{(\lmb\rnd_{\lda}A_{\lda})g\}\ran\\
 & \quad+\lan A_{\lmb}\dot{g},(\lmb\rd_{\lmb}\Lmb_{\psi,\lmb}-\dlt r^{-1}\calO(\tfrac{y}{(y+1)^{2}}))A_{\lmb}g\ran.
\end{align*}
By scaling, we may assume $\lda=1$. Applying (\ref{eq:Lda_psi})
for $\Lmb_{\psi}$, $|\tfrac{\lmb}{(r+\lmb)^{2}}|\aleq r^{-1}$, and
(\ref{eq:Lda_psi_lmb-deriv}) for $\lmb\rd_{\lmb}\Lmb_{\psi,\lmb}$,
we obtain 
\begin{align*}
 & \big|\lmb\rd_{\lmb}|_{\lmb=1}\calM_{0}[\lda;g,\dot{g}]\big|\\
 & \aleq\|(\lmb\rnd_{\lda}A_{\lda})\dot{g}\|_{L^{2}}\||Ag|_{-1}\|_{L^{2}}+\|A\dot{g}\|_{L^{2}}\||(\lmb\rnd_{\lda}A_{\lda})g|_{-1}\|_{L^{2}}+\|A\dot{g}\cdot(y+1)^{-1}|Ag|_{-1}\|_{L^{1}}\\
 & \aleq\||g|_{-2}\|_{L^{2}}\|\lan y\ran^{-2k}r^{-1}\dot{g}\|_{L^{2}}+\||\dot{g}|_{-1}\|_{L^{2}}\|\lan y\ran^{-2k}r^{-1}|g|_{-1}\|_{L^{2}}+\||\dot{g}|_{-1}\|_{L^{2}}\|(y+1)^{-1}|g|_{-2}\|_{L^{2}}\\
 & \aleq\|\bm{g}\|_{\dot{\calH}^{2}}\|\bm{g}\|_{\Mor(1)}.
\end{align*}
This completes the proof.
\end{proof}

\subsection{\label{subsec:Monotonicity-near-multi-bubbles}Monotonicity near
multi-bubble}

In contrast to the single-bubble setting, where the monotonicity estimate
heavily relies on the factorization $H=A^{\ast}A$, the difficulty
of obtaining a monotonicity in the multi-bubbling setting is, to our
knowledge, the lack of such a factorization for the linearized operator
$H_{\calQ}$ around multi-bubbles. In addition, each $\calM_{0}[\lmb_{j};g,\dot{g}]$
cannot enjoy a global-in-space monotonicity control due to the multi-bubble
nature. However, each $\calM_{0}[\lmb_{j};g,\dot{g}]$ enjoys a monotonicity
control in each soliton region. Our key observation is that we can
recover a global-in-space monotonicity by taking a careful linear
combination of $\calM_{0}[\lmb_{j};g,\dot{g}]$.
\begin{prop}[Morawetz-type functional around $\bm{\calQ}(\vec{\iota},\vec{\lmb})$]
\label{prop:Morawetz}Let $k\geq2$ and $J\in\bbN$. For all small
$\dlt_{0}>0$, define the bilinear form 
\[
\calM[\vec{\lda};g,\dot{g}]\coloneqq\sum_{j=1}^{J}\dlt_{0}^{j-1}\calM_{0}[\lda_{j};g,\dot{g}],
\]
for $\vec{\lmb}\in\calP_{J}(1)$ and $\bm{g}\in\dot{\calH}^{2}$,
where we recall $\calM_{0}$ from (\ref{eq:def-M_0}). 
\begin{enumerate}
\item (Boundedness) For any $\bm{g},\bm{h}\in\dot{\calH}^{2}$, we have
\begin{align}
|\calM[\vec{\lmb};g,\dot{g}]| & \aleq\||g|_{-2}|\dot{g}|_{-1}\|_{L^{1}}\aleq\|\bm{g}\|_{\dot{\calH}^{2}}^{2}.\label{eq:Morawetz-bdd}
\end{align}
\item (Monotonicity) There exists $\alp^{*}=\alp^{*}(\dlt_{0})\in(0,1)$
such that for all $(\vec{\iota},\vec{\lda})\in\calP_{J}(\alp^{*})$,
the following hold. For any $\bm{g}\in\dot{\calH}^{2}$ satisfying
the orthogonality conditions 
\begin{equation}
\lan\calZ_{;i},g\ran=\lan\calZ_{;i},\dot{g}\ran=0,\qquad\forall i\in\{1,\dots,J\},\label{eq:mor-orth}
\end{equation}
we have
\begin{equation}
-\{\calM[\vec{\lmb};g,-H_{\calQ(\vec{\iota},\vec{\lda})}g]+\calM[\vec{\lmb};\dot{g},\dot{g}]\}\geq c\sum_{j=1}^{J}\dlt_{0}^{j-1}\|\bm{g}\|_{\Mor(\lda_{j})}^{2}\eqqcolon c\|\bm{g}\|_{\Mor}^{2}\label{eq:Morawetz-monotonicity}
\end{equation}
for some $c>0$ independent of $\dlt_{0}$. 
\item ($\lmb\rd_{\lmb}$-estimates) For any $\bm{g}\in\dot{\calH}^{2}$,
we have 
\begin{equation}
|\lmb_{j}\rd_{\lmb_{j}}\calM[\vec{\lmb};g,\dot{g}]|\aleq\lda_{j}^{1/2}\|\bm{g}\|_{\dot{\calH}^{2}}\|\bm{g}\|_{\Mor},\quad\forall j\in\{1,\dots,J\}.\label{eq:Morawetz-lmb-rd-lmb-bdd}
\end{equation}
\item (Smallness for modulation vectors) For all $\alp\in(0,\alp^{*}]$
and $(\vec{\iota},\vec{\lda})\in\calP_{J}(\alp)$, we have 
\begin{equation}
|\calM[\vec{\lda};g,\Lda Q_{\ul{;j}}]|\aleq\{\dlt_{0}+o_{\alp\to0}(1)\dlt_{0}^{-J}\}\cdot\lda_{j}^{-1/2}\dlt_{0}^{(j-1)/2}\|\bm{g}\|_{\Mor},\quad\forall j\in\{1,\dots,J\}.\label{eq:Mor-LdaQ}
\end{equation}
\end{enumerate}
\end{prop}

\begin{rem}
The estimate (\ref{eq:Morawetz-monotonicity}) is the key monotonicity
estimate. This will be used to propagate the $\dot{\calH}^{2}$-energy
for the radiation in later modulation analysis. The additional estimates
(\ref{eq:Morawetz-lmb-rd-lmb-bdd}) and (\ref{eq:Mor-LdaQ}) are also
necessary because each $\lda_{j}$ can change in time. 
\end{rem}

\begin{rem}[On $\dlt_{0}$ and $\alp$]
We have two parameters $\dlt_{0}$ and $\alp$ in (\ref{eq:Mor-LdaQ}),
which looks the most technical. Later, this estimate will be applied
when $o_{\alp\to0}(1)$ is much smaller than $\dlt_{0}$, so that
$o_{\alp\to0}(1)\dlt_{0}^{-J}$ can be absorbed into the first term
$\dlt_{0}$. We keep $\dlt_{0}$ in this estimate, and a small $\dlt_{0}$
will finally be chosen in the $\dot{\calH}^{2}$-energy estimate later
(Proposition~\ref{prop:energy-Morawetz}). See Remark~\ref{rem:on-dlt_0}
for more details.
\end{rem}

\begin{proof}
\uline{Proof of \mbox{(\ref{eq:Morawetz-bdd})}}. This easily follows
from (\ref{eq:M0-bdd}): 
\[
|\calM[\vec{\lmb};g,\dot{g}]|\leq\tsum{j=1}J\dlt_{0}^{j-1}|\calM_{0}[\lmb_{j};g,\dot{g}]|\aleq\||g|_{-2}|\dot{g}|_{-1}\|_{L^{1}}\aleq\|\bm{g}\|_{\dot{\calH}^{2}}^{2}.
\]

\uline{Proof of \mbox{(\ref{eq:Morawetz-lmb-rd-lmb-bdd})}}. By
$\lmb_{j}\rd_{\lmb_{j}}\calM[\vec{\lmb};g,\dot{g}]=\dlt_{0}^{j-1}(\lmb\rd_{\lmb}\calM_{0})[\lmb_{j};g,\dot{g}]$
and (\ref{eq:M0-lmb-rd-lmb-bdd}), we get 
\[
|\lmb_{j}\rd_{\lmb_{j}}\calM[\vec{\lmb};g,\dot{g}]|\aleq\dlt_{0}^{j-1}\lda_{j}^{1/2}\|\bm{g}\|_{\dot{\calH}^{2}}\|\bm{g}\|_{\Mor(\lmb_{j})}\aleq\dlt_{0}^{(j-1)/2}\lda_{j}^{1/2}\|\bm{g}\|_{\dot{\calH}^{2}}\|\bm{g}\|_{\Mor}.
\]

\uline{Proof of \mbox{(\ref{eq:Morawetz-monotonicity})}}. We begin
with 
\begin{align}
 & -\{\calM[\vec{\lmb};g,-H_{\calQ(\vec{\iota},\vec{\lda})}g]+\calM[\vec{\lmb};\dot{g},\dot{g}]\}\nonumber \\
 & =-\tsum{i=1}J\dlt_{0}^{i-1}\{\calM_{0}[\lda_{i};g,-H_{\calQ(\vec{\iota},\vec{\lda})}g]+\calM_{0}[\lda_{i};\dot{g},\dot{g}]\}\nonumber \\
 & =-\tsum{i=1}J\dlt_{0}^{i-1}\{\calM_{0}[\lda_{i};g,-H_{\lda_{i}}g]+\calM_{0}[\lda_{i};\dot{g},\dot{g}]\}+\tsum{i=1}J\dlt_{0}^{i-1}\calM_{0}[\lda_{i};g,G_{i}g],\label{eq:Mor-tmp15}
\end{align}
where we recall $G_{i}\coloneqq H_{\calQ(\vec{\iota},\vec{\lda})}-H_{\lda_{i}}$.
Note that (\ref{eq:M0-monotonicity}) and (\ref{eq:mor-orth}) imply
\[
-\sum_{i=1}^{J}\dlt_{0}^{i-1}\{\calM_{0}[\lda_{i};g,-H_{\lda_{i}}g]+\calM_{0}[\lda_{i};\dot{g},\dot{g}]\}\geq c\sum_{i=1}^{J}\dlt_{0}^{i-1}\|\bm{g}\|_{\Mor(\lmb_{i})}^{2}
\]
for some universal constant $c>0$. Therefore, it suffices to show
\begin{equation}
\sum_{i=1}^{J}\dlt_{0}^{i-1}\calM_{0}[\lda_{i};g,G_{i}g]\leq\frac{c}{2}\sum_{i=1}^{J}\dlt_{0}^{i-1}\|\bm{g}\|_{\Mor(\lmb_{i})}^{2}.\label{eq:3.19}
\end{equation}

We claim that 
\begin{align}
|\calM_{0}[\lda_{i};g,G_{i}g]| & \aleq\int\sum_{j;j\neq i}\Lda Q_{\lda_{j}}\hat{\psi}_{\lda_{i}}\frac{|g|_{-1}^{2}}{r^{3}},\quad\text{where}\quad\hat{\psi}(y)\coloneqq\frac{y}{1+y}.\label{eq:mor-tmp3}
\end{align}
To see this, note that 
\begin{equation}
\calM_{0}[\lda_{i};g,G_{i}g]=\lan A_{\lda_{i}}(G_{i}g),\Lda_{\psi,\lda_{i}}A_{\lda_{i}}g\ran-\dlt\lan A_{\lda_{i}}(G_{i}g),\lda_{i}(r+\lda_{i})^{-2}A_{\lda_{i}}g\ran.\label{eq:Morawetz-error}
\end{equation}
We manipulate the first term of (\ref{eq:Morawetz-error}) using the
anti-symmetry of $\Lda_{\psi,\lda_{i}}$ and $[\Lda_{\psi,\lda_{i}},G_{i}]=\psi_{\lda_{i}}\rnd_{r}G_{i}$:
\begin{align*}
\lan A_{\lda_{i}}(G_{i}g),\Lda_{\psi,\lda_{i}}A_{\lda_{i}}g\ran & =\lan G_{i}A_{\lda_{i}}g,\Lda_{\psi,\lda_{i}}A_{\lda_{i}}g\ran-\lan(\rnd_{r}G_{i})g,\Lda_{\psi,\lda_{i}}A_{\lda_{i}}g\ran\\
 & =-\tfrac{1}{2}\lan[\Lda_{\psi,\lda_{i}},G_{i}]A_{\lda_{i}}g,A_{\lda_{i}}g\ran+\lan\Lda_{\psi,\lda_{i}}(\rnd_{r}G_{i}\cdot g),A_{\lda_{i}}g\ran\\
 & =-\tfrac{1}{2}\tint{}{}\psi_{\lda_{i}}\rnd_{r}G_{i}(A_{\lda_{i}}g)^{2}+\lan\Lda_{\psi,\lda_{i}}(\rnd_{r}G_{i}\cdot g),A_{\lda_{i}}g\ran.
\end{align*}
Inserting this into (\ref{eq:Morawetz-error}) and using $\psi(y)\aeq\frac{y}{1+y}=\hat{\psi}(y)$
(\ref{eq:psi-bound}), (\ref{eq:Lda_psi}), and $\lda_{i}(r+\lda_{i})^{-2}\leq r^{-1}\hat{\psi}_{\lda_{i}}$,
we get 
\begin{align*}
|\calM_{0}[\lda_{i};g,G_{i}g]| & \aleq\int\hat{\psi}_{\lda_{i}}|\rnd_{r}G_{i}||A_{\lda_{i}}g|^{2}+\hat{\psi}_{\lda_{i}}|\rnd_{r}G_{i}\cdot g|_{-1}|A_{\lda_{i}}g|+r^{-1}\hat{\psi}_{\lda_{i}}|A_{\lda_{i}}(G_{i}g)||A_{\lda_{i}}g|\\
 & \aleq\int\hat{\psi}_{\lda_{i}}\{|\rnd_{r}G_{i}||g|_{-1}^{2}+|\rnd_{r}G_{i}\cdot g|_{-1}|g|_{-1}+r^{-1}|G_{i}g|_{-1}|g|_{-1}\}.
\end{align*}
Applying the bound (\ref{eq:pointwise-G_i}) for $G_{i}$ yields (\ref{eq:mor-tmp3}).

Having established (\ref{eq:mor-tmp3}), we turn to finish the proof
of (\ref{eq:3.19}). By (\ref{eq:mor-tmp3}), we have 
\[
\sum_{i}\dlt_{0}^{i-1}|\calM_{0}[\lda_{i};g,G_{i}g]|\leq\td C\sum_{i}\dlt_{0}^{i-1}\int\sum_{j:j\neq i}\Lmb Q_{\lmb_{j}}\hat{\psi}_{\lmb_{i}}\frac{|g|_{-1}^{2}}{r^{3}}
\]
for some universal constant $\td C>0$. In view of $\|\bm{g}\|_{\Mor}^{2}\geq\int\sum_{i}\dlt_{0}^{i-1}\hat{\psi}_{\lda_{i}}\frac{|g|_{-1}^{2}}{r^{3}}$,
the proof of (\ref{eq:3.19}) reduces to showing that 
\begin{equation}
\sum_{i}\dlt_{0}^{i-1}\sum_{j:j\neq i}\Lmb Q_{\lmb_{j}}\hat{\psi}_{\lmb_{i}}\leq\frac{c}{2\td C}\sum_{i}\dlt_{0}^{i-1}\hat{\psi}_{\lmb_{i}}.\label{eq:3.22}
\end{equation}
We separate into regions of $r$ to prove this. 

First, consider $r\in(\bigcup_{m=1}^{J}[R^{-1}\lmb_{m},R\lmb_{m}])^{c}$
for a large universal constant $R>1$ to be chosen. In this region,
$\sum_{j}\Lmb Q_{\lmb_{j}}\aleq R^{-k}$ so 
\[
\tsum i{}\dlt_{0}^{i-1}\tsum{j:j\neq i}{}\Lmb Q_{\lmb_{j}}\hat{\psi}_{\lmb_{i}}\aleq R^{-k}\tsum i{}\dlt_{0}^{k-1}\hat{\psi}_{\lmb_{i}}.
\]
Fixing $R>1$ sufficiently large, we get (\ref{eq:3.22}) in the region
$(\bigcup_{m=1}^{J}[R^{-1}\lmb_{m},R\lmb_{m}])^{c}$. 

On the other hand, suppose $r\in[R^{-1}\lmb_{m},R\lmb_{m}]$ for some
$m\in\{1,\dots,J\}$. Since $R>1$ is fixed in the previous paragraph,
we can ignore any $R$-dependence from now on. In this region, observe
that RHS(\ref{eq:3.22}) has a lower bound 
\begin{equation}
\tsum i{}\dlt_{0}^{i-1}\hat{\psi}_{\lmb_{i}}\ageq\dlt_{0}^{m-1}\hat{\psi}_{\lmb_{m}}\ageq\dlt_{0}^{m-1}.\label{eq:3.23}
\end{equation}
For LHS(\ref{eq:3.22}), we separate into three cases: $i<m$, $i>m$,
and $i=m$. When $i<m$, we use rough bounds $\Lmb Q_{\lmb_{j}}\aleq1$
and $\hat{\psi}_{\lmb_{i}}\aleq1$ to have 
\[
\tsum{i:i>m}{}\dlt_{0}^{i-1}\tsum{j:j\neq i}{}\Lmb Q_{\lmb_{j}}\hat{\psi}_{\lmb_{i}}\aleq\tsum{i:i>m}{}\dlt_{0}^{i-1}\aleq\dlt_{0}\cdot\dlt_{0}^{m-1}.
\]
When $i>m$, we use $\Lmb Q_{\lmb_{j}}\aleq1$ and $\hat{\psi}_{\lmb_{i}}\aeq r/\lmb_{i}\aleq\lmb_{m}/\lmb_{i}\aleq(\alp^{\ast})^{m-i}$
to have 
\[
\tsum{i:i<m}{}\dlt_{0}^{i-1}\tsum{j:j\neq i}{}\Lmb Q_{\lmb_{j}}\hat{\psi}_{\lmb_{i}}\aleq\tsum{i:i<m}{}\dlt_{0}^{i-1}(\alp^{\ast})^{m-i}=\tsum{i:i<m}{}(\dlt_{0}^{-1}\alp^{\ast})^{m-i}\cdot\dlt_{0}^{m-1}.
\]
Finally, when $i=m$, we use $\hat{\psi}_{\lmb_{m}}\aeq1$ and $\Lmb Q(y)\aeq\chf_{y\leq1}y^{k}+\chf_{y\geq1}y^{-k}$
to have 
\[
\dlt_{0}^{m-1}\tsum{j:j\neq m}{}\Lmb Q_{\lmb_{j}}\hat{\psi}_{\lmb_{m}}\aleq\dlt_{0}^{m-1}\{\tsum{j:j<m}{}(\lmb_{m}/\lmb_{j})^{k}+\tsum{j:j>m}{}(\lmb_{m}/\lmb_{j})^{-k}\}\aleq(\alp^{\ast})^{k}\cdot\dlt_{0}^{m-1}.
\]
By the parameter dependence $\alp^{\ast}\ll\dlt_{0}\ll1$, all of
the previous three displays are much smaller than $\dlt_{0}^{m-1}$.
In view of (\ref{eq:3.23}), we have proved (\ref{eq:3.22}) in the
region $[R^{-1}\lmb_{m},R\lmb_{m}]$. This completes the proof of
(\ref{eq:3.19}) and hence (\ref{eq:Morawetz-monotonicity}).

\uline{Proof of \mbox{(\ref{eq:Mor-LdaQ})}}. By the cancellation
(\ref{eq:M0-bubble-can}) and a variant of (\ref{eq:M0-bdd}) (obtained
by moving $\Lmb_{\psi,\lmb}$ to the other side), we have 
\begin{align*}
|\calM[\vec{\lmb};g,\Lmb Q_{\ul{;j}}]| & \leq\sum_{i:i\neq j}\dlt_{0}^{i-1}|\calM_{0}[\lmb_{i};g,\Lda Q_{\ubr{;j}}]|\aleq\sum_{i:i\neq j}\lda_{j}^{-1}\cdot\dlt_{0}^{i-1}\Big\|\frac{y_{i}}{1+y_{i}}|\Lda Q_{\lda_{j}}|_{-2}|g|_{-1}\Big\|_{L^{1}}\\
 & \aleq\lmb_{j}^{-1/2}\|\bm{g}\|_{\Mor(\lda_{j})}\sum_{i:i\neq j}\dlt_{0}^{i-1}\Big\|\frac{y_{i}}{1+y_{i}}\sqrt{y_{j}+1}r^{-1}\Lda Q_{\lda_{j}}\Big\|_{L^{2}}.
\end{align*}
For $i>j$, we use $\dlt_{0}^{i-1}\leq\dlt_{0}\cdot\dlt_{0}^{j-1}$
and $\frac{y_{i}}{1+y_{i}}\leq1$ to have 
\[
\dlt_{0}^{i-1}\Big\|\frac{y_{i}}{1+y_{i}}\sqrt{y_{j}+1}r^{-1}\Lda Q_{\lda_{j}}\Big\|_{L^{2}}\leq\dlt_{0}\cdot\dlt_{0}^{j-1}\|\sqrt{y_{j}+1}r^{-1}\Lda Q_{\lda_{j}}\|_{L^{2}}\aleq\dlt_{0}\cdot\dlt_{0}^{j-1}.
\]
For $i<j$, we use $\frac{y_{i}}{1+y_{i}}\leq y_{i}=\frac{\lda_{j}}{\lda_{i}}y_{j}$
and $\lmb_{j}\leq\alp^{j-i}\lmb_{i}$ to get 
\[
\dlt_{0}^{i-1}\Big\|\frac{y_{i}}{1+y_{i}}\sqrt{y_{j}+1}r^{-1}\Lda Q_{\lda_{j}}\Big\|_{L^{2}}\leq\dlt_{0}^{i-1}\frac{\lda_{j}}{\lda_{i}}\|y_{j}\sqrt{y_{j}+1}r^{-1}\Lda Q_{\lda_{j}}\|_{L^{2}}\aleq\dlt_{0}^{i-1}\frac{\lda_{j}}{\lda_{i}}\aleq(\dlt_{0}^{-1}\alp)^{j-i}\cdot\dlt_{0}^{j-1}.
\]
Combining the previous three displays and using $(\dlt_{0}^{-1}\alp)^{j-i}=o_{\alp\to0}(1)\dlt_{0}^{-J}$
and $\dlt_{0}^{j-1}\|\bm{g}\|_{\Mor(\lda_{j})}\aleq\dlt_{0}^{(j-1)/2}\|\bm{g}\|_{\Mor}$,
we get (\ref{eq:Mor-LdaQ}). This completes  the proof of Proposition~\ref{prop:Morawetz}. 
\end{proof}

\section{\label{sec:Modulation-analytic-setup}Modulation analytic setup for
construction}

In this section, we introduce the modulation analytic setup and main
propositions for the construction of bubble tower solutions. From
now on, 
\[
\iota_{j}=(-1)^{j-1}\qquad\forall j\in\{1,\dots,J\}.
\]

This section and the next section are devoted to the first step mentioned
at the beginning of Section~\ref{subsec:Strategy-of-the-proof}:
to construct a family of solutions $\bm{u}^{t_{0}}$ on $[t_{0},T_{boot}]$
for all $t_{0}\leq T_{boot}$ ($T_{boot}$ is some fixed time independent
of $t_{0}$) that exhibit backward-in-time bubbling as in Theorem~\ref{thm:main-negative-time}
on $[t_{0},T_{boot}]$. To achieve this, we first prescribe a set
of initial data that are close to the desired bubble tower solution
at time $t_{0}$. Next, we set up appropriate bootstrap assumptions
on the parameters of solutions that quantify closeness to the bubble
tower behavior. However, there are some assumptions that cannot be
bootstrapped (or propagated) forward in time; this will be treated
by a shooting argument instead. 

In Section~\ref{subsec:Linearization-of-formal-ODE}, we identify
the parameters requiring a shooting argument by analyzing the linearized
formal ODE system of parameters. In Section~\ref{subsec:Family-of-initial-data},
we introduce the set of initial data at $t_{0}$ and provide a more
rigorous statement of the first step described above. 

\subsection{\label{subsec:Linearization-of-formal-ODE}Linearization of formal
ODE}

In order to identify stable/unstable modulation parameters and to
motivate our choice of initial data, it is instructive to linearize
the formal ODE system (\ref{eq:intro-formal-ODE}) around the exact
solution (\ref{eq:intro-def-lmb-b-exact}). After ignoring nonlinear
terms, for $j\in\{2,\dots,J\}$ ($\lmb_{1}$ and $b_{1}$ will be
treated separately), we obtain 
\begin{align*}
(\lmb_{j}-\lmb_{j}^{\ex})_{t} & \approx-(b_{j}-b_{j}^{\ex}),\\
-(b_{j}-b_{j}^{\ex})_{t} & \approx8k^{2}\kpp^{-1}\Big\{(k-1)\frac{(\lmb_{j}^{\ex})^{k-1}}{(\lmb_{j-1}^{\ex})^{k}}\frac{\lmb_{j}-\lmb_{j}^{\ex}}{\lmb_{j}^{\ex}}-k\frac{(\lmb_{j}^{\ex})^{k-1}}{(\lmb_{j-1}^{\ex})^{k}}\frac{\lmb_{j-1}-\lmb_{j-1}^{\ex}}{\lmb_{j-1}^{\ex}}\Big\}.
\end{align*}
Introducing the variables 
\begin{equation}
\nu_{j}=\frac{\lmb_{j}-\lmb_{j}^{\ex}}{\lmb_{j}^{\ex}}\quad\text{and}\quad\dot{\nu}_{j}=\frac{b_{j}-b_{j}^{\ex}}{b_{j}^{\ex}}\label{eq:def-nu-and-nu-dot}
\end{equation}
and using the identity 
\[
-8k^{2}\kpp^{-1}\frac{(\lmb_{j}^{\ex})^{k-1}}{(\lmb_{j-1}^{\ex})^{k}b_{j}^{\ex}}=\frac{\alp_{j}+1}{|t|},
\]
we obtain the linearized formal ODE system ($j\geq2$)
\begin{equation}
\left|\begin{aligned}\nu_{j,t} & =\frac{1}{|t|}\Big\{-\alp_{j}\nu_{j}+\alp_{j}\dot{\nu}_{j}\Big\},\\
\dot{\nu}_{j,t} & =\frac{1}{|t|}\Big\{(k-1)(\alp_{j}+1)\nu_{j}-(\alp_{j}+1)\dot{\nu}_{j}-k(\alp_{j}+1)\nu_{j-1}\Big\}.
\end{aligned}
\right.\label{eq:linearized-formal-ODE}
\end{equation}

From now on, we ignore $\nu_{j-1}$ in (\ref{eq:linearized-formal-ODE});
this can be justified by assuming $|\nu_{j-1}|\aleq|t|^{-\eps_{j-1}}$
with $\eps_{j-1}>\eps_{j}$ in later modulation analysis. Then, the
ODE system is further decoupled and associated with the matrix 
\begin{equation}
A_{j}\coloneqq\begin{pmatrix}-\alp_{j} & \alp_{j}\\
(k-1)(\alp_{j}+1) & -(\alp_{j}+1)
\end{pmatrix}.\label{eq:def-A_j}
\end{equation}
Note that $\det A_{j}=-(k-2)\alp_{j}(\alp_{j}+1)<0$, so $A_{j}$
has one positive and one negative eigenvalues 
\begin{equation}
\sigma_{j,\pm}=-(\alp_{j}+\tfrac{1}{2})\pm\sqrt{(\alp_{j}+\tfrac{1}{2})^{2}+(k-1)\alp_{j}(\alp_{j}+1)}.\label{eq:def-sigma_j,pm}
\end{equation}
In fact, $A_{j}$ can be diagonalized as
\begin{equation}
A_{j}=P\begin{pmatrix}\sigma_{j,+} & 0\\
0 & \sigma_{j,-}
\end{pmatrix}P^{-1},\qquad P=\begin{pmatrix}1 & 1\\
1+\frac{\sigma_{j,+}}{\alp_{j}} & 1+\frac{\sigma_{j,-}}{\alp_{j}}
\end{pmatrix}.\label{eq:A_j-diagonalization}
\end{equation}
In particular, 
\begin{align}
\Big\{(1+\frac{\sigma_{j,-}}{\alp_{j}})\nu_{j}-\dot{\nu}_{j}\Big\}_{t} & =\frac{\sigma_{j,+}}{|t|}\Big\{(1+\frac{\sigma_{j,-}}{\alp_{j}})\nu_{j}-\dot{\nu}_{j}\Big\},\label{eq:formal-unstable-mode}\\
\Big\{-(1+\frac{\sigma_{j,+}}{\alp_{j}})\nu_{j}+\dot{\nu}_{j}\Big\}_{t} & =\frac{\sigma_{j,-}}{|t|}\Big\{-(1+\frac{\sigma_{j,+}}{\alp_{j}})\nu_{j}+\dot{\nu}_{j}\Big\}.\label{eq:formal-stable-mode}
\end{align}
Since $\sigma_{j,+}>0$ and $\sigma_{j,-}<0$, the first corresponds
to the unstable direction (which requires shooting) and the second
is stable. In fact, we will apply the shooting argument to $\nu_{j}$.

\subsection{\label{subsec:Family-of-initial-data}Family of initial data and
statements of main propositions}

Motivated from the previous linearized dynamics, we are now ready
to define the family of initial data, where we will apply the shooting
argument. For $t_{0}\in(-\infty,-1)$ and $(\eps_{j})_{2\leq j\leq J}\in(0,1)^{J-1}$,
consider the initial data set 
\begin{equation}
O_{init}^{t_{0}}\define\{\bm{u}(t_{0})=\begin{pmatrix}\calQ(\vec{\iota},\vec{\lda}_{0})\\
\sum_{j}b_{0,j}\iota_{j}\Lda Q_{\ubr{\lda_{0,j}}}
\end{pmatrix}\colon\lda_{0,1}=1,\;b_{0,1}=0,\;(\lda_{0,j})_{2\leq j\leq J}\in\calV^{t_{0}}\},\label{eq:O_init}
\end{equation}
where 
\begin{equation}
\calV^{t_{0}}\define\Big\{(\lda_{0,j})_{2\leq j\leq J}\in(0,\infty)^{J-1}\colon\Big|\frac{\lmb_{0,j}}{\lmb_{j}^{\ex}(t_{0})}-1\Big|\leq|t_{0}|^{-\eps_{j}}\Big\}\label{eq:def-calV}
\end{equation}
and $(b_{0,j})_{2\leq j\leq J}\in\bbR^{J-1}$ are determined by $(\lda_{0,j})_{2\leq j\leq J}$
through the relation (cf.~(\ref{eq:formal-stable-mode})) 
\begin{equation}
\frac{b_{0,j}}{b_{j}^{\ex}(t_{0})}-1=-(1+\frac{\sigma_{j,+}}{\alp_{j}})\Big(\frac{\lmb_{0,j}}{\lmb_{j}^{\ex}(t_{0})}-1\Big).\label{eq:initi-cond}
\end{equation}
In fact, the constants $\eps_{j}$s will be chosen in the proof of
Proposition~\ref{prop:Bootstrap} (see Remark~\ref{rem:constants_eps_j}
for a more concrete choice) and $t_{0}$ will be any large negative
time.

We state a more precise version of the first step described at the
beginning of Section~\ref{sec:Modulation-analytic-setup}. 
\begin{prop}[Solutions $\bm{u}^{t_{0}}$]
\label{prop:solutions_u^t_0}There exists $T_{boot}\in(-\infty,-1)$
such that for all $t_{0}\leq T_{boot}$, there exists an initial data
$\bm{u}^{t_{0}}(t_{0})\in O_{init}^{t_{0}}$ whose forward-in-time
evolution $\bm{u}^{t_{0}}$ is defined on $[t_{0},T_{boot}]$ and
can be decomposed as (\ref{eq:decomp-u}) satisfying the following
controls on $[t_{0},T_{boot}]$
\begin{equation}
\begin{aligned}\|\bm{g}\|_{\dot{\calH}^{1}} & \leq|t|^{-1}, & |\lda_{1}-1| & \leq|t|^{-\eps_{1}}, & |b_{1}| & \leq|t|^{-1},\\
\|\bm{g}\|_{\dot{\calH}^{2}} & \leq|t|^{-1-\eps_{0}}, & |\nu_{j}| & \leq|t|^{-\eps_{j}}, & |\dot{\nu}_{j}| & \leq1,
\end{aligned}
\label{eq:boot-assump}
\end{equation}
where $j\in\{2,\dots,J\}$, for some $\eps_{0},\eps_{1},\dots,\eps_{J}\in(0,1)$
independent of $t_{0}$.
\end{prop}

\begin{rem}
\label{rem:constants_eps_j}In fact, we can fix the constants $\eps_{0},\eps_{1},\dots,\eps_{J}\in(0,1)$
to satisfy 
\[
\eps_{0}>\eps_{1}>\dots>\eps_{J},\quad\eps_{0}<\frac{1}{2},\quad\eps_{1}<\frac{2}{k-2},\quad\eps_{j}<\sigma_{j,+}\text{ for }j\in\{2,\dots,J\}.
\]
These conditions will be used in the later analysis.
\end{rem}

In Proposition~\ref{prop:solutions_u^t_0}, all the controls except
$|\nu_{j}|\leq|t|^{-\eps_{j}}$, $j\geq2$, can be bootstrapped. We
will later apply a shooting argument to choose a special initial data
$\bm{u}^{t_{0}}(t_{0})$ in $O_{init}^{t_{0}}$ such that $|\nu_{j}|\leq|t|^{-\eps_{j}}$
holds. We conclude this subsection with the statement of the bootstrap
proposition. 
\begin{prop}[Main bootstrap]
\label{prop:Bootstrap}There exist $\eps_{0},\eps_{1},\dots,\eps_{J}\in(0,1)$
and $T_{boot}\in(-\infty,-1)$ with the following property. 

Let $[t_{0},T]\subset(-\infty,T_{boot}]$ be a time interval and assume
$\bm{u}:[t_{0},T]\to\calT_{J}(\eta_{0})$ is a solution to (\ref{eq:k-equiv-WM})
with $\bm{u}(t_{0})\in O_{init}^{t_{0}}$. Decompose $\bm{u}$ on
$[t_{0},T]$ according to Lemma~\ref{lem:decomposition}. If the
following bootstrap assumptions hold on $[t_{0},T]$
\[
\begin{aligned}\|\bm{g}\|_{\dot{\calH}^{1}} & \leq|t|^{-1}, & |\lda_{1}-1| & \leq|t|^{-\eps_{1}}, & |b_{1}| & \leq|t|^{-1},\\
\|\bm{g}\|_{\dot{\calH}^{2}} & \leq|t|^{-1-\eps_{0}}, & |\nu_{j}| & \leq|t|^{-\eps_{j}}, & |\dot{\nu}_{j}| & \leq1,
\end{aligned}
\tag{\ref{eq:boot-assump}}
\]
where $j\in\{2,\dots,J\}$, then the following estimates hold on $[t_{0},T]$:
\begin{enumerate}
\item (Control of stable modes) For $j\in\{2,\dots,J\}$, 
\begin{equation}
\begin{aligned}\|\bm{g}\|_{\dot{\calH}^{1}} & \leq\tfrac{1}{2}|t|^{-1}, & |\lda_{1}-1| & \leq\tfrac{1}{2}|t|^{-\eps_{1}}, & |b_{1}| & \leq\tfrac{1}{2}|t|^{-1},\\
\|\bm{g}\|_{\dot{\calH}^{2}} & \leq\tfrac{1}{2}|t|^{-1-\eps_{0}}, &  &  & |\dot{\nu}_{j}| & \leq\tfrac{1}{2}.
\end{aligned}
\label{eq:boot-concl}
\end{equation}
\item (Shooting property for $\lda_{j}$ with $j\geq2$) For $j\in\{2,\dots,J\}$,
if $|\nu_{j}(T)|=|T|^{-\eps_{j}}$, then 
\begin{equation}
\frac{d}{dt}(|t|^{\eps_{j}}|\nu_{j}(t)|)\Big|_{t=T}>0.\label{eq:boot-shooting}
\end{equation}
\end{enumerate}
In addition, we have ${\rm dist}(\bm{u}(t),K^{t_{0}})\leq\frac{1}{2}\eta$
for $K^{t_{0}}\define\{\bm{\calQ}(\vec{\iota},\vec{\lda})\colon\frac{1}{2}\lmb_{J}^{\ex}(t_{0})\leq\lmb_{j}\leq2,\ \forall j\geq1\}$,
where $\eta$ is the constant in Lemma~\ref{lem:App-escape-cpt}. 
\end{prop}

\begin{rem}[Modulation parameters at time $t_{0}$]
\label{rem:on-init-cond}By the definition (\ref{eq:O_init}), $O_{init}^{t_{0}}\subset\calT_{J}(\frac{1}{2}\eta_{0})$
for all large negative $t_{0}$ (or, for all $t_{0}\leq T_{boot}$
if we view $|T_{boot}|$ as a large parameter). In particular, $\bm{u}(t_{0})\in O_{init}^{t_{0}}$
can be decomposed according to Lemma~\ref{lem:decomposition} and
the uniqueness statement therein implies 
\begin{equation}
\bm{g}(t_{0})=0,\quad\vec{\lda}(t_{0})=\vec{\lda}_{0},\quad\text{and}\quad\vec{b}(t_{0})=\vec{b}_{0}.\label{eq:4.13}
\end{equation}
In view of (\ref{eq:def-calV}) and (\ref{eq:initi-cond}), $\bm{u}(t)$
also satisfies the bootstrap hypotheses (\ref{eq:boot-assump}) at
$t=t_{0}$. 
\end{rem}

\subsection{Proof of Proposition~\ref{prop:solutions_u^t_0} assuming Proposition~\ref{prop:Bootstrap}}
\begin{proof}
This follows from a standard shooting argument. Let $T_{boot}$ be
as in Proposition~\ref{prop:Bootstrap} and let $t_{0}\leq T_{boot}$
be arbitrary. Let $\bm{u}(t_{0})$ be any initial data in (\ref{eq:O_init})
at time $t_{0}$ and $\bm{u}$ be its forward-in-time evolution with
maximal lifespan $[t_{0},T_{+}(\bm{u}))$. Define the time 
\[
T\define\sup\big\{\tau\in[t_{0},T_{+}(\bm{u}))\colon\bm{u}(t)\in\calT_{J}(\tfrac{1}{2}\eta_{0})\text{ and }\eqref{eq:boot-assump}\text{ holds for all }t\in[t_{0},\tau]\big\},
\]
where $\bm{u}(t)$ is decomposed according to Lemma~\ref{lem:decomposition}.
Note that $T>t_{0}$ by Remark~\ref{rem:on-init-cond}. 

We need to prove that there exists some special initial data $\bm{u}^{t_{0}}(t_{0})$
such that $T>T_{boot}$. Suppose not, i.e., 
\[
T\leq T_{boot}
\]
for all initial data $\bm{u}(t_{0})$ in (\ref{eq:O_init}). We will
derive a contradiction using the Brouwer argument. 

First, we claim that $T<T_{+}(\bm{u})$. Suppose not, i.e., $T_{+}(\bm{u})=T$.
By the definition of $T$ and $T=T_{+}(\bm{u})$, ${\rm dist}(\bm{u}(t),K^{t_{0}})\leq\frac{1}{2}\eta$
for all $t\in[t_{0},T_{+}(\bm{u}))$ for a compact set $K^{t_{0}}=\{\bm{\calQ}(\vec{\iota},\vec{\lda})\colon\frac{1}{2}\lda_{J}^{\ex}(t_{0})\leq\lda_{j}\leq2\}$.
On the other hand, since $T_{+}(\bm{u})<+\infty$ (due to $T_{+}(\bm{u})=T\leq T_{boot}<+\infty$),
Lemma~\ref{lem:App-escape-cpt} implies that ${\rm dist}(\bm{u}(t),K^{t_{0}})>\frac{1}{2}\eta$
for some $t\in[t_{0},T_{+}(\bm{u}))$, which is a contradiction.

By $T<T_{+}(\bm{u})$ and the bootstrap Proposition~\ref{prop:Bootstrap},
there exists $j\in\{2,\dots,J\}$ such that 
\begin{equation}
|\nu_{j}(T)|=|T|^{-\eps_{j}}\label{eq:shooting-tmp1}
\end{equation}
and the shooting property~(\ref{eq:boot-shooting}) holds. Note that
(\ref{eq:boot-shooting}) also implies that $T$ is the first time
$t\in[t_{0},T]$ such that (\ref{eq:shooting-tmp1}) holds for some
$j\in\{2,\dots,J\}$. 

By (\ref{eq:shooting-tmp1}), the following map is well-defined: 
\[
\Phi:(\lda_{0,j})_{2\leq j\leq J}\in\calV^{t_{0}}\mapsto\bm{u}(t_{0})\mapsto(|T|^{\eps_{j}}\nu_{j}(T))_{2\leq j\leq J}\in\rnd[-1,1]^{J-1}.
\]
First, we note that $\calV^{t_{0}}$ is homeomorphic to $[-1,1]^{J-1}$
via $(\lda_{0,j})_{2\leq j\leq J}\mapsto(|t_{0}|^{\eps_{j}}(\frac{\lmb_{0,j}}{\lmb_{j}^{\ex}(t_{0})}-1))_{2\leq j\leq J}$.
Under this identification of $\calV^{t_{0}}$ and $[-1,1]^{J-1}$,
(\ref{eq:4.13}) implies that $\Phi$ is the identity map on the boundary
$\rnd[-1,1]^{J-1}$. Next, we claim that $\Phi$ is continuous. It
suffices to show that $\bm{u}(t_{0})\mapsto(|T|^{\eps_{j}}\nu_{j}(T))_{2\leq j\leq J}$
is continuous. Note that $(\bm{u}(t_{0}),t)\mapsto(|t|^{\eps_{j}}\nu_{j}(t))_{2\leq j\leq J}$
is continuous by the local well-posedness and continuous dependence
of modulation parameters. The shooting property~(\ref{eq:boot-shooting})
ensures that $\bm{u}(t_{0})\mapsto T$ is continuous, so $\bm{u}(t_{0})\mapsto(|T|^{\eps_{j}}\nu_{j}(T))_{2\leq j\leq J}$
is continuous as desired. 

By the previous paragraph, the map $\Phi:[-1,1]^{J-1}\to\rnd[-1,1]^{J-1}$
(note the identification of $\calV^{t_{0}}$ and $[-1,1]^{J-1}$ above)
is continuous and equal to the identity map on the boundary. Such
a map cannot exist by the Brouwer fixed point theorem, so we get a
contradiction. Therefore, $T>T_{boot}$ for some special initial data
$\bm{u}^{t_{0}}(t_{0})\in O_{init}^{t_{0}}$. This completes the proof.
\end{proof}

\section{\label{sec:Proof-of-bootstrap}Proof of bootstrap proposition}

This section is devoted to the proof of Proposition~\ref{prop:Bootstrap}.
We assume $\bm{u}(t)$ is a solution to (\ref{eq:k-equiv-WM}) as
in the assumption of Proposition~\ref{prop:Bootstrap}. $T_{boot}$
is a large negative time ($T_{boot}\ll-1$), where $|T_{boot}|$ can
enlarge in the course of the proof. 

\subsection{Technical estimates}
\begin{lem}
\label{lem:pre-cal}The following estimates hold.
\begin{enumerate}
\item For $j\in\{1,\dots,J\}$, 
\begin{equation}
\lda_{j}\aeq|t|^{-\alp_{j}}\aleq1\quad\text{and}\quad|b_{j}|\aleq|t|^{-\alp_{j}-1}\aeq\lda_{j}|t|^{-1}.\label{eq:pre-lda-b}
\end{equation}
\item For $j\in\{2,\dots,J\}$, 
\begin{equation}
\Big(\frac{\lda_{j}}{\lda_{j-1}}\Big)^{k/2}\aeq\lda_{j}|t|^{-1}\aleq|t|^{-1-\frac{2}{k-2}}\quad\text{and}\quad\frac{1}{\lmb_{j}}\Big(\frac{\lda_{j}}{\lda_{j-1}}\Big)^{k}\aeq\lda_{j}|t|^{-2}.\label{eq:pre-ratio-b-b_t}
\end{equation}
\item (Mapping properties) For $i\in\{1,\dots,J\}$ and $\ell\in\bbR$ with
$-k+1<\ell<k-1$, we have 
\begin{equation}
\sum_{j:j\neq i}\|\Lmb Q_{\ul{;i}}\Lmb Q_{\ul{;j}}\|_{L^{1}}\cdot\lda_{j}^{\ell}\aleq\Big\{\chf_{i\leq J-1}\Big(\frac{\lda_{i+1}}{\lda_{i}}\Big)^{k-1+\ell}+\chf_{i\geq2}\Big(\frac{\lda_{i}}{\lda_{i-1}}\Big)^{k-1-\ell}\Big\}\cdot\lmb_{i}^{\ell}.\label{eq:Pre-mapping-prop}
\end{equation}
For any $\dlt_{0}\in(0,1)$, we also have 
\begin{equation}
\sum_{j:j\neq i}\|\Lmb Q_{\ul{;i}}\Lmb Q_{\ul{;j}}\|_{L^{1}}\cdot\dlt_{0}^{-(j-1)/2}\lda_{j}^{1/2}=o(1)\cdot\dlt_{0}^{-(i-1)/2}\lda_{i}^{1/2}.\label{eq:pre-mapping-dlt}
\end{equation}
\end{enumerate}
\end{lem}

\begin{proof}
\uline{Proof of \mbox{(\ref{eq:pre-lda-b})}}. This is clear from
the bootstrap hypotheses (\ref{eq:boot-assump}) and the definitions
of $\nu_{j}$ and $\dot{\nu}_{j}$ (see (\ref{eq:def-nu-and-nu-dot})
and (\ref{eq:intro-def-lmb-b-exact})).

\uline{Proof of \mbox{(\ref{eq:pre-ratio-b-b_t})}}. By (\ref{eq:pre-lda-b})
and $\frac{k}{2}(\alp_{j}-\alp_{j-1})=\alp_{j}+1$ from (\ref{eq:def-alp_j}),
we get 
\[
\Big(\frac{\lda_{j}}{\lda_{j-1}}\Big)^{k/2}\aeq|t|^{-\alp_{j}-1}\aeq\lda_{j}|t|^{-1}\quad\text{and}\quad\frac{1}{\lda_{j}}\Big(\frac{\lda_{j}}{\lda_{j-1}}\Big)^{k}\aeq\lda_{j}|t|^{-2}.
\]

\uline{Proof of \mbox{(\ref{eq:Pre-mapping-prop})}}. This follows
from (\ref{eq:bubble-int}): 
\[
\sum_{j:j\neq i}\|\Lmb Q_{\ul{;i}}\Lmb Q_{\ul{;j}}\|_{L^{1}}\cdot\lmb_{j}^{\ell}\aleq\Big\{\sum_{j:j>i}\Big(\frac{\lmb_{j}}{\lmb_{i}}\Big)^{k-1+\ell}+\sum_{j:j<i}\Big(\frac{\lmb_{i}}{\lmb_{j}}\Big)^{k-1-\ell}\Big\}\cdot\lmb_{i}^{\ell}.
\]

\uline{Proof of \mbox{(\ref{eq:pre-mapping-dlt})}}. By (\ref{eq:bubble-int}),
we have 
\begin{align*}
 & \sum_{j:j\neq i}\|\Lmb Q_{\ul{;i}}\Lmb Q_{\ul{;j}}\|_{L^{1}}\cdot\dlt_{0}^{-(j-1)/2}\lda_{j}^{1/2}\\
 & \quad\aleq\Big\{\sum_{j:j>i}\Big(\frac{\lmb_{j}}{\lmb_{i}}\Big)^{k-1/2}\dlt_{0}^{-(j-i)/2}+\sum_{j:j<i}\Big(\frac{\lmb_{i}}{\lmb_{j}}\Big)^{k-3/2}\dlt_{0}^{(i-j)/2}\Big\}\cdot\dlt_{0}^{-(i-1)/2}\lmb_{i}^{1/2}.
\end{align*}
Applying $(\frac{\lmb_{j}}{\lmb_{j-1}})^{k-1/2}\dlt_{0}^{-1/2}=o(1)$
and $(\frac{\lmb_{j}}{\lmb_{j-1}})^{k-3/2}\dlt_{0}^{1/2}=o(1)$, we
get (\ref{eq:pre-mapping-dlt}).
\end{proof}
\begin{lem}
We have 
\begin{align}
\||f_{\mathbf{i}}(\vec{\iota},\vec{\lda})|_{-1}\|_{L^{2}} & \aleq|t|^{-2},\label{eq:dotH1-interaction}\\
\|r^{m}H_{\calQ}\Lda Q_{\ubr{;i}}\|_{L^{2}} & \aleq\chf_{m=0}|t|^{-2}+\chf_{m=1}\lda_{i}|t|^{-2}(\log|t|)^{1/2},\qquad\forall m\in\{0,1\}.\label{eq:H2-HQbubble}
\end{align}
We also have for $\lda_{in}<\lda_{out}$
\begin{equation}
\|r^{-3}\sqrt{(y_{in}+1)(r+\lda_{in})}\Lda Q_{\lda_{in}}\Lda Q_{\lda_{out}}\|_{L^{2}}\aleq|t|^{-2}\lda_{in}^{1/2}.\label{eq:inter-dual_Mor}
\end{equation}
\end{lem}

\begin{proof}
Combine (\ref{eq:pre-ratio-b-b_t}) with (\ref{eq:H1-est-interatction}),
(\ref{eq:pre-H_Q-interaction-0})--(\ref{eq:pre-H_Q-interaction-1}),
and (\ref{eq:bubble-dual-Mor}), to obtain (\ref{eq:dotH1-interaction}),
(\ref{eq:H2-HQbubble}), and (\ref{eq:inter-dual_Mor}), respectively.
\end{proof}

\subsection{\label{subsec:Modulation-estimates}Modulation estimates}

In this subsection, we estimate time variations of the modulation
parameters $\lmb_{j}$, $b_{j}$, and their refined versions
\begin{equation}
\hat{b}_{j}\define\kap^{-1}\lan\Lda Q_{\ubr{;j}},\dot{u}\ran,\qquad j\in\{1,\dots,J\}.\label{eq:refined-mod-para}
\end{equation}
The estimates for $\lmb_{j,t}$ and $b_{j,t}$ are obtained by differentiating
in time the orthogonality conditions (\ref{eq:orthog-g}). Due to
our choice of the orthogonality conditions (\ref{eq:orthog-g}), $\lmb_{j,t}\approx-b_{j}$
can be well justified. However, the fact that $\calZ$ is not precisely
$\Lmb Q$ gives an estimate of $b_{j,t}$ that is insufficient to
justify the leading term of (\ref{eq:intro-formal-ODE}). (As mentioned
in Section~\ref{subsec:Coercivity-estimates}, we could not take
$\calZ=\Lmb Q$ when $k=3$ due to the linear coercivity estimates
therein.) A standard remedy is to introduce refined modulation parameters
(\ref{eq:refined-mod-para}). 

To begin with, let us record evolution equations for $g$, $\dot{u}$,
and $\dot{g}$, which are obtained by substituting the decomposition
(\ref{eq:decomp-u}) of $\bm{u}$ into (\ref{eq:k-equiv-WM}): 
\begin{align}
\rd_{t}g & =\tsum j{}(\lda_{j,t}+b_{j})\Lda Q_{\ul{;j}}+\dot{g},\label{eq:g-eqn}\\
\rd_{t}\dot{u} & =-H_{\calQ}g-\NL_{\calQ}(g)+f_{\mathbf{i}}(\vec{\iota},\vec{\lda}),\label{eq:dot-u-eqn}\\
\rd_{t}\dot{g} & =-H_{\calQ}g-\NL_{\calQ}(g)+f_{\mathbf{i}}(\vec{\iota},\vec{\lda})+\tsum j{}\Big\{ b_{j}\frac{\lda_{j,t}}{\lda_{j}}\Lda_{0}\Lda Q_{\ubr{;j}}-b_{j,t}\Lda Q_{\ubr{;j}}\Big\}.\label{eq:dot-g-eqn}
\end{align}
The following is the main result of this subsection.
\begin{prop}[Modulation estimates]
\label{prop:modulation-estimates}For $j\in\{1,\dots,J\}$, we have
rough modulation estimates 
\begin{align}
|\lda_{j,t}+b_{j}| & \aleq o(1)\cdot\min\{\lmb_{j}\|\bm{g}\|_{\dot{\calH}^{2}},\lmb_{j}^{3/2}\|\bm{g}\|_{\Mor}\},\label{eq:mod-lda_j,t+b_j}\\
|\lda_{j,t}| & \aleq\lda_{j}|t|^{-1},\label{eq:mod-lda_j,t}\\
|b_{j,t}| & \aleq\min\{\|\bm{g}\|_{\dot{\calH}^{2}},\lda_{j}^{1/2}\dlt_{0}^{-(j-1)/2}\|\bm{g}\|_{\Mor}\}+\lda_{j}|t|^{-2},\label{eq:mod-b_j,t}
\end{align}
and refined modulation estimates 
\begin{align}
\hat{b}_{j} & =b_{j}+o(|t|^{-\eps_{j}})\lda_{j}|t|^{-1},\label{eq:diff-b-hat_b}\\
\hat{b}_{j,t} & =-8k^{2}\kap^{-1}\frac{1}{\lmb_{j}}\Big(\frac{\lmb_{j}}{\lmb_{j-1}}\Big)^{k}+o(|t|^{-\eps_{j}})\lda_{j}|t|^{-2}\label{eq:refined-mod-b_j,t}
\end{align}
under the convention that $\lmb_{0}=\infty$.
\end{prop}

\begin{proof}
\uline{Proof of \mbox{(\ref{eq:mod-lda_j,t+b_j})}}. Differentiating
in time the orthogonality conditions $\lan\calZ_{\ul{;i}},g\ran=0$
and substituting the equation (\ref{eq:g-eqn}) for $g$ and the orthogonality
conditions for $\dot{g}$ (\ref{eq:orthog-g}), we obtain 
\begin{align*}
0=\frac{d}{dt}\lan\calZ_{\ul{;i}},g\ran & =\lan\calZ_{\ul{;i}},\rd_{t}g\ran-\frac{\lmb_{i,t}}{\lmb_{i}}\lan\Lmb_{0}\calZ_{\ul{;i}},g\ran\\
 & =\tsum j{}\lan\calZ_{\ul{;i}},\Lmb Q_{\ul{;j}}\ran(\lmb_{j,t}+b_{j})-\lmb_{i,t}\lan\lmb_{i}^{-1}\Lmb_{0}\calZ_{\ul{;i}},g\ran.
\end{align*}
This can be rewritten as 
\begin{equation}
\tsum j{}\{\lan\calZ_{\ul{;i}},\Lmb Q_{\ul{;j}}\ran-\chf_{j=i}\lan\lmb_{i}^{-1}\Lmb_{0}\calZ_{\ul{;i}},g\ran\}(\lmb_{j,t}+b_{j})=-b_{i}\lan\lmb_{i}^{-1}\Lmb_{0}\calZ_{\ul{;i}},g\ran.\label{eq:tmp}
\end{equation}
Since $\lan\calZ_{\ul{;i}},\Lmb Q_{\ul{;i}}\ran=1$ and $\lan\lmb_{i}^{-1}\Lmb_{0}\calZ_{\ul{;i}},g\ran=\calO(\|g\|_{\dot{H}_{k}^{1}})=o(1)$,
we can divide both sides of (\ref{eq:tmp}) by the diagonal $1-\lan\lmb_{i}^{-1}\Lmb_{0}\calZ_{\ul{;i}},g\ran=1+o(1)$
to obtain 
\[
\big|\tsum j{}\{\chf_{j=i}+\calO(\chf_{j\neq i}|\lan\calZ_{\ubr{;i}},\Lda Q_{\ubr{;j}}\ran|)\}(\lmb_{j,t}+b_{j})\big|\aleq|b_{i}\lan\lmb_{i}^{-1}\Lmb_{0}\calZ_{\ul{;i}},g\ran|.
\]
For the left hand side, we use $\calO(\chf_{j\neq i}|\lan\calZ_{\ubr{;i}},\Lda Q_{\ubr{;j}}\ran|)=\calO(\chf_{j\neq i}\|\Lmb Q_{\ul{;i}}\Lmb Q_{\ul{;j}}\|_{L^{1}})$.
For the right hand side, we estimate using (\ref{eq:pre-lda-b}):
\begin{align*}
 & |b_{i}\lan\lmb_{i}^{-1}\Lmb_{0}\calZ_{\ul{;i}},g\ran|\\
 & \aleq\frac{|b_{i}|}{\lmb_{i}}\min\{\|r^{2}\Lmb_{0}\calZ_{\ul{;i}}\|_{L^{2}}\|r^{-2}g\|_{L^{2}},\|r^{2}(\lmb_{i}+r)^{1/2}\Lmb_{0}\calZ_{\ul{;i}}\|_{L^{2}}\|r^{-2}(\lmb_{i}+r)^{-1/2}g\|_{L^{2}}\}\\
 & \aleq|t|^{-1}\min\{\lmb_{i}\|\bm{g}\|_{\dot{\calH}^{2}},\lmb_{i}^{3/2}\|\bm{g}\|_{\Mor(\lmb_{i})}\}\aleq o(1)\min\{\lmb_{i}\|\bm{g}\|_{\dot{\calH}^{2}},\lmb_{i}^{3/2}\|\bm{g}\|_{\Mor}\}.
\end{align*}
Hence, 
\[
\big|(\lmb_{i,t}+b_{i})+\tsum j{}\calO(\chf_{j\neq i}\|\Lmb Q_{\ul{;i}}\Lmb Q_{\ul{;j}}\|_{L^{1}})(\lmb_{j,t}+b_{j})\big|\aleq o(1)\min\{\lmb_{i}\|\bm{g}\|_{\dot{\calH}^{2}},\lmb_{i}^{3/2}\|\bm{g}\|_{\Mor}\}.
\]
This together with the mapping property (\ref{eq:Pre-mapping-prop})
gives (\ref{eq:mod-lda_j,t+b_j}).

\uline{Proof of \mbox{(\ref{eq:mod-lda_j,t})}}. This follows from
substituting (\ref{eq:pre-lda-b}) and (\ref{eq:boot-assump}) into
(\ref{eq:mod-lda_j,t+b_j}).

\uline{Proof of \mbox{(\ref{eq:mod-b_j,t})}}. We differentiate
in time the orthogonality condition $\lan\calZ_{\ul{;i}},\dot{g}\ran=0$
and substituting the equation (\ref{eq:dot-g-eqn}) for $\dot{g}$,
to obtain 
\[
0=\frac{d}{dt}\lan\calZ_{\ul{;i}},\dot{g}\ran=\frac{\lmb_{i,t}}{\lmb_{i}}\lan\Lmb_{0}\calZ_{\ul{;i}},\dot{g}\ran+\lan\calZ_{\ul{;i}},-H_{\calQ}g-\NL_{\calQ}(g)+f_{\mathbf{i}}(\vec{\iota},\vec{\lda})+\tsum j{}b_{j}\frac{\lda_{j,t}}{\lda_{j}}\Lda_{0}\Lda Q_{\ubr{;j}}-\tsum j{}b_{j,t}\Lda Q_{\ubr{;j}}\ran.
\]
This can be rewritten as 
\begin{equation}
\begin{aligned}\tsum j{}\lan\calZ_{\ul{;i}},\Lda Q_{\ubr{;j}}\ran b_{j,t} & =-\lan\calZ_{\ul{;i}},H_{\calQ}g\ran+\lan\calZ_{\ul{;i}},f_{\mathbf{i}}(\vec{\iota},\vec{\lda})\ran+\tsum j{}b_{j}\frac{\lda_{j,t}}{\lda_{j}}\lan\calZ_{\ul{;i}},\Lda_{0}\Lda Q_{\ubr{;j}}\ran\\
 & \quad-\lan\calZ_{\ul{;i}},\NL_{\calQ}(g)\ran+\frac{\lmb_{i,t}}{\lmb_{i}}\lan\Lmb_{0}\calZ_{\ul{;i}},\dot{g}\ran.
\end{aligned}
\label{eq:mod-1-1}
\end{equation}
For the left hand side of (\ref{eq:mod-1-1}), by $\lan\calZ,\Lmb Q\ran=1$,
\begin{equation}
\text{LHS}\eqref{eq:mod-1-1}=b_{i,t}+\tsum j{}\chf_{j\neq i}\lan\calZ_{\ul{;i}},\Lda Q_{\ubr{;j}}\ran b_{j,t}.\label{eq:mod-temp2}
\end{equation}

We estimate all terms of the right hand side of (\ref{eq:mod-1-1}).
We integrate by parts to estimate the linear term by 
\begin{align*}
 & |\lan\calZ_{\ul{;i}},H_{\calQ}g\ran|\aleq\|\Lmb Q_{\ul{;i}}\cdot r^{-1}|g|_{-1}\|_{L^{1}}\\
 & \quad\aleq\min\{\|\Lmb Q_{\ul{;i}}\|_{L^{2}}\|r^{-1}|g|_{-1}\|_{L^{2}},\|(\lmb_{i}+r)^{1/2}\Lmb Q_{\ul{;i}}\|_{L^{2}}\|r^{-1}(\lmb_{i}+r)^{-1/2}|g|_{-1}\|_{L^{2}}\}\\
 & \quad\aleq\min\{\|\bm{g}\|_{\dot{\calH}^{2}},\lda_{i}^{1/2}\|\bm{g}\|_{\Mor(\lda_{i})}\}\aleq\min\{\|\bm{g}\|_{\dot{\calH}^{2}},\lda_{j}^{1/2}\dlt_{0}^{-(j-1)/2}\|\bm{g}\|_{\Mor}\}.
\end{align*}
All the other terms will be of size $\calO(|t|^{-2}\lda_{i})$. For
the interaction term, we use (\ref{eq:dotH1-interaction}) to get
\begin{align*}
|\lan\calZ_{\ul{;i}},f_{\mathbf{i}}(\vec{\iota},\vec{\lda})\ran|\aleq\lda_{i}\||f_{\mathbf{i}}(\vec{\iota},\vec{\lda})|_{-1}\|_{L^{2}} & \aleq|t|^{-2}\lda_{i}.
\end{align*}
Next, we use $\|\calZ_{\ul{;i}}\Lda_{0}\Lda Q_{\ubr{;i}}\|_{L^{1}}\aleq1$
and $\chf_{j\neq i}\|\calZ_{\ul{;i}}\Lda_{0}\Lda Q_{\ubr{;j}}\|_{L^{1}}\aleq\chf_{j\neq i}\|\Lmb Q_{\ul{;i}}\Lda Q_{\ubr{;j}}\|_{L^{1}}$,
and then apply (\ref{eq:pre-lda-b}), (\ref{eq:mod-lda_j,t}), and
(\ref{eq:Pre-mapping-prop}) to obtain 
\begin{align*}
\Big|\tsum j{}b_{j}\frac{\lda_{j,t}}{\lda_{j}}\lan\calZ_{\ul{;i}},\Lda_{0}\Lda Q_{\ubr{;j}}\ran\Big| & \aleq|b_{i}\frac{\lda_{i,t}}{\lda_{i}}|+\tsum{j:j\neq i}{}|b_{j}\frac{\lmb_{j,t}}{\lmb_{j}}|\|\Lmb Q_{\ul{;i}}\Lda Q_{\ubr{;j}}\|_{L^{1}}\\
 & \aleq|t|^{-2}\lda_{i}+|t|^{-2}\tsum{j:j\neq i}{}\lda_{j}\|\Lmb Q_{\ul{;i}}\Lda Q_{\ubr{;j}}\|_{L^{1}}\aleq|t|^{-2}\lmb_{i}.
\end{align*}
For the last two terms of (\ref{eq:mod-1-1}), we claim a slightly
stronger bound 
\begin{equation}
\|\Lmb Q_{\ul{;i}}\NL_{\calQ}(g)\|_{L^{1}}+|\frac{\lmb_{i,t}}{\lmb_{i}}|\|\Lmb Q_{\ul{;i}}\dot{g}\|_{L^{1}}\aleq o(|t|^{-\eps_{i}})|t|^{-2}\lmb_{i},\label{eq:mod-claim3}
\end{equation}
which immediately implies 
\[
-\lan\calZ_{\ul{;i}},\NL_{\calQ}(g)\ran+\frac{\lmb_{i,t}}{\lmb_{i}}\lan\Lmb_{0}\calZ_{\ul{;i}},\dot{g}\ran=\calO(|t|^{-2}\lmb_{i})\cdot
\]
For the nonlinear term, we use (\ref{eq:pre-NL}) to have 
\begin{align*}
\|\Lmb Q_{\ul{;i}}\NL_{\calQ}(g)\|_{L^{1}}\aleq\|r\cdot\Lmb Q_{\ul{;i}}\|_{L^{2}}\|r^{-1}\NL_{\calQ}(g)\|_{L^{2}} & \aleq\lda_{i}\|g\|_{\dot{H}_{k}^{2}}^{2}\aleq|t|^{-2\eps_{0}}\cdot|t|^{-2}\lda_{i}.
\end{align*}
For the other term, we use $|\lmb_{i,t}|\aleq|t|^{-1}\lmb_{i}$ to
estimate 
\begin{align*}
|\frac{\lmb_{i,t}}{\lmb_{i}}|\|\Lmb Q_{\ul{;i}}\dot{g}\|_{L^{1}} & \aleq|t|^{-1}\|r\cdot\Lda Q_{\ul{;i}}\|_{L^{2}}\|r^{-1}\dot{g}\|_{L^{2}}\aleq|t|^{-1}\lda_{i}\|\dot{g}\|_{\dot{H}_{k}^{1}}\aleq|t|^{-\eps_{0}}\cdot|t|^{-2}\lda_{i},
\end{align*}
hence completing the proof of the claim (\ref{eq:mod-claim3}). Collecting
all these estimates, we have proved 
\begin{equation}
|\text{RHS}\eqref{eq:mod-1-1}|\aleq\min\{\|\bm{g}\|_{\dot{\calH}^{2}},\lda_{j}^{1/2}\dlt_{0}^{-(j-1)/2}\|\bm{g}\|_{\Mor}\}+|t|^{-2}\lda_{i}.\label{eq:5.20}
\end{equation}

Combining (\ref{eq:mod-temp2}) and (\ref{eq:5.20}), we are led to
\[
b_{i,t}+\tsum j{}b_{j,t}\chf_{j\neq i}\lan\calZ_{\ul{;i}},\Lda Q_{\ubr{;j}}\ran=\calO(\min\{\|\bm{g}\|_{\dot{\calH}^{2}},\lda_{j}^{1/2}\dlt_{0}^{-(j-1)/2}\|\bm{g}\|_{\Mor}\}+|t|^{-2}\lda_{i}).
\]
This together with the mapping properties (\ref{eq:Pre-mapping-prop})
and (\ref{eq:pre-mapping-dlt}) conclude (\ref{eq:mod-b_j,t}). 

\uline{Proof of \mbox{(\ref{eq:diff-b-hat_b})}}. Substituting
the decomposition $\dot{u}=\tsum j{}b_{j}\Lmb Q_{\ul{;j}}+\dot{g}$
into (\ref{eq:refined-mod-para}) and using $\kpp=\|\Lmb Q\|_{L^{2}}^{2}$,
we get 
\[
\hat{b}_{i}-b_{i}=\kap^{-1}\lan\Lda Q_{\ubr{;i}},\dot{g}+\tsum j{}b_{j}\Lda Q_{\ubr{;j}}\ran-b_{i}=\kap^{-1}\lan\Lda Q_{\ubr{;i}},\dot{g}\ran+\kap^{-1}\tsum j{}\chf_{j\neq i}\lan\Lda Q_{\ubr{;i}},\Lda Q_{\ubr{;j}}\ran b_{j}.
\]
Observe 
\begin{align*}
\lan\Lda Q_{\ubr{;i}},\dot{g}\ran & =\calO(\lda_{i}\|\bm{g}\|_{\dot{\calH}^{2}})=\calO(|t|^{-1-\eps_{0}}\lmb_{i})=o(|t|^{-1-\eps_{i}}\lmb_{i}),\\
\tsum j{}\chf_{j\neq i}\lan\Lda Q_{\ubr{;i}},\Lda Q_{\ubr{;j}}\ran b_{j} & =|t|^{-1}\tsum j{}\chf_{j\neq i}\calO(\|\Lda Q_{\ubr{;i}}\Lda Q_{\ubr{;j}}\|_{L^{1}}\lmb_{j})=|t|^{-1}\calO(|t|^{-2}\cdot\lmb_{i})=o(|t|^{-1-\eps_{i}}\lmb_{i}),
\end{align*}
where we used $\eps_{0}>\eps_{i}$, (\ref{eq:Pre-mapping-prop}),
and $\eps_{i}<2$. This completes the proof of (\ref{eq:diff-b-hat_b}).

\uline{Proof of \mbox{(\ref{eq:refined-mod-b_j,t})}}. We differentiate
(\ref{eq:refined-mod-para}) in time and substitute (\ref{eq:dot-u-eqn})
to obtain 
\begin{align*}
\kap\hat{b}_{i,t} & =\frac{d}{dt}\lan\Lda Q_{\ul{;i}},\dot{u}\ran=\frac{\lmb_{i,t}}{\lmb_{i}}\lan\Lmb_{0}\Lmb Q_{\ul{;i}},\dot{u}\ran+\lan\Lda Q_{\ubr{;i}},-H_{\calQ}g-\NL_{\calQ}(g)+f_{\mathbf{i}}(\vec{\iota},\vec{\lda})\ran.
\end{align*}
We then substitute the decomposition (\ref{eq:decomp-u}) of $\dot{u}$
and reorganize as 
\begin{equation}
\begin{aligned}\kap\hat{b}_{i,t} & =\lan\Lda Q_{\ubr{i;}},f_{\mathbf{i}}(\vec{\iota},\vec{\lda})\ran-\lan\Lmb Q_{\ul{;i}},H_{\calQ}g\ran-\lan\Lda Q_{\ul{;i}},\NL_{\calQ}(g)\ran\\
 & \quad+\frac{\lmb_{i,t}}{\lmb_{i}}\lan\Lmb_{0}\Lda Q_{\ul{;i}},\dot{g}\ran+\frac{\lmb_{i,t}}{\lmb_{i}}\tsum j{}b_{j}\lan\Lmb_{0}\Lmb Q_{\ul{;i}},\Lda Q_{\ubr{;j}}\ran.
\end{aligned}
\label{eq:mod-temp8}
\end{equation}

We extract the leading contribution from the first term and estimate
the remainders by $o(|t|^{-\eps_{i}})\lda_{i}|t|^{-2}$. First, by
(\ref{eq:inner-prod-interaction}) with $\max_{j\geq2}\frac{\lmb_{j}}{\lmb_{j-1}}\aeq\frac{\lmb_{2}}{\lmb_{1}}\aeq|t|^{-\frac{2}{k-2}}$,
(\ref{eq:pre-ratio-b-b_t}), and $\frac{2}{k-2}>\eps_{i}$ to have
\begin{align*}
\lan\Lda Q_{\ubr{;i}},f_{\mathbf{i}}(\vec{\iota},\vec{\lda})\ran & =-8k^{2}\frac{1}{\lmb_{i}}\Big(\frac{\lmb_{i}}{\lmb_{i-1}}\Big)^{k}+\calO\Big(\frac{\lmb_{2}}{\lmb_{1}}\cdot\lda_{i}|t|^{-2}+\frac{\lmb_{i+1}}{\lmb_{i}}\cdot\lmb_{i+1}|t|^{-2}\Big)\\
 & =-8k^{2}\frac{1}{\lmb_{i}}\Big(\frac{\lmb_{i}}{\lmb_{i-1}}\Big)^{k}+o(|t|^{-\eps_{i}})\lda_{i}|t|^{-2}.
\end{align*}
Next, for the the linear term, we use symmetry of $H_{\calQ}$, (\ref{eq:H2-HQbubble}),
(\ref{eq:boot-assump}) for $g$, and $1>\eps_{i}$ to have 
\begin{align*}
|\lan\Lmb Q_{\ul{;i}},H_{\calQ}g\ran| & =|\lan H_{\calQ}\Lmb Q_{\ul{;i}},g\ran|\aleq\|rH_{\calQ}\Lmb Q_{\ul{;i}}\|_{L^{2}}\|g\|_{\dot{H}_{k}^{1}}\aleq\lmb_{i}|t|^{-2}(\log|t|)^{1/2}\cdot|t|^{-1}\aleq o(|t|^{-\eps_{i}})\lda_{i}|t|^{-2}.
\end{align*}
Next, we use (\ref{eq:mod-claim3}) to have 
\[
-\lan\Lda Q_{\ul{;i}},\NL_{\calQ}(g)\ran+\frac{\lmb_{i,t}}{\lmb_{i}}\lan\Lmb_{0}\Lda Q_{\ul{;i}},\dot{g}\ran=o(|t|^{-\eps_{i}})\lda_{i}|t|^{-2}.
\]
Finally, for the last term, we use the anti-symmetry of $\Lda_{0}$
and (\ref{eq:pre-lda-b}), (\ref{eq:mod-lda_j,t}), (\ref{eq:Pre-mapping-prop}),
and $2>\eps_{i}$ to have
\begin{align*}
\Big|\frac{\lmb_{i,t}}{\lmb_{i}}\tsum j{}b_{j}\lan\Lmb_{0}\Lmb Q_{\ul{;i}},\Lda Q_{\ubr{;j}}\ran\Big| & =\Big|\frac{\lmb_{i,t}}{\lmb_{i}}\tsum j{}\chf_{j\neq i}b_{j}\lan\Lmb_{0}\Lmb Q_{\ul{;i}},\Lda Q_{\ubr{;j}}\ran\Big|\\
 & \aleq|t|^{-2}\tsum j{}\chf_{j\neq i}\|\Lmb Q_{\ul{;i}}\Lmb Q_{\ul{;j}}\|_{L^{1}}\lmb_{j}\aleq|t|^{-2}\cdot|t|^{-2}\lmb_{i}\aleq o(|t|^{-\eps_{i}})\lda_{i}|t|^{-2}.
\end{align*}
This completes the proof of (\ref{eq:refined-mod-b_j,t}).
\end{proof}

\subsection{\label{subsec:Energy-Morawetz-functional}Energy-Morawetz functional}

As mentioned in Section~\ref{subsec:Strategy-of-the-proof}, an $\dot{\calH}^{2}$-bound
of the form $\|\bm{g}\|_{\dot{\calH}^{2}}\aleq|t|^{-1-}$ is crucial
to justify the formal ODE system (\ref{eq:intro-formal-ODE}). Thus
propagating such $\dot{\calH}^{2}$-control of $\bm{g}$, more precisely,
closing the $\dot{\calH}^{2}$-bootstrap bound (\ref{eq:boot-assump})
for $\bm{g}$, lies at the heart of the analysis. The goal of this
subsection is to introduce a modified energy functional at the $\dot{\calH}^{2}$-level
(\ref{eq:def-energy-Morawetz}) designed for this purpose.

A natural candidate for the $\dot{\calH}^{2}$-energy functional would
be the quadratic form generated by the linearized operator: 
\begin{equation}
\calE_{2}[\vec{\iota},\vec{\lmb};\bm{g}]\coloneqq\lan H_{\calQ}g,H_{\calQ}g\ran+\lan\dot{g},H_{\calQ}\dot{g}\ran.\label{eq:def-calE_2}
\end{equation}
Note that the coercivity estimates (\ref{eq:H_calQ-Hdot1-coer})--(\ref{eq:H_calQ-Hdot2-coer})
with the orthogonality condition (\ref{eq:orthog-g}) imply 
\begin{equation}
\calE_{2}[\vec{\iota},\vec{\lmb};\bm{g}]\aeq\|\bm{g}\|_{\dot{\calH}^{2}}^{2}.\label{eq:calE_2-coer}
\end{equation}
However, this naive choice is not sufficient to propagate the $\dot{\calH}^{2}$-control
of $\bm{g}$. Indeed, typical problematic terms arise from the time
variation of $H_{\calQ}$; explicitly, $\frac{d}{dt}\calE_{2}$ contains
\[
\lan H_{\calQ}g,(\rd_{t}H_{\calQ})g\ran\quad\text{and}\quad\lan\dot{g},(\rd_{t}H_{\calQ})\dot{g}\ran.
\]
If simply estimated by the $\dot{\calH}^{2}$-norm of $\bm{g}$, these
terms are bounded by 
\[
|\lan H_{\calQ}g,(\rd_{t}H_{\calQ})g\ran|+|\lan\dot{g},(\rd_{t}H_{\calQ})\dot{g}\ran|\aleq\Big(\max_{j}\frac{|\lmb_{j,t}|}{\lmb_{j}}\Big)\|\bm{g}\|_{\dot{\calH}^{2}}^{2}.
\]
However, $\lmb_{j}(t)\to0$ (if $j\geq2$) and hence $|\log\lmb_{j}(t)|\to\infty$
as $t\to-\infty$. Thus the factor $\max_{j}\frac{|\lmb_{j,t}|}{\lmb_{j}}$
is never integrable in time and an $\dot{\calH}^{2}$-bound of the
form $\|\bm{g}\|_{\dot{\calH}^{2}}\aleq|t|^{-1-}$ cannot be propagated;
this is a typical difficulty arising in blow-up dynamics.

A well-known strategy to overcome this difficulty is to add correction
terms, whose time derivatives either cancel or dominate (with good
sign) such problematic terms. Our strategy in this work is the latter.
We utilize the multi-bubble Morawetz functional $\calM$ introduced
in Proposition~\ref{prop:Morawetz} whose time derivative control
(i.e., the lower bound $\|\bm{g}\|_{\Mor}^{2}$ of (\ref{eq:Morawetz-monotonicity}))
dominates these problematic terms. 

More explicitly, we consider the \emph{energy-Morawetz functional}
\begin{equation}
\calI[\vec{\iota},\vec{\lmb};\bm{g}]\define\calE_{2}[\vec{\iota},\vec{\lmb};\bm{g}]+\td{\dlt}\calM[\vec{\lmb};g,\dot{g}],\label{eq:def-energy-Morawetz}
\end{equation}
where we fix a small constant $\td{\dlt}>0$ such that $\calI$ is
coercive (thanks to (\ref{eq:calE_2-coer}) and (\ref{eq:Morawetz-bdd})):
\begin{equation}
\calI\aeq\|\bm{g}\|_{\dot{\calH}^{2}}^{2}.\label{eq:Energy-Mor-coer}
\end{equation}
The following proposition realizes the strategy of the previous paragraph\@.
\begin{prop}[Energy-Morawetz functional]
\label{prop:energy-Morawetz}There exist $c',\dlt_{0}>0$ such that
the energy-Morawetz functional $\calI$ enjoys a monotonicity estimate
\begin{align}
\frac{d}{dt}\calI & \leq-c'\|\bm{g}\|_{\Mor}^{2}+\calO(|t|^{-1-\eps_{0}}\|\bm{g}\|_{\dot{\calH}^{2}}^{2}+|t|^{-4}).\label{eq:Energy-Mor-Mono}
\end{align}
\end{prop}

\begin{rem}[On parameter $\dlt_{0}$]
\label{rem:on-dlt_0}The definition of $\calM$ contains the parameter
$\dlt_{0}>0$, and we fix $\dlt_{0}>0$ in the proof of Proposition~\ref{prop:energy-Morawetz}.
More precisely, $\dlt_{0}$ is introduced to control the most delicate
term (\ref{eq:energy-del}) arising in $\frac{d}{dt}\calI$, which
is of size $\calO(\dlt_{0}\|\bm{g}\|_{\Mor}^{2})+\text{(good terms)}$;
see (\ref{eq:I_del-est}). In order to dominate this term by the Morawetz
control $\|\bm{g}\|_{\Mor}^{2}$, we need sufficient smallness of
$\dlt_{0}$. This is the reason why we keep $\dlt_{0}$ in Proposition~\ref{prop:Morawetz}.
\end{rem}

\begin{proof}
Note that we can ignore any dependence on $\td{\dlt}>0$ as it is
already fixed.

\uline{Step 1: Computation of \mbox{$\frac{d}{dt}\calI$}}. In
this step, we compute $\frac{d}{dt}\calI$ and categorize the resulting
terms into three terms $I_{\dot{\calH}^{2}}$, $I_{\Mor}$, and $I_{del}$.
The first term $I_{\dot{\calH}^{2}}$ is a perturbative error, whereas
$I_{\Mor}$ and $I_{del}$ are problematic terms (if simply estimated
using the $\dot{\calH}^{2}$-energy) that will be estimated using
the Morawetz control as well. Here, the last term $I_{del}$ requires
an extra care as explained in Remark~\ref{rem:on-dlt_0}.

We begin with the time derivative of $\calE_{2}$:
\[
\frac{d}{dt}\calE_{2}=2\lan H_{\calQ}g,H_{\calQ}\rnd_{t}g\ran+2\lan\rnd_{t}\dot{g},H_{\calQ}\dot{g}\ran+2\lan H_{\calQ}g,(\rnd_{t}H_{\calQ})g\ran+\lan\dot{g},(\rnd_{t}H_{\calQ})\dot{g}\ran.
\]
Substituting the evolution equations (\ref{eq:g-eqn}) and (\ref{eq:dot-g-eqn})
of $\bm{g}$ and noting a cancellation of the terms $\pm\lan H_{\calQ}g,H_{\calQ}\dot{g}\ran$,
we get 
\begin{align*}
 & \lan H_{\calQ}g,H_{\calQ}\rnd_{t}g\ran+\lan\rnd_{t}\dot{g},H_{\calQ}\dot{g}\ran\\
 & =\tsum j{}(\lda_{j,t}+b_{j})\lan H_{\calQ}g,H_{\calQ}\Lda Q_{\ubr{;j}}\ran+\lan-\NL_{\calQ}(g),H_{\calQ}\dot{g}\ran+\lan f_{\mathbf{i}}(\vec{\iota},\vec{\lda}),H_{\calQ}\dot{g}\ran\\
 & \quad+\tsum j{}\frac{\lda_{j,t}}{\lda_{j}}b_{j}\lan\Lda_{0}\Lda Q_{\ubr{;j}},H_{\calQ}\dot{g}\ran-\tsum j{}b_{j,t}\lan\Lda Q_{\ubr{;j}},H_{\calQ}\dot{g}\ran.
\end{align*}
We turn to the time derivative of $\calM$: 
\[
\frac{d}{dt}\calM[\vec{\lmb};g,\dot{g}]=\tsum j{}\lda_{j,t}\rnd_{\lda_{j}}\calM[\vec{\lmb};g,\dot{g}]+\calM[\vec{\lmb};g,\rnd_{t}\dot{g}]+\calM[\vec{\lmb};\rnd_{t}g,\dot{g}].
\]
Using (\ref{eq:g-eqn}) and (\ref{eq:dot-g-eqn}), the last two terms
are rewritten as 
\begin{align*}
 & \calM[\vec{\lmb};g,\rnd_{t}\dot{g}]+\calM[\vec{\lmb};\rnd_{t}g,\dot{g}]\\
 & =\calM[\vec{\lda};g,-H_{\calQ}g]+\calM[\vec{\lda};g,-\NL_{\calQ}(g)]+\calM[\vec{\lda};g,f_{\mathbf{i}}(\vec{\iota},\vec{\lda})]+\tsum j{}\tfrac{\lmb_{j,t}}{\lmb_{j}}b_{j}\calM[\vec{\lda};g,\Lda_{0}\Lda Q_{\ubr{;j}}]\\
 & \quad-\tsum j{}b_{j,t}\calM[\vec{\lda};g,\Lda Q_{\ubr{;j}}]+\calM[\vec{\lda};\dot{g},\dot{g}]+\tsum j{}(\lda_{j,t}+b_{j})\calM[\vec{\lda};\Lda Q_{\ubr{;j}},\dot{g}].
\end{align*}
Gathering the previous four displays and rearranging, we get 
\begin{equation}
\frac{d}{dt}\calI=\td{\dlt}(\calM[\vec{\lda};g,-H_{\calQ}g]+\calM[\vec{\lda};\dot{g},\dot{g}])+I_{\dot{\calH}^{2}}+I_{\Mor}+I_{del},\label{eq:energy-derivative-I}
\end{equation}
where 
\begin{align}
I_{\dot{\calH}^{2}} & \define2\lan-\NL_{\calQ}(g),H_{\calQ}\dot{g}\ran+2\tsum j{}(\lda_{j,t}+b_{j})\lan H_{\calQ}g,H_{\calQ}\Lda Q_{\ubr{;j}}\ran-2\tsum j{}b_{j,t}\lan\Lda Q_{;\ubr j},H_{\calQ}\dot{g}\ran,\label{eq:energy-H2}\\
I_{\Mor} & \define2\lan f_{\mathbf{i}}(\vec{\iota},\vec{\lda}),H_{\calQ}\dot{g}\ran+\td{\dlt}\calM[\vec{\lda};g,f_{\mathbf{i}}(\vec{\iota},\vec{\lda})]+2\tsum j{}\tfrac{\lmb_{j,t}}{\lmb_{j}}b_{j}\lan\Lda_{0}\Lda Q_{\ubr{;j}},H_{\calQ}\dot{g}\ran\label{eq:energy-mor}\\
 & \quad+\td{\dlt}\tsum j{}\tfrac{\lmb_{j,t}}{\lmb_{j}}b_{j}\calM[\vec{\lda};g,\Lda_{0}\Lda Q_{\ubr{;j}}]+\td{\dlt}\tsum j{}(\lda_{j,t}+b_{j})\calM[\vec{\lda};\Lda Q_{\ubr{;j}},\dot{g}]\nonumber \\
 & \quad+2\lan H_{\calQ}g,(\rnd_{t}H_{\calQ})g\ran+\lan\dot{g},(\rnd_{t}H_{\calQ})\dot{g}\ran+\td{\dlt}\tsum j{}\lda_{j,t}\rnd_{\lda_{j}}\calM[\vec{\lda};g,\dot{g}],\nonumber \\
I_{del} & \define-\td{\dlt}\tsum j{}b_{j,t}\calM[\vec{\lda};g,\Lda Q_{\ubr{;j}}].\label{eq:energy-del}
\end{align}

\uline{Step 2: Proof of \mbox{(\ref{eq:Energy-Mor-Mono})} assuming
Claims \mbox{(\ref{eq:I_H2-est})}, \mbox{(\ref{eq:I_Mor-est})},
\mbox{(\ref{eq:I_del-est})}}. 
\begin{align}
|I_{\dot{\calH}^{2}}| & \aleq|t|^{-1-\eps_{0}}\|\bm{g}\|_{\dot{\calH}^{2}}^{2}+|t|^{-4},\label{eq:I_H2-est}\\
|I_{Mor}| & \aleq|t|^{-2}\|\bm{g}\|_{\Mor}+|t|^{-1}\|\bm{g}\|_{\dot{\calH}^{2}}\|\bm{g}\|_{\Mor}+o(1)\|\bm{g}\|_{\Mor}^{2},\label{eq:I_Mor-est}\\
|I_{del}| & \aleq(\dlt_{0}+o(1))\|\bm{g}\|_{\Mor}^{2}+|t|^{-2}\|\bm{g}\|_{\Mor}.\label{eq:I_del-est}
\end{align}
Assuming these claims, applying the monotonicity (\ref{eq:Morawetz-monotonicity})
to the first term of (\ref{eq:energy-derivative-I}) gives 
\[
\frac{d}{dt}\calI\leq-(c\td{\dlt}-\calO(\dlt_{0})-o(1))\|\bm{g}\|_{\Mor}^{2}+\calO(|t|^{-1-\eps_{0}}\|\bm{g}\|_{\dot{\calH}^{2}}^{2}+|t|^{-2}\|\bm{g}\|_{\Mor}+|t|^{-1}\|\bm{g}\|_{\dot{\calH}^{2}}\|\bm{g}\|_{\Mor}+|t|^{-4}).
\]
where the constant $c>0$ is independent of $\dlt_{0}$. Applying
Young's inequality for the second and third error terms, we get 
\[
\frac{d}{dt}\calI\leq-(\frac{c\td{\dlt}}{2}-\calO(\dlt_{0})-o(1))\|\bm{g}\|_{\Mor}^{2}+\calO(|t|^{-1-\eps_{0}}\|\bm{g}\|_{\dot{\calH}^{2}}^{2}+|t|^{-4}).
\]
Finally, taking $\dlt_{0}>0$ small yields (\ref{eq:Energy-Mor-Mono}).
It remains to prove the claims (\ref{eq:I_H2-est}), (\ref{eq:I_Mor-est}),
and (\ref{eq:I_del-est}). 

\uline{Step 3: Proof of \mbox{(\ref{eq:I_H2-est})}}. Recall 
\[
I_{\dot{\calH}^{2}}=2\lan-\NL_{\calQ}(g),H_{\calQ}\dot{g}\ran+2\tsum j{}(\lda_{j,t}+b_{j})\lan H_{\calQ}g,H_{\calQ}\Lda Q_{\ubr{;j}}\ran-2\tsum j{}b_{j,t}\lan\Lda Q_{;\ubr j},H_{\calQ}\dot{g}\ran.
\]
For the first, we integrate by parts and use (\ref{eq:pre-NL}) with
the bootstrap assumption (\ref{eq:boot-assump}) for $\bm{g}$ to
have 
\[
|\lan\NL_{\calQ}(g),H_{\calQ}\dot{g}\ran|\aleq\||\NL_{\calQ}(g)|_{-1}|\dot{g}|_{-1}\|_{L^{1}}\aleq\||\NL_{\calQ}(g)|_{-1}\|_{L^{2}}\|\dot{g}\|_{\dot{H}_{k}^{1}}\aleq\|\bm{g}\|_{\dot{\calH}^{2}}^{3}\aleq|t|^{-1-\eps_{0}}\|\bm{g}\|_{\dot{\calH}^{2}}^{2}.
\]
For the second, we use (\ref{eq:mod-lda_j,t+b_j}), (\ref{eq:H2-HQbubble}),
and $\lmb_{j}\aleq1$ to have 
\[
|\lda_{j,t}+b_{j}|\cdot|\lan H_{\calQ}g,H_{\calQ}\Lda Q_{\ubr{;j}}\ran|\aleq\lmb_{j}\|\bm{g}\|_{\dot{\calH}^{2}}\cdot\||g|_{-2}\|_{L^{2}}\|H_{\calQ}\Lda Q_{\ubr{;j}}\|_{L^{2}}\aleq|t|^{-2}\|\bm{g}\|_{\dot{\calH}^{2}}^{2}.
\]
For the last term, we use the symmetry of $H_{\calQ}$, (\ref{eq:mod-b_j,t}),
(\ref{eq:H2-HQbubble}), and (\ref{eq:boot-assump}) for $\dot{g}$
to have
\begin{align*}
|b_{j,t}|\cdot|\lan\Lda Q_{\ubr j},H_{\calQ}\dot{g}\ran| & \aleq(\|\bm{g}\|_{\dot{\calH}^{2}}+\lmb_{j}|t|^{-2})\|H_{Q}\Lda Q_{\ubr{;j}}\|_{L^{2}}\|\dot{g}\|_{L^{2}}\\
 & \aleq|t|^{-3}\|\bm{g}\|_{\dot{\calH}^{2}}+\lmb_{j}|t|^{-5}\aleq|t|^{-2}\|\bm{g}\|_{\dot{\calH}^{2}}^{2}+|t|^{-4}.
\end{align*}
This completes the proof of (\ref{eq:I_H2-est}).

\uline{Step 4: Proof of \mbox{(\ref{eq:I_Mor-est})}}. Recall 
\begin{align*}
I_{\Mor} & =2\lan f_{\mathbf{i}}(\vec{\iota},\vec{\lda}),H_{\calQ}\dot{g}\ran+\td{\dlt}\calM[\vec{\lda};g,f_{\mathbf{i}}(\vec{\iota},\vec{\lda})]+2\tsum j{}\tfrac{\lmb_{j,t}}{\lmb_{j}}b_{j}\lan\Lda_{0}\Lda Q_{\ubr{;j}},H_{\calQ}\dot{g}\ran\\
 & \quad+\td{\dlt}\tsum j{}\tfrac{\lmb_{j,t}}{\lmb_{j}}b_{j}\calM[\vec{\lda};g,\Lda_{0}\Lda Q_{\ubr{;j}}]+\td{\dlt}\tsum j{}(\lda_{j,t}+b_{j})\calM[\vec{\lda};\Lda Q_{\ubr{;j}},\dot{g}]\\
 & \quad+2\lan H_{\calQ}g,(\rnd_{t}H_{\calQ})g\ran+\lan\dot{g},(\rnd_{t}H_{\calQ})\dot{g}\ran+\td{\dlt}\tsum j{}\lda_{j,t}\rnd_{\lda_{j}}\calM[\vec{\lda};g,\dot{g}].
\end{align*}
For the first two terms, we integrate by parts the first term, and
apply (\ref{eq:Morawetz-bdd}), (\ref{eq:pointwise-interaction}),
and (\ref{eq:inter-dual_Mor}) to have
\begin{align*}
|\lan f_{\mathbf{i}}(\vec{\iota},\vec{\lda}), & H_{\calQ}\dot{g}\ran|+|\calM[\vec{\lda};g,f_{\mathbf{i}}(\vec{\iota},\vec{\lda})]|\aleq\||f_{\mathbf{i}}(\vec{\iota},\vec{\lda})|_{-1}(|\dot{g}|_{-1}+|g|_{-2})\|_{L^{1}}\\
 & \aleq\tsum{i,j:i<j}{}\|r^{-3}\Lda Q_{\lda_{i}}\Lda Q_{\lda_{j}}(|\dot{g}|_{-1}+|g|_{-2})\|_{L^{1}}\\
 & \aleq\tsum{i,j:i<j}{}\|r^{-3}(y_{j}+1)^{\frac{1}{2}}(r+\lmb_{j})^{\frac{1}{2}}\Lda Q_{\lda_{i}}\Lda Q_{\lda_{j}}\|_{L^{2}}\|\bm{g}\|_{\Mor(\lda_{j})}\\
 & \aleq|t|^{-2}\tsum j{}\lmb_{j}^{1/2}\|\bm{g}\|_{\Mor(\lda_{j})}\aleq|t|^{-2}\|\bm{g}\|_{\Mor}.
\end{align*}
For the third and fourth terms, we use $|\frac{\lmb_{j,t}}{\lmb_{j}}b_{j}|\aleq|t|^{-2}\lmb_{j}$
(see (\ref{eq:pre-lda-b}) and (\ref{eq:mod-lda_j,t})) and (\ref{eq:Morawetz-bdd})
to have 
\begin{align*}
 & |\tsum j{}\tfrac{\lmb_{j,t}}{\lmb_{j}}b_{j}\lan\Lda_{0}\Lda Q_{\ubr{;j}},H_{\calQ}\dot{g}\ran|+|\tsum j{}\tfrac{\lmb_{j,t}}{\lmb_{j}}b_{j}\calM[\vec{\lda};g,\Lda_{0}\Lda Q_{\ubr{;j}}]|\\
 & \aleq|t|^{-2}\tsum j{}\||\Lda_{0}\Lda Q_{\lda_{j}}|_{-1}(|\dot{g}|_{-1}+|g|_{-2})\|_{L^{1}}\aleq|t|^{-2}\tsum j{}\lda_{j}^{1/2}\|\bm{g}\|_{\Mor(\lda_{j})}\aleq|t|^{-2}\|\bm{g}\|_{\Mor}.
\end{align*}
For the fifth term, we use (\ref{eq:mod-lda_j,t+b_j}) and (\ref{eq:Morawetz-bdd})
to have 
\begin{align*}
|\tsum j{}(\lda_{j,t}+b_{j})\calM[\vec{\lda};\Lda Q_{\ubr{;j}},\dot{g}]| & \aleq\tsum j{}\{o(1)\lmb_{j}^{3/2}\|\bm{g}\|_{\Mor}\cdot\||\Lda Q_{\ul{;j}}|_{-2}|\dot{g}|_{-1}\|_{L^{1}}\}\\
 & \aleq\tsum j{}\{o(1)\lmb_{j}^{3/2}\|\bm{g}\|_{\Mor}\cdot\lda_{j}^{-3/2}\|\bm{g}\|_{\Mor(\lda_{j})}\}\aleq o(1)\|\bm{g}\|_{\Mor}^{2}.
\end{align*}
For the sixth and seventh terms, we use $|\lda_{j}\rnd_{\lda_{j}}H_{\calQ}|\aleq r^{-2}\lan y_{j}\ran^{-k}$
and (\ref{eq:mod-lda_j,t}) to have 
\begin{align*}
 & |\lan H_{\calQ}g,(\rnd_{t}H_{\calQ})g\ran|+|\lan\dot{g},(\rnd_{t}H_{\calQ})\dot{g}\ran|\\
 & \aleq\tsum j{}|\tfrac{\lmb_{j,t}}{\lmb_{j}}|\{\||g|_{-2}\|_{L^{2}}\|\lan y_{j}\ran^{-k}r^{-2}g\|_{L^{2}}+\|r^{-1}\dot{g}\|_{L^{2}}\|\lan y_{j}\ran^{-k}r^{-1}\dot{g}\|_{L^{2}}\}\\
 & \aleq|t|^{-1}\|\bm{g}\|_{\dot{\calH}^{2}}\tsum j{}\lda_{j}^{1/2}\|\bm{g}\|_{\Mor(\lda_{j})}\aleq|t|^{-1}\|\bm{g}\|_{\dot{\calH}^{2}}\|\bm{g}\|_{\Mor}.
\end{align*}
Finally, for the last term, we use (\ref{eq:mod-lda_j,t}) and (\ref{eq:Morawetz-lmb-rd-lmb-bdd})
to have 
\[
|\tsum j{}\lda_{j,t}\rnd_{\lda_{j}}\calM[\lmb_{j};g,\dot{g}]|\aleq|t|^{-1}\tsum j{}\lda_{j}^{1/2}\|\bm{g}\|_{\dot{\calH}^{2}}\|\bm{g}\|_{\Mor}\aleq|t|^{-1}\|\bm{g}\|_{\dot{\calH}^{2}}\|\bm{g}\|_{\Mor}.
\]
This completes the proof of (\ref{eq:I_Mor-est}).

\uline{Step 5: Proof of \mbox{(\ref{eq:I_del-est})}}. In this
step, we estimate the most delicate term 
\[
I_{del}=-\td{\dlt}\tsum j{}b_{j,t}\calM[\vec{\lda};g,\Lda Q_{\ubr{;j}}].
\]
First we use (\ref{eq:mod-b_j,t}) to have 
\[
|\tsum j{}b_{j,t}\calM[\vec{\lda};g,\Lda Q_{\ubr{;j}}]|\aleq\tsum j{}\{\lda_{j}^{1/2}\dlt_{0}^{-(j-1)/2}\|\bm{g}\|_{\Mor}+|t|^{-2}\lmb_{j}\}|\calM[\vec{\lda};g,\Lda Q_{\ubr{;j}}]|.
\]
We estimate the second term using (\ref{eq:Morawetz-bdd}):
\[
\tsum j{}|t|^{-2}\lda_{j}|\calM[\vec{\lda};g,\Lda Q_{\ubr{;j}}]|\aleq\tsum j{}|t|^{-2}\lda_{j}\||g|_{-2}|\Lmb Q_{\ul{;j}}|_{-1}\|_{L^{1}}\aleq|t|^{-2}\lmb_{j}^{1/2}\|\bm{g}\|_{\Mor}\aleq|t|^{-2}\|\bm{g}\|_{\Mor}.
\]
We estimate the first term using (\ref{eq:Mor-LdaQ}) with $(\vec{\iota},\vec{\lmb})\in\calP_{J}(o(1))$
to have 
\begin{equation}
\begin{aligned} & \tsum j{}\lda_{j}^{1/2}\dlt_{0}^{-(j-1)/2}\|\bm{g}\|_{\Mor}|\calM[\vec{\lda};g,\Lda Q_{\ubr{;j}}]|\\
 & \aleq\tsum j{}\{\lda_{j}^{1/2}\dlt_{0}^{-(j-1)/2}\|\bm{g}\|_{\Mor}\cdot(\dlt_{0}+o(1))\lmb_{j}^{-1/2}\dlt_{0}^{(j-1)/2}\|\bm{g}\|_{\Mor}\}\aleq(\dlt_{0}+o(1))\|\bm{g}\|_{\Mor}^{2}.
\end{aligned}
\label{eq:5.40}
\end{equation}
This completes the proof of (\ref{eq:I_del-est}) and hence the proof
of (\ref{eq:Energy-Mor-Mono}).
\end{proof}

\subsection{\label{subsec:Proof-of-main-bootstrap}Proof of main bootstrap proposition}

In this subsection, we prove Proposition~\ref{prop:Bootstrap}. Note
that we have ${\rm dist}(\bm{u}(t),K^{t_{0}})\leq\frac{1}{2}\eta$
for $K^{t_{0}}\define\{\bm{\calQ}(\vec{\iota},\vec{\lda})\colon\frac{1}{2}\lmb_{J}^{\ex}(t_{0})\leq\lmb_{j}\leq2\}$
on $[t_{0},T]$, which directly follows from (\ref{eq:boot-assump}).
\begin{lem}[Closing $\dot{\calH}^{1}$-energy, $\lmb_{1}$, and $b_{1}$]
We have 
\begin{equation}
\|\bm{g}\|_{\dot{\calH}^{1}}\leq\frac{1}{2}|t|^{-1},\quad|\lda_{1}-1|\leq\frac{1}{2}|t|^{-\eps_{1}},\quad\text{and}\quad|b_{1}|\leq\frac{1}{2}|t|^{-1}.\label{eq:closing-1}
\end{equation}
\end{lem}

\begin{proof}
\uline{Step 1: Closing \mbox{$\dot{\calH}^{1}$}-energy and \mbox{$b_{1}$}}.
We use energy conservation. Observe using (\ref{eq:pre-lda-b}) and
(\ref{eq:pre-ratio-b-b_t}) that 
\begin{equation}
\max_{j\in\{2,\dots,J\}}|b_{j}|+\max_{j\in\{2,\dots,J\}}\Big(\frac{\lda_{j}}{\lda_{j-1}}\Big)^{k/2}\aleq\lmb_{2}|t|^{-1}\aleq|t|^{-1-\frac{2}{k-2}}.\label{eq:5.42}
\end{equation}
By the definition (\ref{eq:O_init}) of $\bm{u}(t_{0})$, (\ref{eq:multi-bubble-energy}),
and the above display at $t=t_{0}$, we have 
\begin{align*}
\frac{1}{2\pi}|E[\bm{u}(t_{0})]-JE[\bm{Q}]| & \leq\frac{1}{2\pi}|E[\bm{\calQ}(\vec{\iota},\vec{\lmb}(t_{0}))]-JE[\bm{Q}]|+\frac{1}{2}\tint{}{}|\tsum j{}b_{j}(t_{0})\Lmb Q_{\ul{;j}}|^{2}\\
 & \aleq\max_{j\in\{2,\dots,J\}}\Big(\frac{\lda_{j}(t_{0})}{\lda_{j-1}(t_{0})}\Big)^{k}+\max_{j\in\{2,\dots,J\}}|b_{j}(t_{0})|^{2}\aleq|t_{0}|^{-2-\frac{4}{k-2}}.
\end{align*}
Together with energy conservation, we have proved 
\begin{equation}
\frac{1}{2\pi}E[\bm{u}(t)]=\frac{1}{2\pi}JE[\bm{Q}]+\calO(|t_{0}|^{-2-\frac{4}{k-2}}).\label{eq:5.43}
\end{equation}
On the other hand, we expand the nonlinear energy as (we simply write
$\bm{u}=\bm{u}(t)$ and $\calQ=\calQ(\vec{\iota},\vec{\lmb})$)
\begin{align*}
\frac{1}{2\pi}E[\bm{u}] & =\frac{1}{2\pi}E[\bm{\calQ}]+\lan-\rd_{rr}\calQ-\frac{1}{r}\rd_{r}\calQ+k^{2}\frac{f(\calQ)}{r^{2}},g\ran+\frac{1}{2}\lan H_{\calQ}g,g\ran+\frac{1}{2}\|\dot{u}\|_{L^{2}}^{2}+\calO(\|g\|_{\dot{H}_{k}^{1}}^{3}).
\end{align*}
Applying (\ref{eq:multi-bubble-energy}) with (\ref{eq:5.42}) to
the first, (\ref{eq:multi-bubble-small-linear}) with (\ref{eq:5.42})
to the second, (\ref{eq:H_calQ-Hdot1-coer}) to the third, we get
\[
\frac{1}{2\pi}E[\bm{u}]\geq\{\frac{1}{2\pi}JE[\bm{Q}]+\calO(|t|^{-2-\frac{4}{k-2}})\}+\calO(\|g\|_{\dot{H}_{k}^{1}}|t|^{-2-\frac{4}{k-2}})+\{c-\calO(\|g\|_{\dot{H}_{k}^{1}})\}\|g\|_{\dot{H}_{k}^{1}}^{2}+\frac{1}{2}\|\dot{u}\|_{L^{2}}^{2}
\]
for some $c>0$. Combining this with (\ref{eq:5.43}) and $\|g\|_{\dot{H}_{k}^{1}}=o(1)$,
we get 
\[
(c-o(1))\|g\|_{\dot{H}_{k}^{1}}^{2}+\frac{1}{2}\|\dot{u}\|_{L^{2}}^{2}\leq\calO(|t_{0}|^{-2-\frac{4}{k-2}}+|t|^{-2-\frac{4}{k-2}})\leq\calO(|t|^{-2-\frac{4}{k-2}}).
\]
Thus $\|g\|_{\dot{H}_{k}^{1}}+\|\dot{u}\|_{L^{2}}\aleq|t|^{-1-\frac{2}{k-2}}$.
Since $\max_{j\in\{1,\dots,J\}}|b_{j}|+\|\dot{g}\|_{L^{2}}\aleq\|\dot{u}\|_{L^{2}}$,
we finally get 
\begin{equation}
\|g\|_{\dot{H}_{k}^{1}}+\|\dot{g}\|_{L^{2}}+\max_{j\in\{1,\dots,J\}}|b_{j}|\aleq|t|^{-1-\frac{2}{k-2}}.\label{eq:5.44}
\end{equation}
This in particular gives (\ref{eq:closing-1}) for $\bm{g}$ and $b_{1}$.

\uline{Step 2: Closing \mbox{$\lmb_{1}$}}. The rough modulation
estimate (\ref{eq:mod-lda_j,t+b_j}) with $|b_{1}|\aleq|t|^{-1-\frac{2}{k-2}}$
of (\ref{eq:5.44}) gives $|\lda_{1,t}|\aleq|t|^{-1-\td{\eps}}$ with
$\td{\eps}=\min(\frac{2}{k-2},\eps_{0})>\eps_{1}$. Integrating this
from $t_{0}$ to $t$ with $\lmb_{1}(t_{0})=1$, we get $|\lda(t)-1|\aleq|t|^{-\td{\eps}}$.
Since $\td{\eps}>\eps_{1}$, this concludes $|\lda_{1}(t)-1|\leq\frac{1}{2}|t|^{-\eps_{1}}$
as desired. 
\end{proof}
\begin{lem}[Closing $\dot{\calH}^{2}$-energy]
We have 
\begin{equation}
\|\bm{g}(t)\|_{\dot{\calH}^{2}}\leq\frac{1}{2}|t|^{-1-\eps_{0}}.\label{eq:energy-int}
\end{equation}
\end{lem}

\begin{proof}
From (\ref{eq:Energy-Mor-Mono}) and (\ref{eq:boot-assump}), we have
\[
\frac{d}{dt}\calI\leq\calO(|t|^{-1-\eps_{0}}\|\bm{g}\|_{\dot{\calH}^{2}}^{2}+|t|^{-4})\leq\calO(|t|^{-3-3\eps_{0}}+|t|^{-4}).
\]
Integrating this from $t_{0}$ to $t$ gives $\calI(t)\leq\calI(t_{0})+\calO(|t|^{-2-3\eps_{0}}+|t|^{-3})$.
Using (\ref{eq:Energy-Mor-coer}), $\calI(t_{0})=0$ (from $\bm{g}(t_{0})=0$
in (\ref{eq:4.13})), and $\eps_{0}<\frac{1}{2}$, we get 
\[
\|\bm{g}(t)\|_{\dot{\calH}^{2}}^{2}\leq\calO(\calI(t))\leq\calO(|t|^{-2-3\eps_{0}}+|t|^{-3})\leq\frac{1}{4}|t|^{-2-2\eps_{0}}.
\]
This completes the proof.
\end{proof}
\begin{lem}[Closing $\lmb_{j}$ and $b_{j}$ for $j\geq2$]
\label{lem:ODE-analysis}Let $j\in\{2,\dots,J\}$. Recall $\nu_{j}$
and $\dot{\nu}_{j}$ from (\ref{eq:def-nu-and-nu-dot}).
\begin{enumerate}
\item (Control of stable modes) We have 
\begin{equation}
|\dot{\nu}_{j}|\leq\frac{1}{2}.\label{eq:ode-stable}
\end{equation}
\item (Shooting property) If $|\nu_{j}(T)|=|T|^{-\eps_{j}}$, then 
\begin{equation}
\frac{d}{dt}(|t|^{\eps_{j}}|\nu_{j}(t)|)\Big|_{t=T}>0.\label{eq:ode-shoot}
\end{equation}
\end{enumerate}
\end{lem}

\begin{proof}
\uline{Step 1: Simplified ODE system}. We claim that 
\begin{equation}
\left|\begin{aligned}\lda_{j,t}+\hat{b}_{j} & =o(|t|^{-\eps_{j}})\cdot(-b_{j}^{\ex}),\\
-\hat{b}_{j,t} & =(1+o(|t|^{-\eps_{j}}))\cdot\Big(\frac{\lda_{j}}{\lda_{j}^{\ex}}\Big)^{k-1}(-b_{j}^{\ex})_{t}.
\end{aligned}
\right.\label{eq:ode-ode}
\end{equation}
The first equation follows from (\ref{eq:mod-lda_j,t+b_j}), (\ref{eq:diff-b-hat_b}),
$\eps_{0}>\eps_{j}$, and $\lda_{j}|t|^{-1}\aeq-b_{j}^{\ex}$: 
\[
|\lmb_{j,t}+\hat{b}_{j}|\aleq o(1)\cdot\lmb_{j}\|\bm{g}\|_{\dot{\calH}^{2}}+o(|t|^{-\eps_{j}})\lda_{j}|t|^{-1}\aleq o(|t|^{-\eps_{j}})\cdot(-b_{j}^{\ex}).
\]
The second equation follows from (\ref{eq:refined-mod-b_j,t}), (\ref{eq:pre-ratio-b-b_t}),
and $\lmb_{j-1}=(1+o(|t|^{-\eps_{j}}))\lmb_{j-1}^{\ex}$ (which is
due to (\ref{eq:boot-assump}) and $\eps_{j-1}>\eps_{j}$):
\[
-\hat{b}_{j,t}=(8k^{2}\kap^{-1}+o(|t|^{-\eps_{j}}))\cdot\frac{\lda_{j}^{k-1}}{(\lda_{j-1}^{\ex})^{k}}=(1+o(|t|^{-\eps_{j}}))\cdot\Big(\frac{\lda_{j}}{\lda_{j}^{\ex}}\Big)^{k-1}(-b_{j}^{\ex})_{t}.
\]
This completes the proof of (\ref{eq:ode-ode}).

\uline{Step 2: Proof of \mbox{(\ref{eq:ode-stable})}}. By (\ref{eq:ode-ode})
and $\nu_{j}=\calO(|t|^{-\eps_{j}})$, we have $\hat{b}_{j,t}=(1+\calO(|t|^{-\eps_{j}}))\cdot b_{j,t}^{\ex}$.
Integrating this from $t_{0}$ to $t$ gives 
\[
|\hat{b}_{j}(t)-b_{j}^{\ex}(t)|\leq|\hat{b}_{j}(t_{0})-b_{j}^{\ex}(t_{0})|+\calO(|t|^{-\eps_{j}}|b_{j}^{\ex}(t)|).
\]
Applying (\ref{eq:diff-b-hat_b}) to both $\hat{b}_{j}(t)$ and $\hat{b}_{j}(t_{0})$,
dividing both sides by $|b_{j}^{\ex}(t)|$, and using $|b_{j}^{\ex}(t)|\geq|b_{j}^{\ex}(t_{0})|$,
we get 
\[
|\dot{\nu}_{j}(t)|\leq|\dot{\nu}_{j}(t_{0})|+\calO(|t|^{-\eps_{j}}).
\]
At initial time $t_{0}$, (\ref{eq:4.13}), (\ref{eq:initi-cond}),
and (\ref{eq:def-calV}) imply $|\dot{\nu}_{j}(t_{0})|\aleq|\nu_{j}(t_{0})|\aleq|t_{0}|^{-\eps_{j}}$,
so we conclude 
\[
|\dot{\nu}_{j}(t)|\aleq|t_{0}|^{-\eps_{j}}+|t|^{-\eps_{j}}\aleq|t|^{-\eps_{j}}.
\]
This in particular gives (\ref{eq:ode-stable}). 

\uline{Step 3: Proof of \mbox{(\ref{eq:ode-shoot})}}. Introduce
the refined variable (cf.~(\ref{eq:def-nu-and-nu-dot})) 
\begin{equation}
\dot{\hat{\nu}}_{j}\coloneqq\frac{\hat{b}_{j}-b_{j}^{\ex}}{b_{j}^{\ex}}=\dot{\nu}_{j}+o(|t|^{-\eps_{j}}),\label{eq:diff-hat_nu_dot-nu_dot}
\end{equation}
where the last equality is due to (\ref{eq:diff-b-hat_b}). Linearizing
(\ref{eq:ode-ode}) around $(\lmb_{j}^{\ex},b_{j}^{\ex})$, we get
\[
\left|\begin{aligned}(\lda_{j}-\lda_{j}^{\ex})_{t} & =-(\hat{b}_{j}-b_{j}^{\ex})+o(|t|^{-\eps_{j}})\cdot(-b_{j}^{\ex}),\\
-(\hat{b}_{j}-b_{j}^{\ex})_{t} & =((k-1)\nu_{j}+o(|t|^{-\eps_{j}}))\cdot(-b_{j}^{\ex})_{t}.
\end{aligned}
\right.
\]
Using $(\lmb_{j}^{\ex})_{t}=-b_{j}^{\ex}=|t|^{-1}\alp_{j}\cdot\lmb_{j}$
and $(b_{j}^{\ex})_{t}=|t|^{-1}(\alp_{j}+1)\cdot b_{j}^{\ex}$, we
obtain 
\begin{equation}
\left|\begin{aligned}\nu_{j,t} & =\frac{\alp_{j}}{|t|}\Big\{-\nu_{j}+\dot{\hat{\nu}}_{j}+o(|t|^{-\eps_{j}})\Big\},\\
\dot{\hat{\nu}}_{j,t} & =\frac{\alp_{j}+1}{|t|}\Big\{(k-1)\nu_{j}-\dot{\hat{\nu}}_{j}+o(|t|^{-\eps_{j}})\Big\}.
\end{aligned}
\right.\label{eq:ODE-nu-hat}
\end{equation}

The matrix $A_{j}$ introduced in (\ref{eq:def-A_j}) is associated
with the ODE system (\ref{eq:ODE-nu-hat}). Recall the diagonalization
(\ref{eq:A_j-diagonalization}) of $A_{j}$. Define the coordinate
vector of $\hat{\bm{\nu}}_{j}\coloneqq(\nu_{j},\dot{\hat{\nu}}_{j})$
projected onto the unstable and stable mode by 
\begin{equation}
\begin{pmatrix}P_{u}\hat{\bm{\nu}}_{j}\\
P_{s}\hat{\bm{\nu}}_{j}
\end{pmatrix}\define P^{-1}\hat{\bm{\nu}}_{j}\quad\text{so that}\quad\begin{pmatrix}P_{u}\hat{\bm{\nu}}_{j}\\
P_{s}\hat{\bm{\nu}}_{j}
\end{pmatrix}_{t}=\frac{1}{|t|}\Big\{\begin{pmatrix}\sigma_{j,+}P_{u}\hat{\bm{\nu}}_{j}\\
\sigma_{j,-}P_{s}\hat{\bm{\nu}}_{j}
\end{pmatrix}+o(|t|^{-\eps_{j}})\Big\}.\label{eq:5.52}
\end{equation}

Now, we claim that the stable component $P_{s}\hat{\bm{\nu}}_{j}$
is small: 
\begin{equation}
P_{s}\hat{\bm{\nu}}_{j}(t)=o(|t|^{-\eps_{j}}).\label{eq:5.53}
\end{equation}
Indeed, we use $\sigma_{j,-}<0$ and integrate (\ref{eq:5.52}) to
obtain $|P_{s}\hat{\bm{\nu}}_{j}(t)|\leq|P_{s}\hat{\bm{\nu}}_{j}(t_{0})|+o(|t|^{-\eps_{j}})$.
Using (\ref{eq:diff-hat_nu_dot-nu_dot}), we get $|P_{s}\hat{\bm{\nu}}_{j}(t_{0})|\leq|P_{s}\bm{\nu}_{j}(t_{0})|+o(|t_{0}|^{-\eps_{j}})$,
where $\bm{\nu}_{j}=(\nu_{j},\dot{\nu}_{j})$. Since $P_{s}\bm{\nu}_{j}(t_{0})=0$
by (\ref{eq:4.13}) and (\ref{eq:initi-cond}), we get (\ref{eq:5.53}).

Using $\nu_{j}=P_{u}\hat{\bm{\nu}}_{j}+P_{s}\hat{\bm{\nu}}_{j}$,
(\ref{eq:5.52}), and (\ref{eq:5.53}), we have 
\begin{align*}
\nu_{j,t} & =|t|^{-1}\{\sigma_{j,+}P_{u}\hat{\bm{\nu}}_{j}+\sigma_{j,-}P_{s}\hat{\bm{\nu}}_{j}+o(|t|^{-\eps_{j}})\}=|t|^{-1}\{\sigma_{j,+}\nu_{j}+o(|t|^{-\eps_{j}})\}.
\end{align*}
This gives 
\begin{align*}
(|t|^{\eps_{j}}\nu_{j})_{t} & =|t|^{\eps_{j}-1}\{(\sigma_{j,+}-\eps_{j})\nu_{j}+o(|t|^{-\eps_{j}})\}.
\end{align*}
In particular, if $|\nu_{j}(T)|=|T|^{-\eps_{j}}$, then (using $\eps_{j}<\min(\eps_{j-1},\sigma_{j,+})$)
\[
\frac{d}{dt}(|t|^{\eps_{j}}|\nu_{j}(t)|)\Big|_{t=T}>0.
\]
This completes the proof of (\ref{eq:ode-shoot}).
\end{proof}

\section{\label{sec:Proof-of-main-theorem}Proof of main theorem}

In this section, we prove Theorem~\ref{thm:main}. Recall that it
suffices to prove Theorem~\ref{thm:main-negative-time} by time-reversal
symmetry. The proof is done by a standard limiting argument using
the solutions obtained in Proposition~\ref{prop:solutions_u^t_0}.
\begin{proof}[Proof of Theorem~\ref{thm:main-negative-time}]
Apply Proposition~\ref{prop:solutions_u^t_0} for a decreasing time
sequence $t_{0}^{(n)}\searrow-\infty$ in $(-\infty,T_{boot}]$ to
obtain a sequence of solutions $\bm{u}_{n}\define\bm{u}^{t_{0}^{(n)}}$.
Denoting the parameters for $\bm{u}_{n}$ by $\vec{\lmb}^{(n)},\vec{b}^{(n)},\bm{g}_{n}$,
note that for $t\in[t_{0}^{(n)},T_{boot}]$ 
\begin{align*}
\|\bm{u}_{n}(t)-\bm{\calQ}(\vec{\iota},\vec{\lmb}^{\ex}(t))\|_{\dot{\calH}^{1}} & \leq\|\bm{u}_{n}-\bm{\calQ}(\vec{\iota},\vec{\lmb}^{(n)})\|_{\dot{\calH}^{1}}+\|\bm{Q}(\vec{\iota},\vec{\lmb}^{(n)})-\bm{\calQ}(\vec{\iota},\vec{\lmb}^{\ex})\|_{\dot{\calH}^{1}}\\
 & \aleq\|\bm{g}_{n}\|_{\dot{\calH}^{1}}+\tsum j{}|b_{j}^{(n)}|+\tsum j{}|\nu_{j}^{(n)}|.
\end{align*}
Applying (\ref{eq:boot-assump}) gives (for some constant $C>0$ independent
of $n$) 
\begin{equation}
\|\bm{u}_{n}(t)-\bm{\calQ}(\vec{\iota},\vec{\lmb}^{\ex}(t))\|_{\dot{\calH}^{1}}\leq C|t|^{-\eps_{J}}.\label{eq:proof-tmp}
\end{equation}

Fix $T_{0}\in(-\infty,T_{boot}]$ such that $|T_{0}|^{-\eps_{J}}\leq\frac{1}{2C}\eta$,
where $\eta>0$ is the constant in Lemma~\ref{lem:App-weak-lim}.
This ensures $\|\bm{u}_{n}(t)-\bm{\calQ}(\vec{\iota},\vec{\lmb}^{\ex}(t))\|_{\dot{\calH}^{1}}\leq\frac{1}{2}\eta$
for all $t\in[t_{0}^{(n)},T_{0}]$ by (\ref{eq:proof-tmp}).\textcolor{red}{{}
}Now, as $\bm{u}_{n}(T_{0})-\bm{\calQ}(\vec{\iota},\vec{\lmb}^{\ex}(T_{0}))$
is a bounded sequence in $\calE$ by (\ref{eq:proof-tmp}), $\bm{u}_{n}$
has a subsequence (which we still denote by $\bm{u}_{n}$) such that
$\bm{u}_{n}(T_{0})\weakto\bm{u}_{0}$ for some $\bm{u}_{0}\in\bm{\calQ}(\vec{\iota},\vec{\lmb}^{\ex}(T_{0}))+\calE=\calE_{0,J\mathrm{mod}2}$.
Let $\bm{u}(t)$ be the solution of (\ref{eq:k-equiv-WM}) with initial
data $\bm{u}(T_{0})=\bm{u}_{0}$. For any $T\in(-\infty,T_{0}]$ and
for all large $n$, we have $\sup_{t\in[T,T_{0}]}{\rm dist}(\bm{u}_{n}(t),K^{T})\leq\frac{1}{2}\eta$,
where $K^{T}\define\{\bm{\calQ}(\vec{\iota},\vec{\lmb}^{\ex}(t))\colon t\in[T,T_{0}]\}$.
Applying Lemma~\ref{lem:App-weak-lim} on $[T,T_{0}]$, $\bm{u}_{n}(t)\weakto\bm{u}(t)$
for all $t\in[T,T_{0}]$. As $T$ was arbitrary, $\bm{u}_{n}(t)\weakto\bm{u}(t)$
for all $t\in(-\infty,T_{0}]$. By the lower semi-continuity under
weak limit, (\ref{eq:proof-tmp}) also holds for $\bm{u}(t)$ for
all $t\in(-\infty,T_{0}]$. This completes the proof of Theorem~\ref{thm:main-negative-time}.
\end{proof}

\appendix

\section{\label{sec:computation-for-bubble}Proof of multi-bubble interaction
estimates}
\begin{proof}[\uline{Proof of \mbox{(\ref{eq:bubble-int})} and \mbox{(\ref{eq:bubble-dual-Mor})}}]
\uline{}Noting that $\Lda Q\aleq\chf_{y\leq1}y^{k}+\chf_{y\geq1}y^{-k}$,
we obtain 
\begin{equation}
\Lda Q_{\lda_{in}}\Lda Q_{\lda_{out}}\aleq\chf_{r\leq\lmb_{in}}\frac{r^{2k}}{\lmb_{in}^{k}\lmb_{out}^{k}}+\chf_{\lmb_{in}\leq r\leq\lmb_{out}}\frac{\lmb_{in}^{k}}{\lmb_{out}^{k}}+\chf_{r\geq\lmb_{out}}\frac{\lmb_{in}^{k}\lmb_{out}^{k}}{r^{2k}}.\label{eq:A.1}
\end{equation}
Using this and $k>1$, we obtain (\ref{eq:bubble-int}) and (\ref{eq:bubble-dual-Mor}):
\[
\|\Lda Q_{\ubr{\lda_{in}}}\Lda Q_{\ubr{\lda_{out}}}\|_{L^{1}}\aleq\frac{1}{\lmb_{in}\lmb_{out}}\cdot\frac{\lmb_{in}^{k}\lmb_{out}^{k}}{\lmb_{out}^{2k-2}}=\Big(\frac{\lda_{in}}{\lda_{out}}\Big)^{k-1}
\]
and
\begin{align*}
\|r^{-3}\sqrt{(y_{in}+1)(r+\lda_{in})}\Lda Q_{\lda_{in}}\Lda Q_{\lda_{out}}\|_{L^{2}}\\
=\frac{1}{\lmb_{in}^{1/2}}\Big\|\frac{\lmb_{in}+r}{r^{3}}\Lda Q_{\lda_{in}}\Lda Q_{\lda_{out}}\Big\|_{L^{2}} & \aleq\frac{1}{\lmb_{in}^{1/2}}\cdot\frac{\lmb_{in}^{k}}{\lmb_{in}\lmb_{out}^{k}}=\frac{1}{\lmb_{in}^{3/2}}\Big(\frac{\lmb_{in}}{\lmb_{out}}\Big)^{k}.\qedhere
\end{align*}
\end{proof}
\begin{proof}[\uline{Proof of \mbox{(\ref{eq:inner-prod-interaction})}}]
Let $i\in\{1,\dots,J\}$. Following the proof of \cite[Lemma 2.23]{JendrejLawrie2025JAMS}
(the ordering $\lmb_{1},\dots,\lmb_{J}$ there is the reverse compared
to our case), we write 
\begin{equation}
\sin(2\calQ)-\tsum j{}\sin2Q_{;j}=-\sin2Q_{;i-1}\sin^{2}Q_{;i}+\Phi_{i}(\vec{\iota},\vec{\lda}),\label{eq:A.2}
\end{equation}
where $\Phi_{i}(\vec{\iota},\vec{\lda})$ satisfies the pointwise
estimate 
\begin{equation}
|\Phi_{i}(\vec{\iota},\vec{\lda})|\aleq(\Lmb Q_{\lda_{i}})^{2}\Lmb Q_{\lmb_{i+1}}+\Lmb Q_{\lda_{i}}\sum_{j:j\neq i}(\Lmb Q_{\lda_{j}})^{2}.\label{eq:A.3}
\end{equation}
Recalling the definition of $f_{\mathbf{i}}(\vec{\iota},\vec{\lda})$
and applying (\ref{eq:A.2}), we have 
\begin{align*}
\lan\Lmb Q_{;i},f_{\mathbf{i}}(\vec{\iota},\vec{\lda})\ran & =\lan\Lmb Q_{;i},\frac{k^{2}}{r^{2}}\{\sin2Q_{;i-1}\sin^{2}Q_{;i}-\Phi_{i}(\vec{\iota},\vec{\lda})\}\ran.
\end{align*}
For the first term, we use $\sin2Q=-4y^{-k}+\calO(\chf_{y\leq1}y^{-k}+\chf_{y\geq1}y^{-3k})$
to have 
\[
\sin2Q_{;i}=-4\iota_{i}\frac{\lmb_{i}^{k}}{r^{k}}+\calO\Big(\chf_{r\leq\lmb_{i}}\frac{\lmb_{i}^{k}}{r^{k}}+\chf_{r\geq\lmb_{i}}\frac{\lmb_{i}^{3k}}{r^{3k}}\Big).
\]
Together with the residue computation $\int_{0}^{\infty}(\Lmb Q(y))^{3}4y^{-k}\frac{dy}{y}=8k^{2}$,
this gives 
\begin{align*}
\int r^{-2}\sin2Q_{;i}\sin^{3}Q_{;i-1} & =-\iota_{i-1}\iota_{i}8k^{2}\Big(\frac{\lda_{i-1}}{\lda_{i}}\Big)^{k}+\calO\Big((1+\log(\frac{\lda_{i-1}}{\lda_{i}}))\Big(\frac{\lda_{i-1}}{\lda_{i}}\Big)^{3k}\Big).
\end{align*}
We estimate the contribution from $\Phi_{i}(\vec{\iota},\vec{\lmb})$.
For the first term of RHS(\ref{eq:A.3}), we use (\ref{eq:A.1}) to
have 
\[
\|r^{-2}(\Lmb Q_{\lda_{i}})^{3}\Lmb Q_{\lmb_{i+1}}\|_{L^{1}}\aleq\Big(\frac{\lmb_{i}}{\lmb_{i+1}}\Big)^{k}.
\]
For the second term of RHS(\ref{eq:A.3}), we use (\ref{eq:A.1})
to have 
\[
\sum_{j:j\neq i}\|r^{-2}(\Lmb Q_{\lda_{i}})^{2}(\Lmb Q_{\lda_{j}})^{2}\|_{L^{1}}\aleq(1+\log(\frac{\lmb_{i-1}}{\lmb_{i}}))\Big(\frac{\lmb_{i}}{\lmb_{i-1}}\Big)^{2k}+(1+\log(\frac{\lmb_{i}}{\lmb_{i+1}}))\Big(\frac{\lmb_{i+1}}{\lmb_{i}}\Big)^{2k}.
\]
This completes the proof of (\ref{eq:inner-prod-interaction}).
\end{proof}
\begin{proof}[\uline{Proof of \mbox{(\ref{eq:pointwise-G_i})}}]
As $H_{\calQ}-H_{\lda_{i}}=\frac{k^{2}}{r^{2}}(\cos2\calQ-\cos2Q_{;i})$
and $\Lda Q=k\sin Q$, it suffices to prove 
\begin{equation}
|\cos2\calQ-\cos2Q_{;i}|_{-\ell}\aleq_{\ell}r^{-\ell}\sum_{j:j\neq i}\sin Q_{\lmb_{j}}.\label{eq:B.2}
\end{equation}
Observe 
\begin{align*}
 & \cos2\calQ-\cos2Q_{;i}\\
 & =\{\cos(2\tsum{j:j\neq i}{}Q_{;j})-1\}\cos2Q_{;i}-\sin(2\tsum{j:j\neq i}{}Q_{;j})\sin2Q_{;i}\\
 & =-2\sin^{2}(\tsum{j:j\neq i}{}Q_{;j})\cos2Q_{;i}-4\sin(\tsum{j:j\neq i}{}Q_{;j})\sin Q_{;i}\cos(\tsum{j:j\neq i}{}Q_{;j})\cos Q_{;i}.
\end{align*}
Together with a rough bound $|\cos Q_{;i}|_{-\ell}+|\sin Q_{;i}|_{-\ell}\aleq_{\ell}r^{-\ell}$,
applying $|\sin(\tsum{j:j\neq i}{}Q_{;j})|_{-\ell}\aleq_{\ell}\sum_{j:j\neq i}|\sin Q_{;j}|_{-\ell}\aleq_{\ell}r^{-\ell}\sin Q_{\lmb_{j}}$
to the above gives (\ref{eq:B.2}).
\end{proof}
\begin{proof}[\uline{Proof of \mbox{(\ref{eq:pointwise-interaction})}}]
\uline{}As $f_{\mathbf{i}}(\vec{\iota},\vec{\lda})=-\frac{k^{2}}{2r^{2}}\{\sin2\calQ-\sum_{j}\sin2Q_{;j}\}$
and $\Lmb Q=k\sin Q$, it suffices to prove 
\begin{equation}
|\sin(2\calQ)-\tsum{j=1}J\sin2Q_{;j}|_{-\ell}\aleq_{\ell}r^{-\ell}\sum_{i,j:1\leq j<i\leq J}\sin Q_{\lda_{j}}\sin Q_{\lda_{i}}.\label{eq:B.3}
\end{equation}
We proceed by induction on $J$ to prove this. If $J=1$, there is
nothing to prove. Let $J\geq2$ and assume the $J-1$ case. We begin
with 
\begin{align*}
\sin2\calQ-\tsum{j=1}J\sin2Q_{;j} & =\{\sin(2\tsum{j<J}{}Q_{;j})-\tsum{j<J}{}\sin2Q_{;j}\}\cos2Q_{;J}\\
 & \quad+\{\cos(2\tsum{j<J}{}Q_{;j})-1\}\sin2Q_{;J}+\{\cos2Q_{;J}-1\}\tsum{j<J}{}\sin2Q_{;j}.
\end{align*}
For the first term, apply the induction hypothesis and a rough bound
$|\cos Q_{;j}|_{-\ell}+|\sin Q_{;j}|_{-\ell}\aleq_{\ell}r^{-\ell}$.
For the remaining terms, write $\cos(2\tsum{j<J}{}Q_{;j})-1=-2\sin^{2}(\tsum{j<J}{}Q_{;j})$
and $\cos2Q_{;J}-1=-2\sin^{2}Q_{;J}$. Applying $|\sin Q_{;j}|_{-\ell}\aleq_{\ell}r^{-\ell}\sin Q_{\lmb_{j}}$
for each $j$, we conclude (\ref{eq:B.3}).
\end{proof}
\begin{proof}[\uline{Proof of \mbox{(\ref{eq:pre-H_Q-interaction-0})} and \mbox{(\ref{eq:pre-H_Q-interaction-1})}}]
Let $m\in\{0,1\}$. By $H_{\calQ}\Lda Q_{;i}=(H_{Q}-H_{\lda_{i}})\Lda Q_{;i}$
and (\ref{eq:pointwise-G_i}), we obtain 
\[
\|r^{m}H_{\calQ}\Lda Q_{;i}\|_{L^{2}}\aleq\sum_{j:j\neq i}\|r^{m-2}\Lda Q_{\lda_{j}}\Lda Q_{\lda_{i}}\|_{L^{2}}.
\]
We use (\ref{eq:A.1}) to have 
\begin{align*}
\sum_{j:j>i}\|r^{m-2}\Lda Q_{\lda_{j}}\Lda Q_{\lda_{i}}\|_{L^{2}} & \aleq\sum_{j:j>i}\Big\{(\chf_{m=0}\frac{1}{\lmb_{j}}+\chf_{m=1}(1+\log(\frac{\lmb_{i}}{\lmb_{j}}))^{1/2}\Big(\frac{\lmb_{j}}{\lmb_{i}}\Big)^{k}\Big\}\\
 & \aleq\chf_{m=0}\frac{1}{\lmb_{i+1}}+\chf_{m=1}(1+\log(\frac{\lmb_{i}}{\lmb_{i+1}}))^{1/2}\Big(\frac{\lmb_{i+1}}{\lmb_{i}}\Big)^{k}
\end{align*}
and
\begin{align*}
\sum_{j:j<i}\|r^{m-2}\Lda Q_{\lda_{j}}\Lda Q_{\lda_{i}}\|_{L^{2}} & \aleq\sum_{j:j<i}\Big\{\chf_{m=0}\frac{1}{\lmb_{i}}+\chf_{m=1}(1+\log(\frac{\lmb_{j}}{\lmb_{i}}))^{1/2}\Big(\frac{\lmb_{i}}{\lmb_{j}}\Big)^{k}\Big\}\\
 & \aleq\chf_{m=0}\frac{1}{\lmb_{i}}+\chf_{m=1}(1+\log(\frac{\lmb_{i-1}}{\lmb_{i}}))^{1/2}\Big(\frac{\lmb_{i}}{\lmb_{i-1}}\Big)^{k}.
\end{align*}
This completes the proof of (\ref{eq:pre-H_Q-interaction-0}) and
(\ref{eq:pre-H_Q-interaction-1}).
\end{proof}
\begin{proof}[\uline{Proof of \mbox{(\ref{eq:H1-est-interatction})}}]
By (\ref{eq:pointwise-interaction}), (\ref{eq:A.1}), and $k>1$,
we get 
\[
\||f_{\mathbf{i}}(\vec{\iota},\vec{\lmb})|_{-1}\|_{L^{2}}\aleq\sum_{i,j:i>j}\|r^{-3}\Lda Q_{\lda_{j}}\Lda Q_{\lda_{i}}\|_{L^{2}}\aleq\sum_{i,j:i>j}\frac{1}{\lmb_{i}^{2}}\Big(\frac{\lmb_{i}}{\lmb_{j}}\Big)^{k}\aeq\max_{j\in\{2,\dots,J\}}\frac{1}{\lmb_{j}^{2}}\Big(\frac{\lmb_{j}}{\lmb_{j-1}}\Big)^{k}.\qedhere
\]
\end{proof}

\section{Some linear and nonlinear estimates}
\begin{lem}
\label{lem:coer-A}Let $k\geq2$. For $g$ such that the first term
of each right hand side is finite, 
\begin{align}
\int\frac{|Ag|^{2}}{(1+y)y^{2}} & \geq c\int\frac{|\rd_{y}g|^{2}+|y^{-1}g|^{2}}{(1+y)y^{2}}-C\lan\calZ,g\ran^{2},\label{eq:coer-A-1}\\
\int\frac{|Ag|^{2}}{(1+y)^{2}} & \geq c\int\frac{|\rd_{y}g|^{2}+|y^{-1}g|^{2}}{(1+y)^{2}}-C\lan\calZ,g\ran^{2},\label{eq:coer-A-2}
\end{align}
for some constants $c,C>0$.
\end{lem}

\begin{proof}[Proof sketch]
The estimate (\ref{eq:coer-A-1}) (resp., (\ref{eq:coer-A-2})) is
a standard consequence of the following facts: (i)~as $y\to0$, $A\approx-\rd_{y}+\frac{k}{y}$
with $k>1$ (resp., $k>0$), (ii)~as $y\to\infty$, $A\approx-\rd_{y}-\frac{k}{y}$
with $k>-\frac{3}{2}$ (resp., $k>-1$), (iii)~$\ker A=\mathrm{span}\{\Lmb Q\}$,
and (iv)~$\lan\calZ,\Lmb Q\ran=1$ with $(1+y)^{1/2}y^{2}\calZ\in L^{2}$
(resp., $(1+y)y\calZ\in L^{2}$). 

Let us only briefly sketch the argument for (\ref{eq:coer-A-1}) and
rather refer the reader to the proof of \cite[Lemma B.2]{Kim2025JEMS}
(replace $m$ by $k-1$) and \cite[Lemma B.2]{RaphaelRodnianski2012Publ.Math.}
for closely related proofs. 

By the properties (i) and (ii) (proceed similarly as in \cite[Lemma B.2]{Kim2025JEMS}),
one can obtain a subcoercivity estimate 
\[
\int\frac{|Ag|^{2}}{(1+y)y^{2}}+\int\chf_{y\aeq1}|g|^{2}\aeq\int\frac{|\rd_{y}g|^{2}+|y^{-1}g|^{2}}{(1+y)y^{2}}.
\]
This together with (iii) and (iv), and proof by contradiction give
a coercivity estimate
\[
\int\frac{|Ag|^{2}}{(1+y)y^{2}}\aeq\int\frac{|\rd_{y}g|^{2}+|y^{-1}g|^{2}}{(1+y)y^{2}}\qquad\text{if }\lan\calZ,g\ran=0.
\]
For general $g$ without the orthogonality condition $\lan\calZ,g\ran=0$,
decompose $g=c\Lmb Q+\td g$ with $c=\lan\calZ,g\ran$ so that $\lan\calZ,\td g\ran=0$,
apply the previous coercivity estimate to $\td g$, and replace $\td g$
by $g$ with errors involving $c=\calO(|\lan\calZ,g\ran|)$.
\end{proof}
\begin{lem}
We have 
\begin{equation}
\||\NL_{\calQ}(g)|_{-1}\|_{L^{2}}\aleq\|g\|_{\dot{H}_{k}^{2}}^{2}.\label{eq:pre-NL}
\end{equation}
\end{lem}

\begin{proof}
We use the well-known weighted $L^{\infty}$-estimate $\|r^{-1}g\|_{L^{\infty}}\aleq\|g\|_{\dot{H}_{k}^{2}}$
adapted to $k$-equivariant ($k\geq2$) functions. With this estimate,
it suffices to prove the following (rough) pointwise estimate 
\begin{equation}
|\NL_{\calQ}(g)|_{-1}\aleq r^{-2}|g|_{-1}|g|\label{eq:app-1}
\end{equation}
because (\ref{eq:pre-NL}) then follows from 
\[
\||\NL_{\calQ}(g)|_{-1}\|_{L^{2}}\aleq\|r^{-1}|g|_{-1}\|_{L^{2}}\|r^{-1}g\|_{L^{\infty}}\aleq\|g\|_{\dot{H}_{k}^{2}}^{2}.
\]

We show (\ref{eq:app-1}). Recalling the definition of $\NL_{\calQ}(g)$
and $f(u)=\frac{\sin2u}{2}$, we compute 
\begin{align*}
\NL_{\calQ}(g) & =\frac{k^{2}}{2r^{2}}(\cos2\calQ\sin2g+\sin2\calQ\cos2g-\sin2\calQ-2g\cos2\calQ)\\
 & =\frac{k^{2}}{2r^{2}}((\cos2g-1)\sin2\calQ+(\sin2g-2g)\cos2\calQ).
\end{align*}
For the first term, observe 
\begin{align*}
|r^{-2}(\cos2g-1)\sin2\calQ|_{-1} & \aleq r^{-2}\{|\cos2g-1||\sin2\calQ|_{-1}+|\rd_{r}g||\sin2g||\sin2\calQ|\}\\
 & \aleq r^{-2}\{|g|^{2}r^{-1}+|\rd_{r}g||g|\}\aleq r^{-2}|g|_{-1}|g|.
\end{align*}
For the second term, observe 
\begin{align*}
|r^{-2}(\sin2g-2g)\cos2\calQ|_{-1} & \aleq r^{-2}\{|\sin2g-2g||\cos2\calQ|_{-1}+|\rd_{r}g||\cos2g-1|\}\\
 & \aleq r^{-2}\{|g|^{2}r^{-1}+|\rnd_{r}g||g|\}\aleq r^{-2}|g|_{-1}|g|.
\end{align*}
This completes the proof of (\ref{eq:app-1}) and hence (\ref{eq:pre-NL}).
\end{proof}

\section{Blow-up criterion and weak limit of solutions}
\begin{lem}[{\cite[Corollary A.4]{Jendrej2019AJM}}]
\label{lem:App-escape-cpt}Let $\ell,m\in\bbZ$. There exists a constant
$\eta>0$ such that the following holds. Let $\bm{u}:[t_{0},T_{+})\to\calE_{\ell,m}$
be a maximal solution to (\ref{eq:k-equiv-WM}) with $T_{+}<+\infty$.
Then, for any compact set $K\subset\calE_{\ell,m}$ all of whose elements
have the form $\bm{\calQ}(\vec{\iota},\vec{\lda})$, there exists
$\tau<T_{+}$ such that ${\rm dist}(\bm{u}(t),K)\define\inf_{\bm{v}\in K}\|\bm{u}(t)-\bm{v}\|_{\dot{\calH}^{1}}>\eta$
for $t\in[\tau,T_{+})$.
\end{lem}

\begin{proof}
This is the analogue of \cite[Corollary A.4]{Jendrej2019AJM} (the
case of $\ell=m=0$). One uses the nonlinear decomposition of \cite[Lemma 2.13]{JendrejLawrie2025JAMS}
instead to incorporate $\ell,m\in\bbZ$ to get the conclusion.
\end{proof}
\begin{lem}[{\cite[Corollary A.6]{Jendrej2019AJM}}]
\label{lem:App-weak-lim}Let $\ell,m\in\bbZ$. There exists a constant
$\eta>0$ such that the following holds. Let $K\subset\calE_{\ell,m}$
be a compact set, whose elements have the form of $\bm{\calQ}(\vec{\iota},\vec{\lda})$,
and let $\bm{u}_{n}:[T_{1},T_{2}]\to\calE_{\ell,m}$ be a sequence
of solutions to (\ref{eq:k-equiv-WM}) such that 
\[
{\rm dist}(\bm{u}_{n}(t),K)=\inf_{\bm{v}\in K}\|\bm{u}_{n}(t)-\bm{v}\|_{\dot{\calH}^{1}}\leq\eta,\quad\text{for all }n\in\bbN\text{ and }t\in[T_{1},T_{2}].
\]
Suppose that $\bm{u}_{n}(T_{2})\weakto\bm{u}_{0}\in\calE_{\ell,m}$.
Then the solution $\bm{u}(t)$ to (\ref{eq:k-equiv-WM}) with initial
data $\bm{u}(T_{2})=\bm{u}_{0}$ is defined for $t\in[T_{1},T_{2}]$
and 
\[
\bm{u}_{n}(t)\weakto\bm{u}(t),\qquad\text{for all }t\in[T_{1},T_{2}].
\]
\end{lem}

\begin{proof}
This is the analogue of \cite[Corollary A.6]{Jendrej2019AJM} (the
case of $\ell=m=0$). As in the previous lemma, one uses the nonlinear
decomposition of \cite[Lemma 2.13]{JendrejLawrie2025JAMS} instead
to get the conclusion.
\end{proof}
\bibliographystyle{abbrv}
\bibliography{References}

\end{document}